%% file: GD-manifold-arxiv-v1.tex
\documentclass[a4paper,11pt]{article}

\input{packages.tex}
\input{commands.tex}
\numberwithin{equation}{section}

\crefname{equation}{}{}
\crefrangeformat{equation}{(#3#1#4--#5#2#6)}
\renewcommand{\tilde}{\widetilde}

\graphicspath{{images}}

\bibliography{references.bib, IP.bib}

\def\Lp#1{\mathrm{L}^{#1}}
\def\Wkp#1#2{\mathrm{W}^{#1,#2}}

\def\Ck#1{\mathrm{C}^{#1}}

\def\diag{\mathrm{diag}}

\def\Span{\mathrm{span}}

\title{Gradient Descent on Point Clouds and Applications in Learned Operator Correction}
\author[1]{Andreas Hauptmann}
\author[2]{Yury Korolev}
\author[3]{Matthew Thorpe}
\date{August 2026}
\affil[1]{Research Unit of Mathematical Sciences,\protect\\ University of Oulu,\protect\\ Pentti Kaiteran katu 1, Linnanmaa, Finland. \vspace{\baselineskip}}
\affil[2]{Department of Mathematical Sciences,\protect\\ University of Bath,\protect\\ Bath, BA2 7AY, UK. \vspace{\baselineskip}}
\affil[3]{Department of Statistics,\protect\\ University of Warwick,\protect\\ Coventry, CV4 7AL, UK.}

\begin{document}

\maketitle

\begin{abstract}
We consider the problem of minimising an energy over an unknown manifold that is  given implicitly by a point cloud.
For a known manifold one can define a gradient descent scheme analogously to the classical construction in Euclidean spaces.
However, when the manifold is not known one has to simultaneously estimate the manifold whilst minimising the energy.
We define a gradient descent scheme which remains in a neighbourhood of the manifold and, under suitable stability and sampling assumptions, converges to a neighbourhood of a local minimiser whose size vanishes as the time step and sampling errors vanish.
As an example we show the application of the methodology to learning operator corrections in inverse problems.
\end{abstract}

\keywords{optimisation on manifolds, inverse problems, gradient flows}

\subjclass{65K10, 65M32, 68T07, 49M15}

\section{Introduction} \label{sec:Intro}

Data is typically assumed to satisfy the manifold assumption.
That is, even though in many applications (in particular imaging) the extrinsic dimension of data is large the intrinsic dimension is much smaller.
Mathematically this can be modelled by assuming that the data lives on a manifold $\cM\subset\bbR^D$.
Of course, in practice this manifold is usually unknown but given a data set one can attempt to learn the manifold, and many such methods exist in the literature, e.g.~\cite{dey07,dey04,chazal08,cheng05,boissonnat10,niyogi11}.

Many data driven problems can be posed as variational problems, a typical inverse problem is given in~\eqref{eq:var-reg-A} below, but more generally we may think of the problem of trying to minimize $\cE(x)$ with respect to $x\in\cM$.
One way to look at this problem is via the method of Lagrange multipliers, i.e. to consider the unconstrained saddle point problem: $\inf_{x\in\bbR^D} \sup_{\lambda\in\bbR} \l \cE(x) + \lambda d(x,\cM)^2 \r$ where $d(x,\cM)$ is the  distance from $x$ to $\cM$.
Our aim in this paper is to develop this methodology in the situation where the manifold $\cM$ is unknown but we have access to samples $\{x_i\}_{i=1}^n\subset\cM$.
Our method is to use gradient descent in two steps.
In the first we use the projection (onto the tangent space) of the Euclidean gradient of $\cE$, and in the second we project back onto the manifold.
Since the manifold is unknown we approximate the manifold and its tangent planes from samples $\{x_i\}_{i=1}^n$. 

\paragraph{Manifold recovery from point cloud samples.}
The simplest case of estimating a manifold is perhaps the problem of estimating a curve.
Since the 1960's various spline models have been used for estimating curves from pointwise observations and~\cite{wahba90} provides a good overview of the early development of this field.
(A related problem is fitting a curve on a manifold which can be thought of as a spline problem on the manifold~\cite{samir12}.)
Via~\cite{brown96} the $n\to\infty$ spline smoothing consistency problem is equivalent to a small noise limit for a stochastic ODE, e.g.~\cite{agapiou13,bissantz04,goldenshluger00}.

The minimax rate of convergence (in Hausdorff distance) of estimating a suitably nice manifold from noisy observations is $n^{-\frac{2}{d+2}}$ (where $d$ is the dimension of the manifold)~\cite{genovese2012minimax}.
Within the same paper an estimator was constructed that, up to logarithms, achieved the minimax rate but was infeasible to implement in practice.
Instead a simpler algorithm was proposed that achieved a Hausdorff error rate of $\l\frac{\log n}{n}\r^{\frac{1}{D}}$ where $D$ is the dimension of the ambient space.

None of the above methods explicitly try to recover the tangent plane of the manifold (which we need to project back to the manifold).
One model for manifolds that includes tangential information is the varifolds framework (varifolds are of course more general).
Approximating continuum varifolds by discrete varifolds has attracted interest, for example~\cite{buet17,buet17a,buet15}, and higher order properties of manifolds such as curvature can be recovered.
However, whilst the varifold construction provides a framework for studying convergence it does not define a method for approximating tangent planes (the tangential directions are assumed to be either given or are constructed from the point cloud).
The typical method for estimating tangent planes is local-PCA (which we use in our method) and a thorough convergence analysis has been performed in~\cite{little17}.

\paragraph{Gradient descent on manifolds.}
If the manifold is known then our method is similar to the projected gradient descent scheme in~\cite{hauswirth16} where the authors consider a constrained variational problem over a manifold.
As in the Euclidean case, faster convergence can be achieved using second order methods such as Newton's method, e.g.~\cite{smith14}.
And when gradients are corrupted by noise (or are approximated to reduce computational burden) stochastic gradient descent enjoys the same convergence properties as in the Euclidean case~\cite{bonnabel13}.
When the manifold has more structure then specialised algorithms can be developed, and one area that has attracted particular interest is optimisation on matrix manifolds (which includes the problem of finding eigenvalues and eigenvectors), e.g.~\cite{absil08}.

We also mention the recent paper~\cite{diepeveen2025iso} that considers optimisation on learned manifolds in the so-called iso-Riemannian framework, which redefines Riemannian geometry and relies on an appropriately redefined notion of convexity. Our approach is simpler and avoids constructing the manifold geometry, instead relying on the ambient Euclidean geometry and learned projections.

\paragraph{Inverse problems and modelling errors. }
As our main application we consider learned operator correction in inverse problems. We consider an operator equation
\begin{equation}\label{eq:Ax=y}
    A x = y,
\end{equation}
where $x \in X$ is the quantity of interest, typically an image, that we want to reconstruct from indirect measurements $y \in Y$ using the forward model (or forward operator) $A \colon X \to Y$, where $X$ and $Y$ are vector spaces and the operator $A$ is linear. In this paper, we restrict ourselves to finite-dimensional spaces $X,Y$.

Although problem~\eqref{eq:Ax=y} is posed over the whole space $X$, we will later assume that the solution lies in some manifold $\cM \subset X$ which is given implicitly via a collection of samples. 

The role of the forward model is to predict, for any $x \in X$, the corresponding measurement $y(x) \in Y$, which can then be compared to the experimentally observed $y$ from~\eqref{eq:Ax=y}. To be useful, the model $A$ has to be an accurate enough reflection of the measurement process. 

Having such a model can be difficult for several reasons. Firstly, our understanding of the physics of the measurement process can be limited. Secondly, the model can contain parameters that are estimated experimentally and therefore contain errors. Thirdly, and this is the type of problem we are focusing on, the exact model $A$ may be computationally expensive and its use can be prohibitive in time-sensitive applications. 

In this case one has to resort to a computationally less expensive approximation $\tilde A$, which is often based on simplified physics. Neglecting this misfit between the real process that generated the measured data and the simplified model $\tilde A$ used in reconstruction algorithms can dramatically decrease the reconstruction quality~\cite{kaipio_somersalo}.

Handling model errors is a long-standing topic in the inverse problems literature, e.g.~\cite{TGSYag,kennedy:2001,arridge:2006,kaipio2007statistical,calvetti:2014, LB_MB_YK_CBS:2020,riis:2021} (the list is far from being exhaustive). With the increased use of machine learning in inverse problems, methods that learn the forward model (or its correction) from training samples have been proposed~\cite{asp-kor-sch-2020,lunz21,dehoop:2023}, see also~\cite{arridge2023inverse} for a more detailed review. 

If the problem~\eqref{eq:Ax=y} is ill-posed, regularisation is required to solve it numerically. A popular method is to replace~\eqref{eq:Ax=y} with a variational problem
\begin{equation}\label{eq:var-reg-A}
    \min_{x \in X} \frac12 \| Ax - y \|_Y^2 + \alpha R(x),
\end{equation}
where $R \colon X \to \RI$ is a regularisation functional and  $\alpha>0$ is the regularisation parameter that balances the influence of the so-called data fidelity term $\| Ax - y \|_Y^2$ and the regularisation term $R(x)$. The variational problem~\eqref{eq:var-reg-A} is  typically solved by an iterative method. We refer to~\cite{Benning_Burger_modern:2018} for a review of modern variational regularisation.

If the exact model $A$ is prohibitively computationally expensive, one can replace it with the inexpensive approximation $\tilde A$ and obtain
\begin{equation}\label{eq:var-reg-tilde-A}
    \min_{x \in X} \frac12 \| \tilde Ax - y \|_Y^2 + \alpha R(x).
\end{equation}
This can work reasonably well, provided that the regularisation parameter $\lambda$ is adjusted to both measurement error (error in $y$) and the modelling error (error in $\tilde A$), e.g.~\cite{TGSYag}. However, this is a ``worst-case'' approach in the sense that it increases the effect of the regulariser to ensure stability, but doesn't actually correct the modelling error.

The Bayesian approximation error framework~\cite{kaipio2007statistical} allows one  to estimate the modelling error by drawing samples from the prior distribution and evaluating the mismatch between $A$ and $\tilde A$ on those samples. These evaluations are used to learn a conditional distribution (typically approximated by a Gaussian) of the modelling error $(A-\tilde A)x$ given $x$. This can also be repeated sequentially using estimates of the posterior~\cite{calvetti:2014}.

The Bayesian approximation error approach allows one to use the statistics of the modelling error rather than resorting to a worst-case estimate, but it still does not \emph{correct} this error and only adjusts the regularisation, albeit in a more sophisticated way than classical deterministic methods. This motivated the study of \emph{learned} operator correction methods~\cite{lunz21} which attempt to actually correct for the mismatch between the exact and approximate models.

\paragraph{Learned operator correction. }
Learned operator correction methods use a selection of samples $\{x_i\}_{i=1}^n \subset \cM$ and evaluate the expensive model $A$ to obtain $\{y_i \defeq A x_i\}_{i=1}^n$. Using the training pairs $\{(x_i,y_i)\}_{i=1}^n$,~\cite{lunz21} proposed to train a neural network $F_\Theta:Y \to Y$ such that $F_\Theta(\tilde A x) \approx Ax$ in order to correct the error in the computationally cheap approximation.
Then one can attempt to solve the following variational regularisation problem instead of~\eqref{eq:var-reg-A} 
\begin{equation}\label{eq:var-reg-Atheta}
    \min_{x \in X} \frac12 \| A_\Theta(x) - y \|_Y^2 + \alpha R(x),
\end{equation}
where the corrected learned model $A_\Theta(x) \defeq F_\Theta(\tilde A x)$ replaces the exact model $A$.

The problem with this approach is that, although it is reasonable to expect that $A_\Theta(x) \approx Ax$ near the training pairs (and, by proxy, near the manifold $\cM$ if the samples are dense enough), there is no reason to expect that the approximation remains valid away from $\cM$. Solving~\eqref{eq:var-reg-Atheta} iteratively (e.g. with gradient descent) may require evaluations of the model away from the samples, which will result in erroneous values of the gradient. In \cite{lunz21} it was proposed to solve this problem by adding points along the iteration path to the training set, which made the training prohibitively expensive and very cumbersome. 

Gradient descent on point clouds that we develop in this paper can alleviate this problem. 
It allows us to  solve~\eqref{eq:var-reg-Atheta} iteratively while staying near the training samples, as demonstrated numerically in Section~\ref{sec:correction} where we apply the method to photoacoustic tomography (PAT) as a prototypical example of an inverse problem.

\paragraph{Main contributions.}
Our main novelty lies in  developing a provably convergent gradient descent method that operates directly on unknown manifolds represented only by point clouds. Unlike classical approaches that rely on explicit geometric knowledge, the method combines local tangent estimation and projection to ensure iterates remain close to the data manifold. This framework is further shown to be particularly effective in learned operator correction for inverse problems, where it stabilises optimisation by preventing excursions into regions where learned models are inaccurate.

\paragraph{Organisation of the paper.} 
In the next section we develop our gradient descent algorithm.
Sections~\ref{sec:SynEx} and~\ref{sec:correction} contain two examples.
The first of these two sections gives a synthetic example that visualises the difference between our approach and attempting gradient descent in the ambient space.
The second gives a more involved application to learned operator correction.
We finish the paper with a summary and conclusions.

\section{Gradient Descent on Learned Manifolds} \label{sec:GD}

We suppose we have samples $\{x_i\}_{i=1}^n\subset \cM \subseteq \bbR^D$, where $\cM$ is a compact $d$-dimensional twice continuously differentiable Riemannian manifold with positive reach embedded in $\bbR^D$, and a differentiable energy $\cE:\R^D\to[0,+\infty]$, where the gradient $\nabla \cE$ is Lipschitz continuous.
In the first two subsections we assume we know the manifold $\cM$ and we have access to a parameterisation of the tangent space $\{t_j(x)\}_{j=1}^d$ for every $x$ (which we assume we can extend to a neighbourhood of $\cM$).
Our aim is, given $\xi^0$, to find a sequence $\xi^k$ such that $d(\xi^k,\cM)\to 0$ and $\cE(\xi^k)\to \min_{x\in\cM} \cE(x)$.
We start with the continuum time formulation (Section~\ref{subsec:GD:Cts}), before discussing the discrete time formulation (Section~\ref{subsec:GD:Dis}) and finally we include manifold estimation (Section~\ref{subsec:GD:Unknown}).

\subsection{Continuum time gradient descent on known manifolds} \label{subsec:GD:Cts}

We rewrite the constrained variational problem in unconstrained form:
\begin{align*}
\inf_{x\in\cM} \cE(x) =  \inf_{x\in\bbR^D} \sup_{\lambda\geq 0} \l \cE(x) + \frac{\lambda}{2} \|x-P_{\cM}(x)\|^2 \r 
\end{align*}
where $P_{\cM}(x)$ is the projection of $x$ onto $\cM$ which is well-defined if $\cM$ has a positive reach as long as $x$ is sufficiently close to $\cM$, so that $d(x,\cM) = \|x-P_{\cM}(x)\|$.
We will construct our (continuum-time) algorithm so that our minimising sequence $\xi(t)$ converges to the manifold, and in particular as long as we start sufficiently close to the manifold we will never move outside of the domain of $P_{\cM}$. 

Let $n(x) = P_{\cM}(x) - x$ be the vector that takes $x$ to $\cM$. It is well-known that under the positive reach assumption $n(x)$ is normal to $\cM$ and $\frac12\nabla  \|n(x)\|^2 = -n(x)$. (Note that we use an opposite sign from the usual convention for the normal vector.)

\begin{proposition}[{e.g.~\cite[Thm. 4.8]{federer1959curvature}}]\label{prop:GD:Cont:NormProp}
There exists $r_0$ such that for all $x$ with $d(x,\cM)\leq r_0$ we have:
\begin{enumerate}
    \item $n(x)\in \Span(\{n_i(P_{\cM}(x))\}_{i=1}^{D-d})$ where $\{n_i(P_{\cM}(x))\}_{i=1}^{D-d}$ form a basis of the orthogonal complement of the tangent space;
    \item 
    $\frac12\nabla  \|n(x)\|^2 = -n(x)$. 
    \item $n(x)$ is Lipschitz on $\{x \colon \norm{n(x)} \leq r_0\}$.
\end{enumerate}
\end{proposition}

Now, the first possible construction for a minimising sequence $\xi(t)$ would be to consider the gradient flow
\[ \dot{\xi}(t) = -\nabla \cE(\xi(t)) + \lambda n(\xi(t)). \]
We note that when $\nabla\cE(\xi(t))\cdot n(\xi(t))<0$ then the energy helps us to move closer to the manifold (since the energy increases away from the manifold).
However, the converse is also true, namely when $\nabla\cE(\xi(t))\cdot n(\xi(t))>0$ the energy term pushes us away from the manifold.
In order to better control the distance of the sequence $\xi(t)$ from the manifold we remove the normal component of the gradient of $\cE$.

More precisely, we consider the following gradient flow
\begin{equation} \label{eq:GD:Cont:GF}
\dot{\xi}(t) = - \Pi_{\cT_{\xi(t)}\cM} \nabla \cE(\xi(t)) + \lambda n(\xi(t)) ,
\end{equation}
where $\Pi_{\cT_x\cM} v = \sum_{j=1}^d v\cdot t_j(x) t_j(x)$ is the projection onto the tangent plane at $x$ (spanned by the orthonormal tangent vectors $\{t_j(x)\}_{j=1}^d$).
We can also write $\Pi_{\cT_x\cM} = \Sigma_0(x)\Sigma_0(x)^\top$ where 
\[ \Sigma_0(x) = [t_1(x), t_2(x), \dots, t_d(x)] \in\bbR^{D\times d}. \]

With this construction (and using Proposition~\ref{prop:GD:Cont:NormProp}) we have 
\begin{align}
\frac12 \frac{\dd}{\dd t} \|n(\xi(t))\|^2 & = \frac12 \nabla \|n(\cdot)\|^2\lfloor_{\xi(t)} \cdot \dot{\xi}(t) \notag \\
 & = - n(\xi(t)) \cdot \dot{\xi}(t) \notag \\
 & = \sum_{j=1}^d \nabla \cE(\xi(t)) \cdot t_j(\xi(t)) \underbrace{t_j(\xi(t)) \cdot n(\xi(t))}_{\approx 0} - \lambda \|n(\xi(t))\|^2 \notag \\
 & \approx - \lambda \|n(\xi(t))\|^2 \label{eq:GD:Cont:DecaytoManifoldApprox}
\end{align}
with the approximation exact if $n(\xi(t))$ is exactly normal to the manifold at $\xi(t)$. 
In particular, we have the approximation $\|n(\xi(t))\| \approx \|n(\xi(0))\| e^{-\lambda t}$ and so we expect $\xi(t)$ to converge exponentially onto the manifold; this is made precise in the proposition below.
Note that this is true for all $\lambda>0$ which motivates treating $\lambda$ as a constant in the sequel.

\begin{proposition}
\label{prop:GD:Cont:NormDecay}
Assume there exists $r_0>0$ and $L>0$ such that if $d(x,\cM)\leq r_0$ then $\|t_j(x) - t_j(P_{\cM}(x))\|\leq L \|x-P_{\cM}(x)\|$.
Let $M=\sup_{x\,:\, d(x,\cM)\leq r_0} \|\nabla \cE(x)\|$ and suppose that the penalisation $\lambda$ is sufficiently large so that $\alpha:=\lambda-dLM(D-d) >0$.
Then, $\|n(\xi(0))\|\leq r_0$ implies
\[ \|n(\xi(t))\| \leq \|n(\xi(0))\| e^{-\alpha t}. \]
\end{proposition}

\begin{proof}
We will make the above calculation (i.e.~\eqref{eq:GD:Cont:DecaytoManifoldApprox}) precise.
Using Proposition~\ref{prop:GD:Cont:NormProp} we have that $n(\xi(t)) \in \Span(\{ n_i(P_{\cM}(\xi(t)))\}_{i=1}^{D-d})$ (assuming $\|n(\xi(t))\|\leq r_0$) where $\{n_i(P_{\cM}(\xi(t)))\}_{i=1}^{D-d}$ form a basis in the orthogonal complement of the tangent space at $P_{\cM}(\xi(t))$.
Hence,
\begin{align*}
t_j(\xi(t)) \cdot n(\xi(t)) & = \sum_{i=1}^{D-d} t_j(\xi(t)) \cdot n_i(P_{\cM}(\xi(t))) \, n(\xi(t)) \cdot n_i(P_{\cM}(\xi(t))) \\
 & = \sum_{i=1}^{D-d} \l t_j(\xi(t)) - t_j(P_{\cM} (\xi(t))) \r \cdot n_i(P_{\cM}(\xi(t))) \, n(\xi(t)) \cdot n_i(P_{\cM}(\xi(t))) \\
 & \qquad \qquad + \sum_{i=1}^{D-d} \underbrace{t_j(P_{\cM} (\xi(t))) \cdot n_i(P_{\cM}(\xi(t)))}_{=0} \, n(\xi(t)) \cdot n_i(P_{\cM}(\xi(t))) \\
 & = \sum_{i=1}^{D-d} \l t_j(\xi(t)) - t_j(P_{\cM} (\xi(t))) \r \cdot n_i(P_{\cM}(\xi(t))) \, n(\xi(t)) \cdot n_i(P_{\cM}(\xi(t))).
\end{align*}
Using the Lipschitz assumption on the tangent vectors we have
\begin{equation} \label{eq:GD:Cont:tangentnormalbound}
\la t_j(\xi(t)) \cdot n(\xi(t)) \ra \leq L(D-d) \|\xi(t) - P_{\cM}(\xi(t))\| \|n(\xi(t))\| = L(D-d) \|n(\xi(t))\|^2.
\end{equation}
Substituting this into~\eqref{eq:GD:Cont:DecaytoManifoldApprox} and using the bound on $\nabla\cE$ (assuming $\|n(\xi(t))\|\leq r_0$) we have
\begin{align*}
\frac12 \frac{\dd}{\dd t} \|n(\xi(t))\|^2 & \leq \l L(D-d) \la \sum_{j=1}^d \nabla \cE(\xi(t)) \cdot t_j(\xi(t)) \ra - \lambda \r \|n(\xi(t))\|^2 \\
 & \leq \l dLM(D-d) - \lambda \r \|n(\xi(t))\|^2.
\end{align*}
By Gr\"onwall's inequality
\[ \|n(\xi(t))\|^2 \leq \|n(\xi(0))\|^2 e^{2(dLM(D-d) - \lambda)t}. \]
Substituting $\alpha = \lambda - dLM(D-d)$ shows that $\|n(\xi(t))\|$ is decreasing, and therefore if $\|n(\xi(0))\|\leq r_0$ then $\|n(\xi(t))\|\leq r_0$ for all $t>0$.
In particular, the above argument holds and $\|n(\xi(t))\|\leq \|n(\xi(0))\|e^{-\alpha t}$.
\end{proof}

The disadvantage of this approach is that we lose some control over the rate at which the energy is minimized (in fact we cannot in general guarantee that $\cE(\xi(t))$ is decreasing).
Since,
\[ \frac{\dd }{\dd t} \cE(\xi(t)) = \nabla \cE(\xi(t)) \cdot \dot{\xi}(t) = - \underbrace{\|\Pi_{\cT_{\xi(t)}\cM}\nabla \cE(\xi(t))\|}_{\geq 0} + \lambda \nabla \cE(\xi(t)) \cdot n(\xi(t)) \]
then when $\nabla \cE(\xi(t)) \cdot n(\xi(t))<0$ the manifold projection helps the energy minimization, but when $\nabla \cE(\xi(t)) \cdot n(\xi(t))>0$ the manifold projection moves in a direction that increases $\cE$.
Of course, if $\nabla \cE(\xi(t)) \cdot n(\xi(t))\leq 0$ then we can expect exponential and monotonic convergence of $\cE(\xi(t))$ to its minimum. 
We can, however, still show convergence to a minimiser under a geodesic convexity assumption.

\begin{proposition}
\label{prop:GD:Cont:ConvGFMin}
Assume that the conditions in Proposition~\ref{prop:GD:Cont:NormDecay} hold,
including $\|n(\xi(0))\|\le r_0$. Let $\xi(t)$ solve~\eqref{eq:GD:Cont:GF}. Then every
accumulation point of $\xi(t)$ as $t\to\infty$ belongs to $\cM$ and is a
stationary point of $\cE$ on $\cM$.
If, in addition, $\cE$ is geodesically convex, then every accumulation
point is a global minimiser of $\cE$ on $\cM$. If $\cE$ is strictly
geodesically convex, then the minimiser is unique and
\[
\xi(t)\to \xi^*
\qquad\text{as }t\to\infty,
\]
where $\xi^*$ is the unique minimiser of $\cE$ on $\cM$.
\end{proposition}

\begin{proof}
By Proposition~\ref{prop:GD:Cont:NormDecay}, there exists $\alpha>0$ such that
\[
\|n(\xi(t))\|
\le
\|n(\xi(0))\|e^{-\alpha t}.
\]
In particular, $\xi(t)$ remains in the tubular neighbourhood of $\cM$ on which all
quantities are well-defined, and $d(\xi(t),\cM)\to0$.

We first show that the tangential component of the gradient vanishes as
$t\to\infty$. Set
\[
G(t):=\Pi_{\cT_{\xi(t)}\cM}\nabla \cE(\xi(t)).
\]
Using the flow equation~\eqref{eq:GD:Cont:GF}, we compute
\begin{align*}
\frac{\dd}{\dd t}\cE(\xi(t))
&=
\nabla \cE(\xi(t))\cdot \dot \xi(t)
\\
&=
-\left\|\Pi_{\cT_{\xi(t)}\cM}\nabla \cE(\xi(t))\right\|^2
+
\lambda \nabla \cE(\xi(t))\cdot n(\xi(t))
\\
&=
-\|G(t)\|^2
+
\lambda \nabla \cE(\xi(t))\cdot n(\xi(t)) \\
& \le
-\|G(t)\|^2
+
\lambda M\|n(\xi(0))\|e^{-\alpha t}.
\end{align*}
Integrating from $0$ to $T$ gives
\begin{align*}
\int_0^T \|G(t)\|^2\,\dd t
&\le
\cE(\xi(0))-\cE(\xi(T))
+
\lambda M\|n(\xi(0))\|\int_0^T e^{-\alpha t}\,\dd t
\\
&\le
\cE(\xi(0))
+
\frac{\lambda M}{\alpha}\|n(\xi(0))\|
\end{align*}
since $\cE(\xi(T))\ge0$. Therefore the right-hand side is bounded independently of $T$ and $G\in \Lp{2}([0,\infty);\R^D)$.
We now claim that $G(t)\to0$ as $t\to\infty$. 
Indeed, the map
\[
x\mapsto \Pi_{\cT_x\cM}\nabla \cE(x)
\]
is Lipschitz on the tubular neighbourhood, and the trajectory $\xi(t)$ has
bounded velocity there because
\[
\|\dot \xi(t)\|
\le
\|\nabla \cE(\xi(t))\|+\lambda\|n(\xi(t))\|
\le
M+\lambda r_0.
\]
Therefore $G(t)$ is uniformly continuous. Combined with $G\in \Lp{2}([0,\infty))$ this implies that $G(t)\to0$ as $t\to\infty$.

Now let $t_k\to\infty$ be any sequence. Since $\cM$ is compact and
$d(\xi(t),\cM)\to0$, the sequence $\{\xi(t_k)\}$ is bounded and has a
convergent subsequence, not relabelled, such that
\[
\xi(t_k)\to \xi^*.
\]
Since $d(\xi(t_k),\cM)\to0$, we have $\xi^*\in\cM$. Moreover, by continuity of
$x\mapsto \Pi_{\cT_x\cM}\nabla\cE(x)$ and by the fact that $G(t_k)\to0$, we
obtain
\[
0
=
\lim_{k\to\infty}
\Pi_{\cT_{\xi(t_k)}\cM}\nabla \cE(\xi(t_k))
=
\Pi_{\cT_{\xi^*}\cM}\nabla \cE(\xi^*).
\]
Thus $\xi^*$ is a stationary point of $\cE$ on $\cM$.

If $\cE$ is strictly geodesically convex then $\xi^*$ is the global minimiser of $\cE$ on $\cM$ and it is unique. Since any sequence of positive times $t_k \to \infty$ contains a subsequence converging to the same $\xi^*$, all such sequences will have to converge, and we can write $\xi(t) \to \xi^*$.
\end{proof}

\subsection{Discrete time gradient descent on known manifolds} \label{subsec:GD:Dis}

The aim of this section is to discretise the gradient flow~\eqref{eq:GD:Cont:GF}.
We break the gradient flow into two steps: in the first step we move in the direction of the projection (onto the tangent space) of the gradient in the ambient space $\nabla \cE(\xi^k)$, i.e. 
\begin{equation}\label{eq:GD:GradDes}
\tilde{\xi}^{k+1} = \xi^k - \tau_t \Pi_{\cT_{\xi^k}\cM} \nabla \cE(\xi^k), \tag{*}    
\end{equation}
and in the second step, we project $\tilde{\xi}^{k+1}$ back towards the manifold, i.e. 
\begin{equation}\label{eq:GD:ManProj}
\xi^{k+1} = \tilde{\xi}^{k+1} + \tau_n \lambda n(\tilde{\xi}^{k+1}). \tag{**}
\end{equation}

We show that the discrete-time gradient descent steps (\ref{eq:GD:GradDes}-\ref{eq:GD:ManProj}) converge to the continuum-time gradient descent PDE~\eqref{eq:GD:Cont:GF}.

The first step is to show the analogous result to Proposition~\ref{prop:GD:Cont:NormDecay}, that is, we can control how close the discrete gradient descent moves to the manifold.

\begin{proposition}
\label{prop:GD:Dis:NormDecay}
Assume there exists $r_0>0$ and $L>0$ such that if $d(x,\cM)\leq r_0$ then $\|t_j(x) - t_j(P_{\cM}(x))\|\leq L\|x-P_{\cM}(x)\|$.
Let $M=\sup_{x\,:\, d(x,\cM)\leq r_0} \|\nabla \cE(x)\|<+\infty$, let $\xi\mapsto\Pi_{\cT_\xi\cM}$ be Lipschitz continuous and suppose that the penalisation $\lambda$ is sufficiently large so that $\alpha := \lambda - dLM(D-d)>0$.
Further assume that $\tau<\frac{1}{\alpha}$.
Then, there exists $c>0$ such that $d(\tilde \xi_0,\cM) = \leq \frac{r_0}{1+\tau \lambda}$ implies
\[ d(\tilde \xi^k,\cM) \leq \max\lb \sqrt{\frac{c\tau}{\alpha}}, (1-\alpha\tau)^{k/2} d(\tilde \xi_0,\cM) \rb \]
where $\tilde{\xi}^k$ solves (\ref{eq:GD:GradDes}-\ref{eq:GD:ManProj}) with $\tau_n=\tau_t = \tau$ and $\xi^0=\xi_0$.
\end{proposition}

\begin{proof}
Recall that $d(x,\cM) = \|n(x)\|$ for any $x$ sufficiently close to $\cM$ for the projection to be well-defined. Assume $\|n(\tilde{\xi}^k)\|\leq \frac{r_0}{1+\tau\lambda} \leq r_0$.
Then, $\|\xi^k-\tilde{\xi}^k\| = \tau\lambda \|n(\tilde{\xi}^k)\|$ implies $\|n(\xi^k)\|\leq \|\xi^k-\tilde{\xi}^k\| + \|n(\tilde{\xi}^k)\| = (\tau\lambda +1)\|n(\tilde{\xi}^k)\| \leq r_0$.
Moreover, since
\[ \tilde{\xi}^{k+1} = \tilde{\xi}^k + \tau\lambda n(\tilde{\xi}^k) - \tau \Pi_{\cT_{\xi^k}\cM} \nabla \cE(\xi^k), \]
then $\|\tilde{\xi}^{k+1}-\tilde{\xi}^k\|\leq \tau(\lambda r_0 + M)$.

Applying a Taylor expansion to $x\mapsto \|n(x)\|^2$ and making use of Proposition~\ref{prop:GD:Cont:NormProp} we have
\begin{align*}
\| n(\tilde{\xi}^{k+1})\|^2 & = \| n(\tilde{\xi}^k)\|^2 - 2\tau n(\tilde{\xi}^k) \cdot \l \lambda n(\tilde{\xi}^k) - \Pi_{\cT_{\xi^k}\cM}\nabla \cE(\xi^k) \r + O(\|\tilde{\xi}^{k+1} - \tilde{\xi}^k\|^2).
\end{align*}
First we note that,
\begin{align*}
\| \Pi_{\cT_{\xi^k}\cM} \nabla \cE(\xi^k) - \Pi_{\cT_{\tilde{\xi}^k}\cM} \nabla \cE(\tilde{\xi}^k) \| & \leq \| \Pi_{\cT_{\xi^k}\cM} \nabla \cE(\xi^k) - \Pi_{\cT_{\tilde{\xi}^k}\cM} \nabla \cE(\xi^k) \| \\
 & \qquad \qquad + \| \Pi_{\cT_{\tilde{\xi}^k}\cM} \nabla \cE(\xi^k) - \Pi_{\cT_{\tilde{\xi}^k}\cM} \nabla \cE(\tilde{\xi}^k) \| \\
 & \leq M \| \Pi_{\cT_{\xi^k}\cM} - \Pi_{\cT_{\tilde{\xi}^k}\cM} \|_{\op} + \| \nabla \cE(\xi^k) - \nabla \cE(\tilde{\xi}^k) \| \\
 & \leq C \|\xi^k - \tilde{\xi}^k\|
\end{align*}
for some $C>0$ using the Lipschitz assumption on $\nabla \cE$.

Now, using~\eqref{eq:GD:Cont:tangentnormalbound},
\[ \la n(\tilde{\xi}^k) \cdot \Pi_{\cT_{\tilde{\xi}^k}\cM}\nabla \cE(\tilde{\xi}^k) \ra = \la\sum_{j=1}^d t_j(\tilde{\xi}^k) \cdot \nabla \cE(\tilde{\xi}^k) \, t_j(\tilde{\xi}^k) \cdot n(\tilde{\xi}^k)\ra \leq LMd(D-d) \|n(\tilde{\xi}^k)\|^2. \]
Hence, for $\alpha = \lambda - LMd(D-d)$,
\begin{align*}
\| n(\tilde{\xi}^{k+1})\|^2 & = \| n(\tilde{\xi}^k)\|^2 - 2\tau n(\tilde{\xi}^k) \cdot \l \lambda n(\tilde{\xi}^k) - \Pi_{\cT_{\tilde{\xi}^k}\cM}\nabla \cE(\tilde{\xi}^k) \r \\
 & \qquad \qquad + O\l\tau \|n(\tilde{\xi}^k)\| \|\xi^k-\tilde{\xi}^k\| + \|\tilde{\xi}^{k+1} - \tilde{\xi}^k\|^2\r \\
 & \leq (1-2\tau\lambda + 2\tau LMd(D-d)) \|n(\tilde{\xi}^k)\|^2 + C\tau \|n(\tilde{\xi}^k)\| \|\xi^k - \tilde{\xi}^k\| + C\|\xi^k - \tilde{\xi}^k\|^2 \\
 & \leq (1-2\alpha\tau) \|n(\tilde{\xi}^k)\|^2  + C\tau^2 r_0(\lambda r_0 + M) + C\tau^2 (\lambda r_0+M)^2 \\
 & \leq (1-2\alpha\tau) \|n(\tilde{\xi}^k)\|^2  + C^\prime\tau^2.
\end{align*}
This implies the result with $c = \frac{C^\prime}{2}$.
\end{proof}

We are now in a position to show convergence of the discrete gradient descent steps to the continuum PDE.

\begin{theorem}
Assume the conditions in Proposition~\ref{prop:GD:Dis:NormDecay} hold.
Fix an arbitrary $T>0$ and let $\xi$ solve~\eqref{eq:GD:Cont:GF} on $[0,T]$ with $\xi(0) = \xi_0$. 
Define $t_k = k\tau$ and
\[ \xi_\tau(t) = \xi^k + \frac{t-t_k}{\tau}(\xi^{k+1} - \xi^k) \qquad \text{for } t\in (t_k,t_{k+1}] \]
where $\xi^k$ solves (\ref{eq:GD:GradDes}-\ref{eq:GD:ManProj}) with $\tau_n=\tau_t = \tau$ and $\xi^0=\xi_0$.
Then,
\[ \|\xi_\tau - \xi\|_{\Lp{2}([0,T])} \to 0 \]
as $\tau\to 0$. Here, for the sake of simplicity, we use the shorthand notation $\Lp{2}([0,T])$ for the space of vector-valued functions $\Lp{2}([0,T];\R^D)$, ditto for Sobolev spaces later.
\end{theorem}

\begin{proof}
For $T>0$ choose $K_\tau\in\bbN$ such that $K_\tau \tau \leq T <(K_\tau+1)\tau$.
For $\tau>0$ sufficiently small and from Proposition~\ref{prop:GD:Dis:NormDecay} we know that if $\|n(\tilde{\xi}^0)\|\leq \frac{r_0}{1+\tau\lambda}$ then $\|n(\tilde{\xi}^k)\|\leq \frac{r_0}{1+\tau\lambda} \leq r_0$ for all $k\in\bbN$.
Similarly,
\begin{align*}
\|n(\xi^k)\| & \leq \|n(\xi^k) - n(\tilde{\xi}^k)\| + \|n(\tilde{\xi}^k)\| \\
 & \leq \|\xi^k - \tilde{\xi}^k\| + \|n(\tilde{\xi}^k)\| \\
 & = \tau\lambda \|n(\tilde{\xi}^k)\| + \|n(\tilde{\xi}^k)\| \\
 & \leq r_0.
\end{align*}
We can also infer the existence of $C>0$, independent of $\tau$, such that $\sup_{k\in\{0,1,\dots, K_\tau\}} \|\xi^k \|\leq C$.
Indeed,
\begin{align*}
\|\xi^k\| & \leq \|\xi^{k-1}\| + \tau \|\Pi_{\cT_{\xi^{k-1}}\cM} \nabla \cE(\xi^{k-1}) \| + \lambda\tau \|n(\tilde{\xi}^k)\| \\
 & \leq \|\xi^{k-1}\| + \tau M + \lambda\tau r_0 \\
 & \leq \|\xi^0\| + \tau Mk + \lambda \tau k r_0 \\
 & \leq \|\xi^0\| + T(M+\lambda r_0) =: C.
\end{align*}

The uniform bound of $\xi^k$ will allow us to now prove an $\Lp{2}$ bound on $\xi_\tau$ that is uniform in $\tau$ (assuming $\tau\leq \tau_0$). Using the inequality $\norm{a+b}^2 \leq 2 (\norm{a}^2+\norm{b}^2)$ valid for any $a,b \in \R^D$, we get
\begin{align*}
\|\xi_\tau\|_{\Lp{2}([0,T])}^2 & = \int_0^T \|\xi_\tau(t)\|^2 \, \dd t \\
 & \leq 2\sum_{k=0}^{K_\tau} \int_{t_k}^{t_{k+1}} \|\xi^k\|^2 + (t-t_k)^2 \|\Pi_{\cT_{\xi^k}\cM} \nabla \cE (\xi^k) - \lambda n(\tilde{\xi}^{k+1})\|^2 \, \dd t \\
 & \leq 2\sum_{k=0}^{K_\tau} \l \tau \|\xi^k\|^2 + \frac{2\tau^3}{3} \|\Pi_{\cT_{\xi^{k-1}}\cM} \nabla \cE(\xi^{k-1}) \|^2 + \frac{2\lambda^2\tau^3}{3} \|n(\tilde{\xi}^k)\|^2 \r \\
 & \leq 2C^2\tau (K_\tau+1) + \frac{4\tau^3 (K_\tau+1) M^2}{3} + \frac{4\lambda^2 \tau^3 (K_\tau+1) r_0^2}{3} \\
 & \leq 2C^2(T+\tau) + \frac{4\tau^2 (T+\tau)M^2}{3} + \frac{4\lambda^2 \tau^2 (T+\tau) r_0^2}{3}.
\end{align*}

Next, we show that $\dot{\xi}_\tau$ is also bounded in $\Lp{2}$.
This follows from 
\begin{align*}
\|\dot{\xi}_\tau\|_{\Lp{2}([0,T])}^2 & = \int_0^T \|\dot{\xi}_\tau(t)\|^2 \, \dd t \\
 & \leq \sum_{k=0}^{K_\tau} \int_{t_k}^{t_{k+1}} \big\|\underbrace{\dot{\xi}_\tau(t)}_{=\frac{\xi^{k+1}-\xi^k}{\tau}}\big\|^2 \, \dd t \\
 & = \frac{1}{\tau} \sum_{k=0}^{K_\tau} \| \xi^{k+1} - \xi^k\|^2 \\
 & = \tau \sum_{k=0}^{K_\tau} \| \Pi_{\cT_{\xi^k}\cM} \nabla \cE (\xi^k) - \lambda n(\tilde{\xi}^{k+1}) \|^2 \\
 & \leq 2(K_\tau+1) \tau M^2 + 2\tau\lambda^2 (K_\tau+1) r_0^2 \\
 & \leq 2(T+\tau)M^2 + 2\lambda^2 r_0^2 (T+\tau).
\end{align*}

Since $\{\xi_\tau\}_{\tau>0}$ is bounded in $\Wkp{1}{2}([0,T])$ then there exists $\xi^*\in\Wkp{1}{2}([0,T])$ and a subsequence $\xi_{\tau_m}\to \xi^*$ in $\Lp{2}([0,T])$ and $\xi_{\tau_m}\weakto \xi^*$ in $\Wkp{1}{2}([0,T])$ as $m\to\infty$.
Furthermore, by the Rellich-Kondrachov theorem in one dimension (e.g.~\cite[Thm. 4.8]{brezis2011functional}) the embedding $\Wkp{1}{2}([0,T]) \ssubset \mathrm C([0,T])$ is compact and we get that $\xi_{\tau_m}\to \xi^*$ in uniformly on $[0,T]$ and therefore $\|n(\xi^*(t))\|\leq r_0$ for all $t\in[0,T]$.

We are left to show that $\xi^*$ is the solution to~\eqref{eq:GD:Cont:GF}.
We first show that $\xi^*$ satisfies~\eqref{eq:GD:Cont:GF} in a weak sense, then we show that it is a classical solution.
Indeed, as the classical solution to~\eqref{eq:GD:Cont:GF} is unique this will imply that the whole sequence converges.

Take any $u\in\Wkp{1}{2}([0,T])$, then
\begin{align*}
\int_0^T \dot{\xi}_\tau(t) u(t) \, \dd t & = \frac{1}{\tau} \sum_{k=0}^{K_\tau} \int_{t_k}^{t_{k+1}\wedge T} (\xi^{k+1} - \xi^k) u(t) \, \dd t \\
 & = -\sum_{k=0}^{K_\tau} \int_{t_k}^{t_{k+1}\wedge T} \Pi_{\cT_{\xi^k}\cM} \nabla \cE(\xi^k) u(t) \, \dd t + \lambda\sum_{k=0}^{K_\tau} \int_{t_k}^{t_{k+1}\wedge T} n(\tilde{\xi}^{k+1}) u(t) \, \dd t \\
 & = -\underbrace{\sum_{k=0}^{K_\tau} \int_{t_k}^{t_{k+1}\wedge T} \Pi_{\cT_{\xi_\tau(t)}\cM} \nabla \cE(\xi_\tau(t)) u(t) \, \dd t}_{=:I_1(\tau)} + \underbrace{\lambda\sum_{k=0}^{K_\tau} \int_{t_k}^{t_{k+1}\wedge T} n(\tilde{\xi}_\tau(t)) u(t) \, \dd t}_{=:I_2(\tau)} \\
 & \qquad \qquad - \underbrace{\sum_{k=0}^{K_\tau} \int_{t_k}^{t_{k+1}\wedge T} \l \Pi_{\cT_{\xi^k}\cM} \nabla \cE(\xi^k) - \Pi_{\cT_{\xi_\tau(t)}\cM} \nabla \cE(\xi_\tau(t)) \r u(t) \, \dd t}_{=:I_3(\tau)} \\
 & \qquad \qquad - \underbrace{\lambda\sum_{k=0}^{K_\tau} \int_{t_k}^{t_{k+1}\wedge T} \l n(\tilde{\xi}_\tau(t)) - n(\tilde{\xi}^{k+1})\r u(t) \, \dd t}_{=:I_4(\tau)}
\end{align*}
where $\tilde{\xi}_\tau$ is defined by 
\[ \tilde{\xi}_\tau(t) = \tilde{\xi}^k + \frac{t-t_k}{\tau}(\tilde{\xi}^{k+1} - \tilde{\xi}^k) \qquad \text{for } t\in (t_k,t_{k+1}]. \]

\paragraph{Term $I_1(\tau)$.}
We show that $I_1(\tau)\to \int_0^T \Pi_{\cT_{\xi^*(t)}\cM} \nabla \cE(\xi^*(t)) u(t) \, \dd t$ as $\tau\to 0$.
If $L_{\nabla \cE}$, $L_{\cT\cM}$ are the Lipschitz constants for $\nabla \cE$, $x\mapsto\Pi_{\cT_{x}\cM}$ then we have
\begin{align*}
& \la \int_0^T \l \Pi_{\cT_{\xi_\tau(t)}\cM} \nabla \cE(\xi_\tau(t)) - \Pi_{\cT_{\xi^*(t)}\cM} \nabla \cE(\xi^*(t)) \r u(t) \, \dd t \ra \\
& \qquad \qquad \leq \int_0^T \lda \Pi_{\cT_{\xi_\tau(t)}\cM} \nabla \cE(\xi_\tau(t)) - \Pi_{\cT_{\xi_\tau(t)}\cM} \nabla \cE(\xi^*(t)) \rda \lda u(t)\rda \, \dd t \\
& \qquad \qquad \qquad \qquad + \int_0^T \lda \Pi_{\cT_{\xi_\tau(t)}\cM} \nabla \cE(\xi^*(t)) - \Pi_{\cT_{\xi^*(t)}\cM} \nabla \cE(\xi^*(t))\rda \lda u(t) \rda \, \dd t \\
& \qquad \qquad \leq \int_0^T \lda \nabla \cE(\xi_\tau(t)) - \nabla \cE(\xi^*(t)) \rda \lda u(t)\rda \, \dd t \\
& \qquad \qquad \qquad \qquad + \int_0^T \lda \Pi_{\cT_{\xi_\tau(t)}\cM} - \Pi_{\cT_{\xi^*(t)}\cM} \rda_{\op} \lda \nabla \cE(\xi^*(t))\rda \lda u(t) \rda \, \dd t \\
& \qquad \qquad \leq (L_{\nabla\cE} + L_{\cT\cM} M) \int_0^T \lda \xi_\tau(t) - \xi^*(t) \rda \lda u(t)\rda \, \dd t \\
& \qquad \qquad \leq (L_{\nabla\cE} + L_{\cT\cM} M) \|u\|_{\Lp{2}([0,T])} \|\xi_\tau - \xi^*\|_{\Lp{2}([0,T])} \to 0.
\end{align*}

\paragraph{Term $I_2(\tau)$.}
We show that $I_2(\tau)\to \int_0^T n(\xi^*(t)) u(t) \, \dd t$ as $\tau\to 0$.
We start by showing that $\|\tilde{\xi}_\tau - \xi^*\|_{\Lp{2}}\to 0$.
For $t\in (t_k,t_{k+1}]$,
\begin{align*}
\tilde{\xi}_\tau(t) & = \xi^k - \tau\lambda n(\tilde{\xi}^k) + \frac{t-t_k}{\tau} \l \xi^{k+1} - \tau\lambda n(\tilde{\xi}^{k+1}) - \xi^k + \tau\lambda n(\tilde{\xi}^k) \r \\
 & = \xi^k + \frac{t-t_k}{\tau} (\xi^{k+1}-\xi^k) + \tau\lambda n(\tilde{\xi}^k) \frac{t-t_k-\tau}{\tau} - \tau\lambda n(\tilde{\xi}^{k+1})\frac{t-t_k}{\tau} \\
 & = \xi_\tau(t) + \lambda \l n(\tilde{\xi}^k) (t-t_{k+1}) + n(\tilde{\xi}^{k+1}) (t-t_k)\r.
\end{align*}
So,
\begin{align*}
\lda \tilde{\xi}_\tau - \xi^*\rda_{\Lp{2}([0,T])} & \leq \|\tilde{\xi}_\tau - \xi_\tau\|_{\Lp{2}([0,T])} + \| \xi_\tau - \xi^*\|_{\Lp{2}([0,T])} \\
 & \leq \lambda \l \sum_{k=0}^{K_\tau} \int_{t_k}^{t_{k+1}} \l n(\tilde{\xi}^k) (t-t_{k+1}) + n(\tilde{\xi}^{k+1}) (t-t_k)\r^2 \, \dd t \r^{\frac12} + \| \xi_\tau - \xi^*\|_{\Lp{2}([0,T])} \\
 & \leq 2\lambda r_0 \l (K_\tau+1)\int_0^\tau t^2 \, \dd t \r^{\frac12} + \| \xi_\tau - \xi^*\|_{\Lp{2}([0,T])} \\
 & \leq 2\lambda r_0 \sqrt{\frac{\tau^3(K_\tau+1)}{3}} + \| \xi_\tau - \xi^*\|_{\Lp{2}([0,T])} \\
 & \leq 2\lambda r_0\tau \sqrt{\frac{T+\tau}{3}} + \| \xi_\tau - \xi^*\|_{\Lp{2}([0,T])} \to 0
\end{align*}
as $\tau\to 0$.

Now, if $L_\cM$ is the Lipschitz constant of $\cP_{\cM}$ then
\begin{align*}
\la \int_0^T \l n(\tilde{\xi}_\tau(t)) - n(\xi^*(t))\r u(t) \, \dd t \ra & = \la \int_0^T \l \underbrace{n(\tilde{\xi}_\tau(t)) + \tilde{\xi}_\tau}_{=\cP_{\cM}(\tilde{\xi}_\tau(t))} - \underbrace{\l n(\xi^*(t)) + \xi^*(t) \r}_{=\cP_{\cM}(\xi^*(t))} - \tilde{\xi}_\tau + \xi^*(t) \r u(t) \, \dd t \ra \\
 & \leq \int_0^T \|\cP_{\cM}(\tilde{\xi}_\tau(t)) - \cP_{\cM}(\xi^*(t)) \| \|u(t)\| \, \dd t \\
 & \qquad \qquad + \int_0^T \| \tilde{\xi}_\tau(t) - \xi^*(t) \| \|u(t)\| \, \dd t \\
 & \leq (L_\cM+1) \int_0^T \| \tilde{\xi}_\tau(t) - \xi^*(t) \| \|u(t)\| \, \dd t \\
 & \leq (L_\cM + 1) \| \tilde{\xi}_\tau - \xi^* \|_{\Lp{2}([0,T])} \|u\|_{\Lp{2}([0,T])} \to 0
\end{align*}
as $\tau\to 0$.

\paragraph{Term $I_3(\tau)$.}
We show that $I_3(\tau)\to 0$ as $\tau\to 0$.
Letting $L_{\nabla \cE}$ be the Lipschitz constant of $\nabla \cE$ we have
\begin{align*}
|I_3(\tau)| & \leq \sum_{k=0}^{K_\tau} \int_{t_k}^{t_{k+1}\wedge T} \lda \Pi_{\cT_{\xi^k}\cM}\nabla \cE(\xi^k) - \Pi_{\cT_{\xi^k}\cM}\nabla \cE(\xi_\tau(t))\rda \lda u(t) \rda \, \dd t \\
 & \qquad \qquad + \sum_{k=0}^{K_\tau} \int_{t_k}^{t_{k+1}\wedge T} \lda \Pi_{\cT_{\xi^k}\cM} \cE(\xi_\tau(t)) - \Pi_{\cT_{\xi_\tau(t)}\cM} \cE(\xi_\tau(t)) \rda \lda u(t) \rda \, \dd t \\
 & \leq L_{\nabla \cE} \sum_{k=0}^{K_\tau} \int_{t_k}^{t_{k+1}\wedge T} \lda \xi^k - \xi_\tau(t)\rda \lda u(t) \rda \, \dd t \\
 & \qquad \qquad + M \sum_{k=0}^{K_\tau} \int_{t_k}^{t_{k+1}\wedge T} \lda \Pi_{\cT_{\xi^k}\cM} - \Pi_{\cT_{\xi_\tau(t)}\cM}\rda_{\op} \lda u(t) \rda \, \dd t \\
 & \leq L_{\nabla \cE}\|u\|_{\Lp{2}([0,T])} \l \sum_{k=0}^{K_\tau} \int_{t_k}^{t_{k+1}\wedge T} \lda \xi^k - \xi_\tau(t)\rda^2 \, \dd t \r^{\frac12} \\
 & \qquad \qquad + M\|u\|_{\Lp{2}([0,T])} \l \sum_{k=0}^{K_\tau} \int_{t_k}^{t_{k+1}\wedge T} \lda \Pi_{\cT_{\xi^k}\cM} - \Pi_{\cT_{\xi_\tau(t)}\cM}\rda_{\op}^2 \, \dd t \r^{\frac12}.
\end{align*}
We will consider the two terms separately.

First, note that 
\[ \frac{1}{\tau} \|\xi^{k+1}-\xi^k\| = \| \Pi_{\cT_{\xi^k}\cM} \nabla \cE(\xi^k) - \lambda n(\tilde{\xi}^{k+1})\| \leq \|\nabla \cE(\xi^k)\|+\lambda \|n(\tilde{\xi}^{k+1})\| \leq C \]
for some $C>0$.

Now, on the one hand we have
\begin{align*}
\sum_{k=0}^{K_\tau} \int_{t_k}^{t_{k+1}\wedge T} \lda \xi^k - \xi_\tau(t)\rda^2 \, \dd t & = \sum_{k=0}^{K_\tau} \int_{t_k}^{t_{k+1}\wedge T} \frac{(t-t_k)^2}{\tau^2} \|\xi^{k+1}-\xi^k\|^2 \, \dd t \\
 & \leq C^2 \sum_{k=0}^{K_\tau} \int_{t_k}^{t_{k+1}} (t-t_k)^2 \, \dd t \\
 & = \frac{C^2 (K_\tau+1) \tau^3}{3} \\
 & \leq \frac{C^2 (T+\tau)\tau^2}{3} \to 0
\end{align*}
as $\tau\to 0$.
And on the other hand, using that there exists $L_{\cT\cM}$ such that $\lda \Pi_{\cT_{x}\cM} - \Pi_{\cT_{y}\cM}\rda_{\op}^2 \leq L_{\cT\cM} \|x-y\|$,
\begin{align*}
\sum_{k=0}^{K_\tau} \int_{t_k}^{t_{k+1}\wedge T} \lda \Pi_{\cT_{\xi^k}\cM} - \Pi_{\cT_{\xi_\tau(t)}\cM}\rda_{\op}^2 \, \dd t & \leq L_{\cT\cM}^2 \sum_{k=0}^{K_\tau} \int_{t_k}^{t_{k+1}\wedge T} \lda \xi^k - \xi_\tau(t)\rda^2 \, \dd t \\
 & \leq \frac{C^2 L_{\cT\cM}^2 (T+\tau)\tau^2}{3} \to 0
\end{align*}
as $\tau\to 0$.
Hence $I_3(\tau)\to 0$.

\paragraph{Term $I_4(\tau)$.}
We show that $I_4(\tau)\to 0$ as $\tau\to 0$. Remembering that $u \in \Wkp{1}{2}([0,1])$ and hence continuous and letting $L_\cM$ be the Lipschitz constant for $\cP_{\cM}$ we have
\begin{align*}
|I_4(\tau)| & \leq \lambda \|u\|_{\Lp{\infty}} \sum_{k=0}^{K_\tau} \int_{t_k}^{t_{k+1}\wedge T} \|n(\tilde{\xi}_\tau(t)) - n(\tilde{\xi}^{k+1}) \| \, \dd t \\
 & = \lambda \|u\|_{\Lp{\infty}} \sum_{k=0}^{K_\tau} \int_{t_k}^{t_{k+1}\wedge T} \Big\| \underbrace{n(\tilde{\xi}_\tau(t)) + \tilde{\xi}_\tau(t)}_{=\cP_{\cM}(\tilde{\xi}_\tau(t))} - \underbrace{\l n(\tilde{\xi}^{k+1}) + \tilde{\xi}^{k+1}\r}_{=\cP_{\cM}(\tilde{\xi}^{k+1})} - \tilde{\xi}_\tau(t) + \tilde{\xi}^{k+1} \Big\| \, \dd t \\
 & \leq \lambda L_\cM \|u\|_{\Lp{\infty}} \sum_{k=0}^{K_\tau} \int_{t_k}^{t_{k+1}\wedge T} \lda \tilde{\xi}_\tau(t) - \tilde{\xi}^{k+1} \rda \, \dd t 
 + \lambda \|u\|_{\Lp{\infty}} \sum_{k=0}^{K_\tau} \int_{t_k}^{t_{k+1}\wedge T} \lda \tilde{\xi}_\tau(t) - \tilde{\xi}^{k+1} \rda \, \dd t \\
 & \leq \lambda \|u\|_{\Lp{\infty}} \l L_{\cM} + 1\r \sum_{k=0}^{K_\tau} \int_{t_k}^{t_{k+1}\wedge T} \lda \frac{\tau+t_k-t}{\tau} \l\tilde{\xi}^{k+1} - \tilde{\xi}^{k} \r \rda \, \dd t \\
 & \leq \frac{\lambda \tau \|u\|_{\Lp{\infty}}(L_\cM+1)}{2} \sum_{k=0}^{K_\tau} \lda \tilde{\xi}^{k+1} - \tilde{\xi}^{k} \rda.
\end{align*}
Now we can infer,
\[ \lda \tilde{\xi}^{k+1} - \tilde{\xi}^{k} \rda = \lda \xi^{k+1} - \xi^{k} - \tau\lambda n(\tilde{\xi}^{k+1}) + \tau\lambda n(\tilde{\xi}^{k}) \rda \leq C\tau + \tau \lambda \lda n(\tilde{\xi}^{k+1}) - n(\tilde{\xi}^{k}) \rda \leq (C+2\lambda r_0)\tau. \]
Continuing with the bound in $I_4$ we have
\[ |I_4(\tau)| \leq \frac{(C+\lambda r_0)\lambda \tau^2 \|u\|_{\Lp{\infty}}(L_\cM+1)(K_\tau+1)}{2} \leq \frac{(C+2\lambda r_0)\lambda \tau \|u\|_{\Lp{\infty}}(L_\cM+1)(T+\tau)}{2} \to 0 \]
as $\tau\to 0$.

\paragraph{Converging to a weak solution.}
Collecting $I_1$--$I_4$ we have
\[ \int_0^T \dot{\xi}_\tau(t) u(t) \, \dd t \to \int_0^T \l-\Pi_{\cT_{\xi^*(t)}\cM} \nabla \cE(\xi^*(t)) + \lambda n(\xi^*(t))\r u(t) \, \dd t \]
as $\tau\to 0$ for all $u\in\Wkp{1}{2}([0,T])$.
From the weak convergence of the derivative  $\dot{\xi}_\tau\weakto \dot{\xi}^*$ in $\Lp{2}([0,T])$ we also have
\[ \int_0^T \dot{\xi}_\tau(t) u(t) \, \dd t \to \int_0^T \dot{\xi}^*(t) u(t) \, \dd t. \]
Hence, $\xi^*$ satisfies (the weak form of)~\eqref{eq:GD:Cont:GF}.

\paragraph{Equivalence of weak and strong solutions.}

We have that
\[ \dot{\xi}(t) = -\Pi_{\cT_{\xi^*(t)}\cM} \nabla \cE(\xi^*(t)) + \lambda n(\xi^*(t)) \]
for almost every $t\in[0,T]$.
But since the RHS is continuous then this can be extended to every $t\in[0,T]$.
By the Fundamental Theorem of Calculus $\xi$ is differentiable in the classical sense and so the weak solution is a classical solution.
\end{proof}

\subsection{Gradient descent on unknown manifolds} \label{subsec:GD:Unknown}

We now relax the assumption that the tangent plane is known at every $x\in\cM$ and assume  that we have access only to a parametrisation of the tangent space $\{\hat{t}^{(n,\eps)}_j(x_i)\}_{j=1}^d$  for every sample $\{x_i\}_{i=1}^n$. 
We will use $\{\hat{t}^{(n,\eps)}_j(x_i)\}_{j=1}^d$ to approximate $t_j(\xi^k)$.
In particular, we approximate the tangent projection (Step~\eqref{eq:GD:GradDes}) by averaging projections in the directions given by projections of the gradient $\nabla \cE(\xi^k)$ onto the tangent spaces at nearby samples $\{x_i\}_{i\in\cI}$, i.e.
\begin{equation} \label{eq:GD:Manifold:GradDes}
\tilde{\xi}^{k+1} = \xi^k - \tau_t \sum_{i\in\cI^k} w_i^k \Pi_{x_i,n,\eps} \nabla \cE(\xi^k) 
\end{equation}
where $\sum_{i\in\cI^k} w_i^k = 1$, $\Pi_{x_i,n,\eps} = \hat{\Sigma}_{n,\eps}(x_i) \hat{\Sigma}_{n,\eps}(x_i)^\top$,
\[ \hat{\Sigma}_{n,\eps}(x) = \ls \hat{t}_1^{(n,\eps)}(x), \hat{t}_2^{(n,\eps)}(x), \dots, \hat{t}_d^{(n,\eps)}(x) \rs \in \bbR^{D\times d} \]
and $\hat{t}_i^{(n,\eps)}$ are the (ordered) eigenvectors of $\hat{C}_{n,\eps}(x)\in\bbR^{D\times D}$ defined by 
\[ [\hat{C}_{n,\eps}(x)]_{ab} = \frac{1}{\hat{N}_{n,\eps}(x)} \sum_{x_i \,:\, |x_i-x|<\eps} [x-x_i]_a [x-x_i]_b, \qquad \hat{N}_{n,\eps}(x) = \#\{x_i\,:\, |x_i-x|<\eps\}, \]
where $[\cdot]_{a}$ denotes the $a$-th component of a vector. For example, we could choose the set $\cI^k = \cI(\xi^k)$ to index the nearest neighbours of $\xi^k$, or the subset of $\{x_i\}_{i=1}^n$ that are closer than some specified tolerance from $\xi^k$.
The choice of the weights, $\{w_i^k\}_{i\in\cI^k}$, could be, for example, uniform or such that $\xi^k$ is the weighted average (using the weights $\{w_i^k\}_{i\in\cI^k}$) of $\{x_i\}_{i\in\cI^k}$, i.e.
\[ w_i^k = \frac{1}{|\cI^k|} \qquad \text{or} \qquad \xi^k = \sum_{i\in\cI^k} w_i^k x_i. \]

For the second step, we approximate the movement of $\tilde{\xi}^{k+1}$ towards $\cM$ by estimating the normal vector to the tangent at $\tilde{\xi}^{k+1}$ by
\begin{equation}
\hat{n}(\tilde{\xi}^{k+1};\xi^k) =  \sum_{i\in\cI^k} w_i^{k} \Pi_{x_i,n,\eps} (\tilde{\xi}^{k+1}-m(\xi^k)) - (\tilde{\xi}^{k+1} - m(\xi^k))
\label{eq:def-hat-n}
\end{equation} 
where 
\begin{equation} \label{eq:def-m}
m(\xi^k) = \sum_{i\in\cI^k} w_i^k x_i
\end{equation} 
is the ``local average of the manifold'' and with the notation $\hat{n}(\tilde{\xi}^{k+1};\xi^k)$ we emphasise that the neighbourhood $\cI_k$ is calculated at the point $\xi^k$ rather than $\tilde \xi^{k+1}$. However, we will also use the notation $\hat{n}(\tilde{\xi}^{k+1}) = \hat{n}(\tilde{\xi}^{k+1};\xi^k)$ when no confusion can arise.
Hence Step~\eqref{eq:GD:ManProj} is replaced by
\begin{equation} \label{eq:GD:Manifold:ManProj}
\xi^{k+1} = \tilde{\xi}^{k+1} + \tau_n \lambda \hat n(\tilde{\xi}^{k+1}).
\end{equation}

Algorithm~\ref{alg:GD:Manifold} gives pseudo-code for the method (where for simplicity we assume $w_i^k=\frac{1}{|\cI^k|}$).

\begin{algorithm}[ht]
\caption{Gradient Descent on Manifolds}
\label{alg:GD:Manifold}
\begin{algorithmic}
\State {\bf Input:} Feature vectors $\{x_i\}_{i=1}^n$, dimension $d$ of the manifold, number $K$ of nearest neighbours, the gradient of the energy $\nabla\cE:\bbR^D\to \bbR^D$, tangent space basis vectors $\{\hat{t}^{(n)}_j(x_i)\}_{j=1}^d$ for each $x_i$ ($i=1,\dots,n$), step sizes $\tau_n$ and $\tau_t$, normal weighting $\lambda$, and initial guess $\xi^{0}$.
\State {\bf Output:} Estimate $\xi^*$ of the minimiser of $\cE$ on the manifold $\cM$.
\State
\State Set $k=0$.
\While{Not converged}
\State Find $K$ nearest neighbours of $\xi^{k}$ in $\{x_i\}_{i=1}^n$, let $\cI^k$ index the set of K-NN's.
\State Define $v = \frac{1}{|\cI^k|} \sum_{i\in \cI^k} \Pi_{x_i,n,\eps} \nabla\cE(\xi^k)$.
\State Update $\tilde{\xi}^{k+1} = \xi^k - \tau_t v$. \Comment{Gradient descent step.}
\State Define $m^k = m(\xi^k) = \frac{1}{|\cI^k|} \sum_{i\in\cI^k} x_i$.
\State Define $\hat{n}(\tilde{\xi}^{k+1}) = \frac{1}{|\cI^k|} \sum_{i\in\cI^k} \Pi_{x_i,n,\eps} (\tilde{\xi}^{k+1} - m^k) - (\tilde{\xi}^{k+1} - m^k)$.
\State Update $\xi^{k+1} = \tilde{\xi}^{k+1} + \tau_n \lambda \hat{n}(\tilde{\xi}^{k+1})$. \Comment{Projection onto the manifold.}
\EndWhile
\State Set $\xi^* = \xi^{k+1}$.
\end{algorithmic}
\end{algorithm}

Note that the tangent spaces are approximated by local PCA.
For a tangent space at $x_i$ we choose a subset of feature vectors $\{x_j\}_{j=1}^n$ that are near $x_i$. 
More specifically, we specify a radius $\eps$ such that we use any $x_j$ satisfying $\|x_j-x_i\|\leq \eps$ for the local PCA approximation. 
That is, we define the empirical covariance matrix $\hat{C}_{n,\eps}(x_i)$ and use the $d$ leading eigenvectors to define our approximation of the tangent space.
See Algorithm~\ref{alg:GD:TanPlane} for pseudo-code.
We note that the tangent space approximation can be pre-computed for Algorithm~\ref{alg:GD:Manifold}.

We can bound the difference between the PCA approximated tangent plane approximation and the true tangent plane projection through the following theorem.
We assume that $x_i\iid \mu \in \cP(\cM)$ where $\mu$ has a $\Ck{2}$ density (with respect to the volume form on $\cM$).
The result is well-known if the density is constant, e.g.~\cite{tyagi2013tangent}, but as far as the authors are aware the proof has not been written down for a density (we also observe that whilst the rate in Theorem~\ref{thm:GD:Unknown:TngConv} coincides with the rate in the uniform density case, if $\rho$ is not twice continuously differentiable then we don't believe the same rate holds).
In the theorem $\|\cdot\|_{\rmF}$ is the Frobenius norm.

\begin{theorem}
\label{thm:GD:Unknown:TngConv}
Assume $\cM$ is a smooth manifold and $x_i\iid \mu$ where $\mu\in\cP(\cM)$ has a density $\rho\in\Ck{2}(\cM)$ with respect to the volume form on the manifold that is bounded above and below by strictly positive constants.
Then, there exists $C>c>0$, $C^\prime>0$ such that if $C\eps^d<\alpha<1$ then 
\[ \| \Pi_{x,n,\eps} - \Pi_{\cT_x\cM} \|_{\rmF} \leq \alpha + C^\prime\eps^2 \]
for all $x\in \{x_i\}_{i=1}^n$ with probability at least $1-Cne^{-cn\alpha\eps^d}$.
\end{theorem}

The theorem is a consequence of the following two lemmas (Lemmas~\ref{lem:GD:Unknown:Dis2CtsNL} and~\ref{lem:GD:Unknown:nonlocal-to-local})  that compare $\Pi_{x,n,\eps}$ and $\Pi_{x}$ with $\Pi_{x,\eps} = \Sigma_\eps(x) \Sigma_\eps(x)^\top$ where
\[ \Sigma_\eps(x) = \ls t_1^{(\eps)}(x), t_2^{(\eps)}(x), \dots, t_d^{(\eps)}(x)\rs \in \bbR^{D\times d} \]
and $t_i^{(\eps)}$ are the (ordered) eigenvectors of $C_\eps(x)\in\bbR^{D\times D}$ defined by
\[ [C_\eps(x)]_{ab} = \frac{1}{\mu(\cM\cap B(x,\eps))} \int_{\cM\cap B(x,\eps)} [x-y]_a [x-y]_b \, \dd \mu(y). \]

\begin{lemma}
\label{lem:GD:Unknown:Dis2CtsNL}
Assume $x_i\iid \mu$ where $\mu\in\cP(\cM)$ has a density $\rho:\cM\to [0,+\infty)$ with respect to the volume form on the manifold that is bounded above and below by strictly positive constants.
Then, there exists $C>c>0$ such that if $C\eps^d<\alpha<1$ then 
\[ \| \Pi_{x,n,\eps} - \Pi_{x,\eps} \|_{\rmF} \leq \alpha \]
for all $x\in \{x_i\}_{i=1}^n$ with probability at least $1-Cne^{-cn\alpha\eps^d}$.
\end{lemma}

\begin{proof}
By the Davis-Kahan theorem
\begin{equation} \label{eq:GD:Unknown:DKThm}
 \|\Pi_{x,n,\eps} - \Pi_{x,\eps} \|_{\rmF} \leq \frac{\lda\l\hat{C}_{n,\eps}(x) - C_\eps(x)\r \Sigma_\eps(x)\rda_{\rmF}}{|[\hat{\lambda}_{n,\eps}(x)]_d - [\lambda_\eps(x)]_{d+1}|}
\end{equation}
where $[\hat{\lambda}_{n,\eps}(x)]_k$, $[\lambda_\eps(x)]_k$ are the $k$th eigenvalues of $\hat{C}_{n,\eps}(x)$ and $C_\eps(x)$ respectively.

For any $A\in \bbR^{D\times D}$ and $\Sigma\in\bbR^{D\times d}$ with normalised columns we can bound
\[ \| A\Sigma\|_{\rmF} = \sqrt{\sum_{i=1}^D \sum_{j=1}^d \l \sum_{k=1}^D A_{ik} \Sigma_{kj} \r^2} \!\leq\! \sqrt{\sum_{i=1}^D \sum_{j=1}^d \l \sum_{k=1}^D A_{ik}^2\r \underbrace{\l \sum_{k=1}^D \Sigma_{kj}^2 \r}_{=1}} = \sqrt{d} \sqrt{\sum_{i=1}^D \sum_{k=1}^d A_{ik}^2} = \sqrt{d} \|A\|_{\rmF}. \]
Hence,
\begin{equation} \label{eq:GD:Unknown:FBound}
\lda\l\hat{C}_{n,\eps}(x) - C_\eps(x)\r \Sigma_\eps(x)\rda_{\rmF} \leq \lda\hat{C}_{n,\eps}(x) - C_\eps(x)\rda_{\rmF}.
\end{equation}

Now, letting $m = \frac{1}{\mu(\cM\cap B(x,\eps))} \int_{\cM\cap B(x,\eps)} [x-y]_a [x-y]_b \, \dd \mu(y)$,
\[ \hat{N}_{n,\eps}(x) \la [\hat{C}_{n,\eps}(x)]_{ab} - [C_\eps(x)]_{ab} \ra = \la \sum_{x_i\,:\, |x-x_i|<\eps} \l [x-x_i]_a [x-x_i]_b - m \r \ra <\delta \]
with probability $1-2\exp\l-\frac{c\delta^2}{\hat{N}_{n,\eps}(x)\eps^{d+4}+\delta\eps^2}\r$ by Bernstein's inequality and conditioned on $\hat{N}_{n,\eps}(x)$.
We choose $\delta = n\alpha\eps^{d+2}$ 
so that
\[ \la [\hat{C}_{n,\eps}(x)]_{ab} - [C_\eps(x)]_{ab} \ra \leq \frac{n\alpha\eps^{d+2}}{\hat{N}_{n,\eps}(x)} \]
with probability at least $1-2\exp\l-\frac{cn\alpha^2}{\frac{\hat{N}_{n,\eps}(x)}{n\eps^d}+\frac{\alpha}{\eps^d}}\r$ (conditioned on $\hat{N}_{n,\eps}(x)$).
A similar application of Bernstein's inequality implies
\[ \la\hat{N}_{n,\eps}(x) - n\mu(\cM\cap B(x,\eps))\ra \leq n\eps^d \]
with probability at least $1-2e^{-cn\eps^d}$.
Since $\mu(\cM\cap B(x,\eps))\sim \eps^d$ then there exists $C_1>C_2>0$ such that $\hat{N}_{n,\eps}(x)\in [C_1n\eps^d,C_2n\eps^d]$ with probability at least $1-2e^{-cn\eps^d}$.
It follows that
\[ \la [\hat{C}_{n,\eps}(x)]_{ab} - [C_\eps(x)]_{ab} \ra \leq \alpha\eps^2 \]
with probability at least $1-Ce^{-cn\alpha\eps^{d}}$ (since $\alpha\gtrsim \eps^d$).
Hence, the bound on the Frobenius norm
\begin{equation} \label{eq:GD:Unknown:ProjNormBound}
\|\hat{C}_{n,\eps}(x) - C_\eps(x)\|_{\rmF} \leq \alpha \eps^2
\end{equation}
holds with the same probability.

We now turn to $|[\hat{\lambda}_{n,\eps}(x)]_d - [\lambda_\eps(x)]_{d+1}|$.
By Weyl's theorem and repeating the above argument,
\begin{equation}
\label{eq:GD:Unknown:lambdaddiff}
|[\hat{\lambda}_{n,\eps}(x)]_d - [\lambda_\eps(x)]_d| \leq \|\hat{C}_{n,\eps}(x) - C_\eps(x)\|_{\op} \leq C \|\hat{C}_{n,\eps}(x) - C_\eps(x)\|_{\infty} \leq \beta\eps^2
\end{equation}
with probability at least $1-Ce^{-cn\beta\eps^{d}}$ for any $\beta\gtrsim \eps^d$.

By the Courant--Fischer--Weyl min-max principle we have
\[ [\lambda_\eps(x)]_d = \max_{\dim(A)=d} \min_{\substack{z\in A\\ \|z\|=1}} z^\top C_\eps(x) z \geq \min_{\substack{z\in\cT_x\cM\\ \|z\|=1}} z^\top C_\eps(x)z \]
where the first max is taken over linear subspaces $A\subset \bbR^D$ of dimension $d$ and the inequality comes from choosing $A=\cT_x\cM$.
Take $z\in\cT_x\cM$ with $\|z\|=1$.
After rotating and shifting we can assume (without loss of generality) that $z=(1,0,0,\dots,0)\in \bbR^D$, $x=0\in\bbR^D$ and $\cT_x\cM = \lb y\in\bbR^D\,:\, y_i=0 \,\forall i\in\{d+1,\dots,D\}\rb$.
Then,
\begin{align*}
z^\top C_\eps(x)z & = \frac{1}{\mu(\cM\cap B(0,\eps))} \sum_{a,b=1}^D \int_{\cM\cap B(0,\eps)} y_a y_b z_a z_b \rho(y) \, \dd \Vol_{\cM}(y) \\
 & = \frac{1}{\mu(\cM\cap B(0,\eps))} \int_{\cM\cap B(0,\eps)} y_1^2 \rho(y) \, \dd \Vol_{\cM}(y) \\
 & \geq \frac{\min\rho}{\max\rho \Vol_\cM(B(0,\eps))} \int_{\cM\cap B(0,\eps)} y_1^2 \, \dd \Vol_\cM(y) \\
 & \geq c \eps^2
\end{align*}
for some $c>0$ that depends on $\cM$ but can be chosen independently of $\eps$ and $x$.
Hence,
\begin{equation}
\label{eq:GD:Unknown:lambdadlower}
[\lambda_\eps(x)]_d \geq c \eps^2.
\end{equation}

Similarly, to upper bound $[\lambda_\eps(x)]_{d+1}$ we can use the Courant--Fischer--Weyl min-max principle again to infer,
\[ [\lambda_\eps(x)]_{d+1} = \min_{\dim(A) = D-d} \max_{\substack{z\in A\\ \|z\|=1}} z^\top C_\eps(x)z \leq \max_{\substack{z\in(\cT_x\cM)^\perp \\ \|z\|=1}} z^\top C_\eps(x)z \]
where the first min is taken over linear subspaces $A\subset \bbR^D$ of dimension $D-d$ and the inequality comes from choosing $A=(\cT_x\cM)^\perp$.
Again rotating so that $x=0$ and the tangent plane is as before, but now with $z=(0,\dots, 0,1)$, we have
\[ z^\top C_\eps(x)z = \frac{1}{\mu(\cM \cap B(0,\eps))} \int_{\cM\cap B(0,\eps)} y_D^2 \rho(y) \, \dd \Vol_\cM(y). \]
Working with a local atlas $\phi:\bbR^d\to \bbR^{D-d}$ at $x=0$ (i.e. $(y_{1:d},\varphi(y_{1:d}))\in\cM$ for all $y$ small enough) we can write $y_D = \phi_{D-d}(y_{1:d}) = [\phi(y_{1:d})]_{D-d}$
where $\phi$ is smooth and satisfies $\phi(0)=0$ and $\nabla \phi(0) = 0$.
Hence, by a Taylor expansion to degree 2 we can find some $C>0$ such that $|y_D|\leq C\|y_{1:d}\|^2$.
Then, $z^\top C_\eps(x)z\leq C\eps^4$ and in particular,
\begin{equation}
\label{eq:GD:Unknown:lambdad1upper}
[\lambda_\eps(x)]_{d+1} \leq C\eps^4.
\end{equation}

Combining~\eqref{eq:GD:Unknown:lambdaddiff}, \eqref{eq:GD:Unknown:lambdadlower} and~\eqref{eq:GD:Unknown:lambdad1upper} we have
\[ |[\hat{\lambda}_{n,\eps}(x)]_d - [\lambda_\eps(x)]_{d+1}| \geq |[\lambda_{\eps}(x)]_d - [\lambda_\eps(x)]_{d+1}| - |[\hat{\lambda}_{n,\eps}(x)]_d - [\lambda_\eps(x)]_d| \geq c\eps^2 - C\eps^4 - \beta \eps^2 \]
with probability at least $1-Ce^{-cn\beta\eps^{d}}$.
We choose $\beta$ sufficiently small so that
\begin{equation} \label{eq:GD:Unknown:EVals}
|[\hat{\lambda}_{n,\eps}(x)]_d - [\lambda_\eps(x)]_{d+1}| \geq c\eps^2
\end{equation}
with probability at least $1-Ce^{-cn\eps^{d}}$.

Combining~\eqref{eq:GD:Unknown:FBound}, ~\eqref{eq:GD:Unknown:ProjNormBound} and~\eqref{eq:GD:Unknown:EVals} with~\eqref{eq:GD:Unknown:DKThm} and union bounding we conclude the lemma.
\end{proof}

Next we prove the non-local to local bound.

\begin{lemma}
\label{lem:GD:Unknown:nonlocal-to-local}
Assume $\cM$ is a smooth manifold and $\mu\in\cP(\cM)$ has a $\Ck{2}(\cM)$ density $\rho:\cM\to [0,+\infty)$ with respect to the volume form on the manifold that is bounded above and below by strictly positive constants.
Then, there exists $C>0$ such that
\[ \|\Pi_{x,\eps} - \Pi_{\cT_x\cM} \|_{\rmF} \leq C\eps^2 \]
for all $x\in\cM$ and $\eps>0$.
\end{lemma}

\begin{proof}
Pick $x\in\cM$, for $y$ near $x$ we can write
\[ y = x+t+h(t)=:f(t) \]
where $t\in\cT_x\cM$ and $h(t)\in\cN_x\cM = (\cT_x\cM)^\perp$.
By rotating and shifting we can assume without loss of generality that $x=0$ and 
\begin{align*}
\cT_x\cM & = \lb z\in\bbR^D \,:\, z_i=0 \,\forall i\in\{d+1,\dots,D\}\rb \\
\cN_x\cM & = \lb z\in\bbR^D \,:\, z_i=0 \,\forall i\in\{1,\dots,d\}\rb.
\end{align*}
Since $t\in\cT_x\cM$ is completely determined by $t_{1:d}$ then we will find it more convenient to write
\[ y = x + \l \begin{array}{c} t_{1:d} \\ h(t_{1:d}) \end{array} \r \]
where with an abuse of notation we now let $h:\bbR^d\to \bbR^{D-d}$.
Since the manifold is smooth we can use a Taylor series expansion to write $h(t_{1:d}) = h_0(t_{1:d}) + O(|t|^3)$ where $h_0(z) = z^\top H_0 z$ for some $H_0\in\bbR^{d\times d}$. 
Letting $t = (\eps z^\top,0^\top)^\top$ for $z\in\bbR^d$ we have
\[ y - x = \l \begin{array}{c} \eps z \\ \eps^2 h_0(z) + O(\eps^3) \end{array} \r. \]
Therefore,
\[ (y-x)(y-x)^\top = \l \begin{array}{cc} \eps^2 z z^\top & \eps^3 zh_0(z)^\top + O(\eps^4) \\ \eps^3 h_0(z) z^\top +O(\eps^4) & O(\eps^4) \end{array} \r. \]

If $J_x(v)$ is the Jacobian corresponding to the change of variables
\[ \int_{B(x,\delta)} F(y) \, \dd \Vol_\cM(y) = \int_{\bbR^d} F(h(z)) J_x(z) \, \dd z \]
(valid for $\delta>0$ sufficiently small) then one can bound $|J_x(z)-1|\leq C|z|^2$ for some $C$ depending only on $\cM$ by, for example,~\cite{burago2015graph,burago2001course}.

Then,
\begin{align*}
C_\eps(x) & = \frac{1}{\mu(\cM\cap B(x,\eps))} \int_{\cM\cap B(x,\eps)} (x-y) (x-y)^\top \rho(y) \, \dd \Vol_{\cM}(y) \\
 & = \frac{\eps^d}{\mu(\cM\cap B(0,\eps))} \int_{|\eps z|^2+|h(\eps z)|^2\leq \eps^2} \!\ls\!\! \begin{array}{cc} \eps^2 z z^\top & \eps^3 zh_0(z)^\top + O(\eps^4) \\ \eps^3 h_0(z) z^\top +O(\eps^4) & O(\eps^4) \end{array} \!\!\rs\! \rho(h(\eps z)) J_0(\eps z) \, \dd z.
\end{align*}

Assume $z\in\bbR^d$ satisfies $|\eps z|^2 + |h(\eps z)|^2 \leq \eps^2$.
Then,
\[ \sqrt{ |\eps z|^2 + |h_0(\eps z)|^2} \leq \sqrt{|\eps z|^2 + |h(\eps z)|^2} + |h_0(\eps z) - h(\eps z)| \leq \eps + C\eps^3 \]
for some $C>0$. 
In a similar way, if $|\eps z|^2 + |h_0(\eps z)|^2 \leq 1-C\eps^2$ then (for $C>0$ large enough) $|\eps z|^2 + |h(\eps z)|^2 \leq \eps^2$.
Hence,
\[ \lb z: |z|^2+\la\frac{h_0(\eps z)}{\eps}\ra^2\leq 1-C\eps^2\rb \subseteq \lb z: |z|^2+\la\frac{h(\eps z)}{\eps}\ra^2\leq 1\rb \subseteq \lb z: |z|^2+\la\frac{h_0(\eps z)}{\eps}\ra^2\leq 1+C\eps^2\rb. \]
In particular,
\[ \Vol\l \lb z\,:\,|\eps z|^2 + |h_0(\eps z)|^2 \leq \eps^2\rb \Delta \lb z\,:\,|\eps z|^2 + |h(\eps z)|^2 \leq \eps^2\rb \r = O(\eps^2) \]
where $\Delta$ is the symmetric set difference ($A\Delta B = (A\setminus B) \cup (B\setminus A)$).

By the above, we can modify the domain of integration in $C_\eps$ with the loss of introducing an $O(\eps^4)$ error
\begin{align*}
C_\eps(x) & = \frac{\eps^d}{\mu(\cM\cap B(0,\eps))} \int_{|z|^2+|\eps h_0(z)|^2\leq 1} \ls\! \begin{array}{cc} \eps^2 z z^\top & \eps^3 zh_0(z)^\top + O(\eps^4) \\ \eps^3 h_0(z) z^\top +O(\eps^4) & O(\eps^4) \end{array} \!\rs \rho(h(\eps z)) J_0(\eps z) \, \dd z \\
 & \qquad \qquad + O(\eps^4).
\end{align*}

We consider each of the four parts of $C_\eps$ separately.

In the first part we have
\begin{align*}
& \frac{\eps^{d+2}}{\mu(\cM\cap B(0,\eps))} \int_{|z|^2+|\eps h_0(z)|^2\leq 1} z z^\top \rho(h(\eps z)) J_0(\eps z) \, \dd z \\
& \qquad \qquad = \frac{\eps^{d+2}}{\mu(\cM\cap B(0,\eps))} \int_{|z|^2+|\eps h_0(z)|^2\leq 1} z z^\top \l \rho(0) + \eps\nabla (\rho\circ h)(0) \cdot z + O(\eps^2) \r \l 1+O(\eps^2)\r \, \dd z \\
& \qquad \qquad = \frac{\eps^{d+2}\rho(0)}{\mu(\cM\cap B(0,\eps))} \int_{|z|^2+|\eps h_0( z)|^2\leq 1} z z^\top \, \dd z + \frac{O(\eps^{d+4})}{\mu(\cM\cap B(0,\eps))} \\
& \qquad \qquad = \frac{\eps^{d+2}\rho(0)}{\mu(\cM\cap B(0,\eps))} \int_{|z|^2+|\eps h_0(z)|^2\leq 1} z z^\top \, \dd z + O(\eps^4).
\end{align*}
where for the second equality we use antisymmetry of the function $z\mapsto zz^\top\nabla (\rho\circ h)(0) \cdot z$ and symmetry of the domain $z\in \{|z|^2+|\eps h_0(z)|^2\leq 1\}$ if and only if $-z\in \{|z|^2+|\eps h_0(z)|^2\leq 1\}$ (which uses the fact that $h_0(z) = h_0(-z)$) to infer
\[ \int_{|z|^2+|\eps h_0(z)|^2\leq 1} z z^\top \nabla (\rho\circ h)(0) \cdot z \, \dd z = 0 \]
Now,
\[ \int_{|z|^2+|\eps h_0( z)|^2\leq 1} z_i z_j \, \dd z = \lb \begin{array}{cc} 0 & \text{if } i\neq j \\ \int_{|z|^2+|\eps h_0( z)|^2\leq 1} z_i^2 \, \dd z & \text{if } i= j. \end{array} \rd \]
Let 
\[ \Lambda_\eps = \frac{\eps^d \rho(0)}{\mu(\cM\cap B(0,\eps))} \diag
\l \int_{|z|^2+|\eps h_0( z)|^2\leq 1} z_1^2 \, \dd z,\dots, \int_{|z|^2+|\eps h_0( z)|^2\leq 1} z_d^2 \, \dd z\r \in\bbR^{d\times d}. \]
Therefore,
\[ \frac{\eps^{d+2}}{\mu(\cM\cap B(0,\eps))} \int_{|z|^2+|\eps h_0(z)|^2\leq 1} z z^\top \rho(h(\eps z)) J_x(\eps z) \, \dd z = \eps^2 \Lambda_\eps + O(\eps^4). \]

In the second term we have
\begin{align*}
& \frac{\eps^{d+3}}{\mu(\cM\cap B(0,\eps))} \int_{|z|^2+|\eps h_0(z)|^2\leq 1} \l z h_0(z)^\top + O(\eps)\r \rho(h(\eps z)) J_0(\eps z) \, \dd z \\
& \qquad \qquad = \frac{\eps^{d+3}}{\mu(\cM\cap B(0,\eps))} \int_{|z|^2+|\eps h_0(z)|^2\leq 1} \l z h_0(z)^\top + O(\eps)\r \l \rho(0) + O(\eps)\r \l 1+O(\eps^2)\r \, \dd z \\
& \qquad \qquad = \frac{\eps^{d+3}\rho(0)}{\mu(\cM\cap B(0,\eps))} \int_{|z|^2+|\eps h_0(z)|^2\leq 1} z h_0(z)^\top \, \dd z + O(\eps^4).
\end{align*}
Since $z\mapsto zh_0(z)$ is antisymmetric (since $h_0$ is quadratic with no linear and constant term and the domain of integration is symmetric then $\int_{|z|^2+|\eps h_0(z)|^2\leq 1} z h_0(z)^\top \, \dd z = 0$.
Hence,
\[ \frac{\eps^{d+3}}{\mu(\cM\cap B(0,\eps))} \int_{|z|^2+|\eps h_0(z)|^2\leq 1} \l z h_0(z)^\top + O(\eps)\r \rho(h(\eps z)) J_0(\eps z) \, \dd z = O(\eps^4). \]

Analogously the third term is also $O(\eps^4)$:
\[ \frac{\eps^{d+3}}{\mu(\cM\cap B(0,\eps))} \int_{|z|^2+|\eps h_0(z)|^2\leq 1} \l h_0(z) z^\top + O(\eps)\r \rho(h(\eps z)) J_0(\eps z) \, \dd z = O(\eps^4). \]

By boundedness of $\rho$ we immediately have that the fourth and final term is $O(\eps^4)$:
\[ \frac{\eps^{d+4}}{\mu(\cM\cap B(0,\eps))} \int_{|z|^2+|\eps h_0(z)|^2\leq 1} O(1) \rho(h(\eps z)) J_0(\eps z) \, \dd z = O(\eps^4). \]

We have shown that
\[ C_\eps(x) = \underbrace{\eps^2 \ls \begin{array}{cc} \Lambda_\eps & 0 \\ 0 & 0 \end{array} \rs}_{=:\hat{C}_\eps(x)} + O(\eps^4). \]
In particular, $\|C_\eps(x) - \hat{C}_\eps(x)\|_{\rmF}\leq C\eps^4$.
Note that the eigenvectors of $\hat{C}_\eps(x)$ are the basis vectors $e_1,\dots, e_d$ which span the tangent plane $\cT_x\cM$.
Therefore, projecting onto the tangent plane is equivalent to projecting onto the $d$-dimensional eigenbasis of $\hat{C}_\eps(x)$.

By the Davis-Kahan theorem (e.g.~\cite[Theorem 2]{yu2015useful} we have
\[ \lda \Pi_{x,\eps} - \Pi_{\cT_x\cM} \rda_{\rmF} \leq \frac{2^{\frac32}\|C_\eps(x) - \hat{C}_\eps(x)\|_{\rmF}}{\hat{\lambda}_d - \hat{\lambda}_{d+1}} \]
where $\hat{\lambda}_d,\hat{\lambda}_{d+1}$ are the $d$, $d+1$-eigenvalues of $\hat{C}_\eps(x)$.
From the definition of $\Lambda_\eps$ we have
\[ [\Lambda_\eps]_{dd} = \frac{\eps^d \rho(0)}{\mu(\cM\cap B(0,\eps))} \int_{|z|^2+|\eps h_0(z)|^2\leq 1} z_d^2 \, \dd z \to \frac{1}{\Vol(B(0,1))} \int_{|z|^2\leq 1} z_d^2 \, \dd z. \]
This implies $\hat{\lambda}_d\gtrsim \eps^2$.
Using the above bound on $\|C_\eps(x) - \hat{C}_\eps(x)\|_{\rmF}$, the lower bound on $\hat{\lambda}_d$ and $\hat{\lambda}_{d+1}=0$  
we have
\[ \lda \Pi_{x,\eps} - \Pi_{\cT_x\cM} \rda_{\rmF} \leq C\eps^2. \qedhere \]
\end{proof}

\begin{algorithm}[ht]
\caption{Approximating Tangent Planes}
\label{alg:GD:TanPlane}
\begin{algorithmic}
\State {\bf Input:} Feature vectors $\{x_i\}_{i=1}^n$, dimension $d$ of the manifold, number $k$ of nearest neighbours.
\State {\bf Output:} Approximation of the tangent space basis $\{\hat{t}^{(n)}_j(x_i)\}_{j=1}^d$ for each $i=1,\dots,n$.
\State
\For{$i=1,\dots,n$}
\State Find the $k$ nearest neighbours of $x_i$, let $N_{x_i}$ index the set of $k$NN's.
\State Find the $d$-dimensional (centred) PCA approximation of $\{x_j\}_{j\in N_{x_i}}$, let $\{\hat{t}^{(n)}_j(x_i)\}_{j=1}^d$ be an orthogonal basis for this space.
\EndFor
\end{algorithmic}
\end{algorithm}

Theorem~\ref{thm:GD:Unknown:TngConv} controls the error of the PCA-based
projection only at the sample locations, whereas in the learned scheme~\cref{eq:GD:Manifold:GradDes,eq:def-m,eq:def-hat-n,eq:GD:Manifold:ManProj} we use an averaged projection
$\sum_{i\in\cI(z)}w_i(z)\Pi_{x_i,n,\eps}$
to approximate the tangent projection at the current iterate $z$, which need
not lie on $\cM$. 
Thus, in order to apply the consistency result to the
learned scheme, we also need to control the localisation error between the tangent spaces at nearby samples $x_i$ and the tangent space at $P_\cM(z)$.

\begin{lemma}
\label{lem:GD:Unknown:AveragedTangentConsistency}
Assume the setting of Theorem~\ref{thm:GD:Unknown:TngConv}. Let
$r_0$ be such that the projection is well defined for all $x$ with $d(x,\cM) \le r_0$ and let $\xi\mapsto\Pi_{\cT_\xi\cM}$ be Lipschitz continuous.
For any $z$, with $d(z,\cM)\leq r_0$, let $\cI(z)$ be an index set such that
$\|x_i-z\|\le \eps$ for all $i\in\cI(z)$ and let the weights satisfy
$w_i(z)\ge0$, $\sum_{i\in\cI(z)}w_i(z)=1$.
Define
\[
\widehat \Pi(z)
:=
\sum_{i\in\cI(z)}w_i(z)\Pi_{x_i,n,\eps}.
\]
Then there exists a constant $\tilde C>0$ such that on the high-probability event of
Theorem~\ref{thm:GD:Unknown:TngConv}, one has
\[
\|\widehat \Pi(z)-\Pi_{\cT_z\cM}\|_{\rmF}
\le
\tilde C\bigl(\alpha+\eps+d(z,\cM)\bigr)
\]
for all $z$ sufficiently close to the manifold.
\end{lemma}
\begin{proof}
Let
\[
y:=P_\cM(z).
\]
By the definition of the canonical extension,
\[
\Pi_{\cT_z\cM}=\Pi_{\cT_y\cM}.
\]
Using the convexity of the weights, we obtain
\begin{align*}
\|\widehat \Pi(z)-\Pi_{\cT_z\cM}\|_{\rmF}
&=
\left\|
\sum_{i\in\cI(z)}w_i(z)
\bigl(
\Pi_{x_i,n,\eps}-\Pi_{\cT_y\cM}
\bigr)
\right\|_{\rmF}
\\
&\le
\sum_{i\in\cI(z)}w_i(z)
\|\Pi_{x_i,n,\eps}-\Pi_{\cT_y\cM}\|_{\rmF}.
\end{align*}
For each $i\in\cI(z)$ we split the error as
\[
\|\Pi_{x_i,n,\eps}-\Pi_{\cT_y\cM}\|_{\rmF}
\le
\|\Pi_{x_i,n,\eps}-\Pi_{\cT_{x_i}\cM}\|_{\rmF}
+
\|\Pi_{\cT_{x_i}\cM}-\Pi_{\cT_y\cM}\|_{\rmF}.
\]
The first term is controlled by Theorem~\ref{thm:GD:Unknown:TngConv},
\[
\|\Pi_{x_i,n,\eps}-\Pi_{\cT_{x_i}\cM}\|_{\rmF}
\le
\alpha+C\eps^2.
\]
The second term is controlled by the Lipschitz continuity of the tangent
projection:
\[
\|\Pi_{\cT_{x_i}\cM}-\Pi_{\cT_y\cM}\|_{\rmF}
\le
L_{\cT}\|x_i-y\|.
\]
Moreover,
\[
\|x_i-y\|
\le
\|x_i-z\|+\|z-y\|
\le
\eps+d(z,\cM).
\]
Combining these estimates gives
\[
\|\Pi_{x_i,n,\eps}-\Pi_{\cT_y\cM}\|_{\rmF}
\le
\alpha+C\eps^2
+
L_{\cT}\bigl(\eps+d(z,\cM)\bigr).
\]
Averaging over $i\in\cI(z)$ and ignoring second-order terms in $\eps$ gives the desired estimate. 
\end{proof}

To show that the gradient flow converges we will also need consistency of the normal approximation.

\begin{lemma}
\label{lem:GD:Unknown:NormalConv}
Assume the setting of Theorem~\ref{thm:GD:Unknown:TngConv}.
Let $\xi\mapsto \Pi_{\cT_\xi\cM}$ be Lipschitz continuous, $\cI(x)$ be an index set such that $\norm{x_i - x}\le \eps$ for all $i\in\cI(x)$, and let weights $w_i(x)\ge 0$ satisfy $\sum_{i\in\cI(x)} w_i(x)=1$. 
Let $\hat n(x;z)$ and $m(x)$ be as defined in~\eqref{eq:def-hat-n} and~\eqref{eq:def-m}.
Then, there exists
$\tilde  C>0$ such that 
on the high-probability event of
Theorem~\ref{thm:GD:Unknown:TngConv} one has
\[
\|\hat n(x;z) - n(x)\|
\le \tilde C \l r+\tau + \eps \r \l r+\tau+\eps + \alpha\r 
\quad \text{for all $x,z$ such that $d(x,\cM) \leq r$ and $\norm{x-z}\leq\tau$},
\]
where $n(x)=P_\cM(x)-x$.
\end{lemma}

\begin{proof}
Let $y := P_\cM(x)$. 
Since $\sum_{i\in\cI(z)} w_i(z)=1$, we have
\[
\hat n(x;z) 
=  \sum_{i\in\cI(z)} w_i(z)\, \Pi_{x_i,n,\eps}(x-m(z)) - (x - m(z))
= -\sum_{i\in\cI(z)} w_i(z)\, \big(I - \Pi_{x_i,n,\eps}\big)(x-m(z))
\]
and therefore
\begin{align}
\hat n(x;z) - n(x)
&= -\sum_{i\in\cI(z)} w_i(z)\, \big(I - \Pi_{x_i,n,\eps}\big)(x-m(z)) - (y - x) \nonumber \\
&= -\sum_{i\in\cI(z)} w_i(z) \big(I - \Pi_{x_i,n,\eps}\big)(x-y) - \sum_{i\in\cI(z)} w_i(z) \big(I - \Pi_{x_i,n,\eps}\big)(y - m(z)) + (x - y) \nonumber \\
&= \sum_{i\in\cI(z)} w_i(z) \Pi_{x_i,n,\eps}(x-y) + \sum_{i\in\cI(z)} w_i(z) \big(I - \Pi_{x_i,n,\eps}\big)(m(z) - y). \label{eq:temp-estimate}
\end{align}

To estimate the first term, we observe that
\begin{align*}
\sum_{i\in\cI(z)} w_i(z) \Pi_{x_i,n,\eps}(x-y)
&= \sum_{i\in\cI(z)} w_i(z) \big( \Pi_{x_i,n,\eps} - \Pi_{\cT_y\cM} \big)(x-y)
+ \Pi_{\cT_y\cM}(x-y) \\
&= \sum_{i\in\cI(z)} w_i(z) \big( \Pi_{x_i,n,\eps} - \Pi_{\cT_y\cM} \big)(x-y)
\end{align*}
since $x-y$ is orthogonal to $\cT_y\cM$. 
By the triangle inequality we have
\begin{align*}
\left\| \sum_{i\in\cI(z)} w_i(z) \Pi_{x_i,n,\eps}(x-y) \right\|
&\le \sum_{i\in\cI(z)} w_i(z) \| \Pi_{x_i,n,\eps} - \Pi_{\cT_y\cM} \|_{\op} \, \|x-y\|.
\end{align*}
Further, we have
\[
\| \Pi_{x_i,n,\eps} - \Pi_{\cT_y\cM} \|_{\op}
\le \|\Pi_{x_i,n,\eps} - \Pi_{\cT_{x_i}\cM}\|_{\op}
+ \|\Pi_{\cT_{x_i}\cM} - \Pi_{\cT_y\cM}\|_{\op}.
\]
By Theorem~\ref{thm:GD:Unknown:TngConv}, and because the operator norm is upper bounded by the Frobenius norm, we have
\[
\|\Pi_{x_i,n,\eps} - \Pi_{\cT_{x_i}\cM}\|_{\op}
\le \alpha + C^\prime\eps^2,
\]
with high probability as in
Theorem~\ref{thm:GD:Unknown:TngConv}.

Using that $\xi\mapsto \Pi_{\cT_\xi\cM}$ is Lipschitz we have
\[
\|\Pi_{\cT_{x_i}\cM} - \Pi_{\cT_y\cM}\|_{\op} \le L_{\cT} \norm{x_i - y} \le  L_{\cT} (\eps + \norm{x-y}).
\]
Hence, ignoring the second-order term in $\eps$, we get
\[
\| \Pi_{x_i,n,\eps} - \Pi_{\cT_y\cM} \|_{\op}
\le C^{\prime\prime}\left(\eps + \norm{x-y} + \alpha\right)
\]
and, since $\norm{x-y} = d(x,\cM) \leq r$,
\begin{align}
\left\| \sum_{i\in\cI(z)} w_i(z) \Pi_{x_i,n,\eps}(x-y) \right\|
\le C^{\prime\prime}r\left(\eps + r + \alpha\right). \label{eq:est-term-1}
\end{align}

To estimate the second term in~\eqref{eq:temp-estimate}, let
\[
q := \eps + \tau + r.
\]
First note that
\[
\|m(z)-y\|
\le \|m(z)-z\| + \|z-x\| + \|x-y\|
\le \eps + \tau + r
= q.
\]

We now write
\begin{align*}
\sum_{i\in\cI(z)} w_i(z)\bigl(I-\Pi_{x_i,n,\eps}\bigr)(m(z)-y)
&=
(I-\Pi_{\cT_y\cM})(m(z)-y) \\
&\qquad + \sum_{i\in\cI(z)} w_i(z)\bigl(\Pi_{\cT_y\cM}-\Pi_{x_i,n,\eps}\bigr)(m(z)-y).
\end{align*}
Hence
\begin{align}
\left\|\sum_{i\in\cI(z)} w_i(z)\bigl(I-\Pi_{x_i,n,\eps}\bigr)(m(z)-y)\right\|
&\le \|(I-\Pi_{\cT_y\cM})(m(z)-y)\| \nonumber\\
&\qquad + \sum_{i\in\cI(z)} w_i(z)\|\Pi_{x_i,n,\eps}-\Pi_{\cT_y\cM}\|_{\op}\,\|m(z)-y\|. \label{eq:split-second-term}
\end{align}

To estimate the first term on the right-hand side, choose local coordinates at $y$ so that $\cM$ is given as the graph
\[
u \mapsto y + \binom{u}{h(u)},
\qquad h(0)=0,\quad Dh(0)=0,\quad \|h(u)\|\le C^{\prime\prime\prime}\|u\|^2
\]
for $\|u\|$ sufficiently small. Every point used in the average defining $m(z)$ lies within distance at most $q$ of $y$, because
\[
\|x_j-y\|
\le \|x_j-z\| + \|z-x\| + \|x-y\|
\le \eps + \tau + r
= q.
\]
Writing
\[
x_j-y=\binom{u_j}{h(u_j)},
\]
we obtain
\[
m(z)-y
=
\binom{\sum_{j\in\cI(z)} w_j(z)u_j}{\sum_{j\in\cI(z)} w_j(z)h(u_j)}.
\]
Therefore
\[
\|(I-\Pi_{\cT_y\cM})(m(z)-y)\|
=
\left\|\sum_{j\in\cI(z)} w_j(z)h(u_j)\right\|
\le C^{\prime\prime\prime}\sum_{j\in\cI(z)} w_j(z)\|u_j\|^2
\le C^{\prime\prime\prime} q^2.
\]

For the second term in~\eqref{eq:split-second-term}, we use the estimate already established above:
\[
\|\Pi_{x_i,n,\eps} - \Pi_{\cT_y\cM}\|
\le C^{\prime\prime}\bigl(\eps + \|x-y\| + \alpha\bigr)
\le C^{\prime\prime}(\eps + r + \alpha).
\]
Combining this with $\|m(z)-y\|\le q$ gives
\[
\sum_{i\in\cI(z)} w_i(z)\|\Pi_{x_i,n,\eps}-\Pi_{\cT_y\cM}\|\,\|m(z)-y\|
\le C^{\prime\prime}(\eps+r+\alpha)\,q.
\]

Substituting into~\eqref{eq:split-second-term}, we conclude that
\begin{align}
\left\|\sum_{i\in\cI(z)} w_i(z)\bigl(I-\Pi_{x_i,n,\eps}\bigr)(m(z)-y)\right\|
&\le C^{\prime\prime\prime} q^2 + C^{\prime\prime}(\eps+r+\alpha)\,q. \label{eq:est-term-2-sharp}
\end{align}

Combining~\eqref{eq:est-term-1} and~\eqref{eq:est-term-2-sharp}, we obtain
\[
\|\hat n(x;z)-n(x)\|
\le C_1 r(\eps+r+\alpha)
   + C_2 q^2
   + C_3(\eps+r+\alpha)\,q.
\]
Since $q=\eps+\tau+r$ we can simplify the bound to $\tilde{C}q(q+\alpha)$ for some $\tilde{C}>0$.
\end{proof}

Later we will need to know that the distance of the scheme from the manifold can be controlled.
The following one-step bound allows us to do this.

\begin{lemma}
\label{lem:GD:Unknown:DiscreteDistanceControl}
Assume the setting of Theorem~\ref{thm:GD:Unknown:TngConv} and
Lemma~\ref{lem:GD:Unknown:NormalConv}. Let
\[
M := \sup_{\{x:\, d(x,\cM)\le r_0\}} \|\nabla \cE(x)\| <+\infty
\]
for some radius $r_0>0$ for which the nearest-point projection $P_\cM$ is
well-defined. Let the learned gradient descent scheme be given by\cref{eq:GD:Manifold:GradDes,eq:def-m,eq:def-hat-n,eq:GD:Manifold:ManProj} 
with the weights satisfying  $w_i^k\ge 0$, $\sum_{i\in\cI^k}w_i^k=1$.
Set
\[
\eta := \tau_n\lambda .
\]
Assume that $0\le \eta\le 1$ and that
\[
d(\xi^k,\cM)+\tau_t M \le r_0 .
\]
Then, on the high-probability event from
Theorem~\ref{thm:GD:Unknown:TngConv}, one has
\begin{equation}
\label{eq:learned-distance-recursion-tilde}
d(\xi^{k+1},\cM)
\le
(1-\eta)d(\tilde \xi^{k+1},\cM)
+
\eta \tilde C
\bigl(d(\tilde \xi^{k+1},\cM)+s_k+\eps\bigr)
\bigl(d(\tilde \xi^{k+1},\cM)+s_k+\eps+\alpha\bigr),
\end{equation}
where
\[
s_k := \|\tilde \xi^{k+1}-\xi^k\|
\le \tau_t M .
\]
In particular,
\begin{equation}
\label{eq:learned-distance-recursion}
d(\xi^{k+1},\cM)
\le
(1-\eta)\bigl(d(\xi^k,\cM)+\tau_t M\bigr)
+
\eta \tilde C
\bigl(d(\xi^k,\cM)+2\tau_t M+\eps\bigr)
\bigl(d(\xi^k,\cM)+2\tau_t M+\eps+\alpha\bigr).
\end{equation}
\end{lemma}

\begin{proof}
We first record a simple bound on the length of the approximate tangent step.
Since each $\Pi_{x_i,n,\eps}$ is an orthogonal projection and the weights are
non-negative with unit sum, we have
\begin{align*}
s_k
&=
\left\|
\tau_t
\sum_{i\in\cI^k} w_i^k \Pi_{x_i,n,\eps}\nabla \cE(\xi^k)
\right\|
\\
&\le
\tau_t
\sum_{i\in\cI^k} w_i^k
\|\Pi_{x_i,n,\eps}\nabla \cE(\xi^k)\|
\\
&\le
\tau_t \|\nabla \cE(\xi^k)\|
\le
\tau_t M .
\end{align*}
Consequently,
\[
d(\tilde \xi^{k+1},\cM)
\le
d(\xi^k,\cM)+\|\tilde \xi^{k+1}-\xi^k\|
\le
d(\xi^k,\cM)+\tau_t M
\le r_0,
\]
so that $P_\cM(\tilde \xi^{k+1})$ is well-defined.

Let
\[
y^{k+1}:=P_\cM(\tilde \xi^{k+1}),
\qquad
n(\tilde \xi^{k+1})
=
y^{k+1}-\tilde \xi^{k+1}.
\]
Using the normal update~\eqref{eq:GD:Manifold:ManProj},
we write
\begin{align*}
\xi^{k+1}
&=
\tilde \xi^{k+1}
+
\eta \hat n(\tilde \xi^{k+1};\xi^k)
\\
&=
\tilde \xi^{k+1}
+
\eta n(\tilde \xi^{k+1})
+
\eta\bigl(
\hat n(\tilde \xi^{k+1};\xi^k)
-
n(\tilde \xi^{k+1})
\bigr)
\\
&=
(1-\eta)\tilde \xi^{k+1}
+
\eta y^{k+1}
+
\eta\bigl(
\hat n(\tilde \xi^{k+1};\xi^k)
-
n(\tilde \xi^{k+1})
\bigr).
\end{align*}
Therefore,
\begin{align*}
\xi^{k+1}-y^{k+1}
&=
(1-\eta)(\tilde \xi^{k+1}-y^{k+1})
+
\eta\bigl(
\hat n(\tilde \xi^{k+1};\xi^k)
-
n(\tilde \xi^{k+1})
\bigr).
\end{align*}
Since $y^{k+1}\in\cM$, we have
\[
d(\xi^{k+1},\cM)
\le
\|\xi^{k+1}-y^{k+1}\|.
\]
Thus, using $0\le \eta\le 1$,
\begin{align}
d(\xi^{k+1},\cM)
&\le
(1-\eta)
\|\tilde \xi^{k+1}-y^{k+1}\|
+
\eta
\|\hat n(\tilde \xi^{k+1};\xi^k)
-
n(\tilde \xi^{k+1})\|
\nonumber
\\
&=
(1-\eta)d(\tilde \xi^{k+1},\cM)
+
\eta
\|\hat n(\tilde \xi^{k+1};\xi^k)
-
n(\tilde \xi^{k+1})\|.
\label{eq:pre-normal-error-bound}
\end{align}

We now apply Lemma~\ref{lem:GD:Unknown:NormalConv} with
\[
x=\tilde \xi^{k+1},
\qquad
z=\xi^k,
\qquad
r=d(\tilde \xi^{k+1},\cM),
\qquad
\tau=s_k=\|\tilde \xi^{k+1}-\xi^k\|.
\]
This gives
\[
\|\hat n(\tilde \xi^{k+1};\xi^k)-n(\tilde \xi^{k+1})\|
\le
\tilde C
\bigl(d(\tilde \xi^{k+1},\cM)+s_k+\eps\bigr)
\bigl(d(\tilde \xi^{k+1},\cM)+s_k+\eps+\alpha\bigr).
\]
Substituting this estimate into~\eqref{eq:pre-normal-error-bound} proves
\eqref{eq:learned-distance-recursion-tilde}.

Finally,~\eqref{eq:learned-distance-recursion} follows immediately from~\eqref{eq:learned-distance-recursion-tilde}, $d(\tilde{\xi}^{k+1},\cM)\leq d(\xi^{k},\cM)+\tau_tM$ and $s_k\leq \tau_t M$.
\end{proof}

\begin{remark}
If $\eta = \tau_n \lambda = 1$ then
\[
d(\xi^{k+1},\cM)
\le
\tilde C
\bigl(d(\xi^k,\cM)+2\tau_t M+\eps\bigr)
\bigl(d(\xi^k,\cM)+2\tau_t M+\eps+\alpha\bigr).
\]
\end{remark}

We now state two convergence theorems.
The first gives the convergence of the scheme, i.e. that on a fixed interval $[0,T]$ the solution to the discrete scheme (after an appropriate interpolation) converges to the solution of the time-continuous scheme. 

\begin{theorem}
\label{thm:GD:Unknown:ApproxFlowConv}
Let
\[
U_{r_0}:=\{x\in\R^D:\ d(x,\cM)\le r_0\}
\]
be a tubular neighbourhood of $\cM$ on which the nearest-point projection
$P_\cM$ is well-defined. Assume that the exact vector field
\[
F(x):=-\Pi_{\cT_x\cM}\nabla \cE(x)+\lambda n(x),
\qquad
n(x):=P_\cM(x)-x,
\]
is well-defined on $U_{r_0}$, Lipschitz continuous with Lipschitz
constant $L_F$, and uniformly bounded by
\[
\|F(x)\|\le B
\qquad\text{for all }x\in U_{r_0}.
\]
Assume also that $n$ is Lipschitz on $U_{r_0}$ with Lipschitz constant $L_n$,
and that
\[
M:=\sup_{x\in U_{r_0}}\|\nabla \cE(x)\|<+\infty .
\]

Let $\xi$ denote the exact flow 
\begin{equation}
\label{eq:exact-flow-for-approx-conv}
\dot \xi(t)=F(\xi(t)),
\qquad
\xi(0)=\xi^0,
\end{equation}
and assume that, for some $r_1<r_0$ and some fixed $T>0$,
\[
\xi(t)\in U_{r_1}
\qquad\text{for all }t\in[0,T].
\]

Let $\{\xi^k\}_{k\ge0}$ be generated by the fully discrete learned scheme defined by~\eqref{eq:GD:Manifold:GradDes} and~\eqref{eq:GD:Manifold:ManProj} with $\tau_t=\tau_n=\tau$. 
Assume that 
the following consistency estimates
hold whenever the relevant points lie in $U_{r_0}$:
\begin{align}
\left\|
\sum_{i\in\cI^k}w_i^k\Pi_{x_i,n,\eps}
-
\Pi_{\cT_{\xi^k}\cM}
\right\|_{\op}
&\le \beta,
\label{eq:tangent-consistency-beta}
\\
\|\hat n(x;z)-n(x)\|
&\le \gamma
\qquad
\text{whenever }x,z\in U_{r_0},\ \|x-z\|\le \tau M.
\label{eq:normal-consistency-gamma}
\end{align}
Then, provided $\tau$, $\beta$ and $\gamma$ are sufficiently small, the
piecewise affine interpolation
\[
\xi_\tau(t)
:=
\xi^k+\frac{t-t_k}{\tau}(\xi^{k+1}-\xi^k),
\qquad
t\in[t_k,t_{k+1}],
\qquad
t_k:=k\tau,
\]
is well-defined in $U_{r_0}$ on $[0,T]$ and satisfies
\begin{equation}
\label{eq:approx-flow-error-bound}
\sup_{t\in[0,T]}\|\xi_\tau(t)-\xi(t)\|
\le
C_T\bigl(\tau+\beta+\gamma\bigr),
\end{equation}
where $C_T>0$ depends only on $T$, $L_F$, $L_n$, $M$, $B$, $\lambda$ and
$r_0$, but is independent of $\tau$, $\beta$ and $\gamma$.

In particular, if along a sequence of point clouds and step sizes one has
\[
\tau\to0,
\qquad
\beta\to0,
\qquad
\gamma\to0,
\]
then
\[
\xi_\tau\to \xi
\qquad
\text{uniformly on }[0,T].
\]
\end{theorem}

\begin{proof}
We write
\[
\widehat \Pi^k
:=
\sum_{i\in\cI^k}w_i^k\Pi_{x_i,n,\eps}.
\]
Then the learned scheme can be written as
\[
\tilde \xi^{k+1}
=
\xi^k-\tau \widehat \Pi^k\nabla \cE(\xi^k),
\]
and
\[
\xi^{k+1}
=
\xi^k
-\tau \widehat \Pi^k\nabla \cE(\xi^k)
+\tau\lambda \hat n(\tilde \xi^{k+1};\xi^k).
\]
Hence
\begin{equation}
\label{eq:learned-step-as-euler}
\xi^{k+1}
=
\xi^k
+
\tau F(\xi^k)
+
\tau R^k,
\end{equation}
where
\begin{align*}
R^k
&:=
-\bigl(\widehat \Pi^k-\Pi_{\cT_{\xi^k}\cM}\bigr)\nabla \cE(\xi^k)
+\lambda\bigl(\hat n(\tilde \xi^{k+1};\xi^k)-n(\xi^k)\bigr).
\end{align*}

We first estimate the residual $R^k$. By
\eqref{eq:tangent-consistency-beta},
\[
\left\|
\bigl(\widehat \Pi^k-\Pi_{\cT_{\xi^k}\cM}\bigr)
\nabla \cE(\xi^k)
\right\|
\le
\beta M .
\]
Moreover,
\[
\|\tilde \xi^{k+1}-\xi^k\|
=
\tau \|\widehat \Pi^k\nabla \cE(\xi^k)\|
\le
\tau M,
\]
since $\widehat \Pi^k$ is a convex combination of orthogonal projections.
Therefore \eqref{eq:normal-consistency-gamma} applies with
$x=\tilde \xi^{k+1}$ and $z=\xi^k$, and we get
\begin{align*}
\|\hat n(\tilde \xi^{k+1};\xi^k)-n(\xi^k)\|
&\le
\|\hat n(\tilde \xi^{k+1};\xi^k)-n(\tilde \xi^{k+1})\|
+
\|n(\tilde \xi^{k+1})-n(\xi^k)\|
\\
&\le
\gamma
+
L_n\|\tilde \xi^{k+1}-\xi^k\|
\\
&\le
\gamma+L_nM\tau .
\end{align*}
Consequently
\begin{equation}
\label{eq:residual-bound}
\|R^k\|
\le
M\beta+\lambda\gamma+\lambda L_nM\tau
=: \delta_\tau .
\end{equation}

Let $K_\tau$ be such that $K_\tau\tau\le T<(K_\tau+1)\tau$.
We compare the numerical iterates with the exact solution at the grid points.
Set
\[
e^k:=\xi^k-\xi(t_k).
\]
By Taylor's formula for the exact flow,
\begin{equation} \label{eq:GD:Unknown:xiTaylor}
\xi(t_{k+1})
=
\xi(t_k)
+
\tau F(\xi(t_k))
+
\rho^k,
\end{equation}
where
\[
\rho^k
=
\int_{t_k}^{t_{k+1}}
\bigl(F(\xi(s))-F(\xi(t_k))\bigr)\,\dd s.
\]
Since $F$ is Lipschitz and $\|\dot \xi(s)\|=\|F(\xi(s))\|\le B$, we have
\[
\|\rho^k\|
\le
\int_{t_k}^{t_{k+1}}
L_F\|\xi(s)-\xi(t_k)\|\,\dd s
\le
\int_{t_k}^{t_{k+1}}
L_F B(s-t_k)\,\dd s
=
\frac12 L_F B\tau^2.
\]
Using~\eqref{eq:GD:Unknown:xiTaylor} and~\eqref{eq:learned-step-as-euler} gives
\[
e^{k+1}
=
e^k
+
\tau\bigl(F(\xi^k)-F(\xi(t_k))\bigr)
+
\tau R^k
-
\rho^k.
\]
Using the Lipschitz continuity of $F$ and \eqref{eq:residual-bound}, we obtain
\begin{align}
\|e^{k+1}\|
&\le
(1+\tau L_F)\|e^k\|
+
\tau\delta_\tau
+
\frac12 L_FB\tau^2 .
\label{eq:error-recursion-grid}
\end{align}
Since $e^0=0$,
\begin{align*}
\|e^{k+1}\| & \leq \l \delta_\tau \tau + \frac12 L_F B \tau^2 \r \sum_{i=0}^k \l 1+\tau L_F\r ^i + \l 1+\tau L_F\r^{k+1} \|e^0\| \\
 & = \frac{\l \delta_\tau \tau + \frac12 L_F B \tau^2 \r \l \l 1+\tau L_F\r^{k+1} - 1\r}{\tau L_F} \\
 & \leq \frac{\l \delta_\tau \tau + \frac12 L_F B \tau^2 \r \l e^{L_F\tau(k+1)} - 1\r}{\tau L_F}.
\end{align*}
So then,
\[
\max_{0\le k\le K_\tau+1}\|e^k\|
\le
\frac{e^{L_F(T+\tau)}-1}{L_F}
\left(
\delta_\tau+\frac12 L_FB\tau
\right),
\]
with the usual interpretation that the prefactor is $T+\tau$ when $L_F=0$.
Using the definition of $\delta_\tau$, we get
\begin{equation}
\label{eq:grid-error-bound}
\max_{0\le k\le K_\tau+1}\|\xi^k-\xi(t_k)\|
\le
C_T(\tau+\beta+\gamma).
\end{equation}

It remains to pass from the grid-point estimate to the piecewise affine
interpolation. For $t\in[t_k,t_{k+1}]$,
\begin{align*}
\|\xi_\tau(t)-\xi(t)\|
&\le
\|\xi_\tau(t)-\xi^k\|
+
\|\xi^k-\xi(t_k)\|
+
\|\xi(t_k)-\xi(t)\|.
\end{align*}
The last term is bounded by
\[
\|\xi(t)-\xi(t_k)\|\le B\tau.
\]
For the first term, using the learned update,
\begin{align*}
\|\xi_\tau(t)-\xi^k\|
&\le
\|\xi^{k+1}-\xi^k\|
\\
&\le
\tau
\|\widehat \Pi^k\nabla \cE(\xi^k)\|
+
\tau\lambda\|\hat n(\tilde \xi^{k+1};\xi^k)\|.
\end{align*}
Moreover,
\begin{align*}
\|\hat n(\tilde \xi^{k+1};\xi^k)\|
&\le
\|n(\tilde \xi^{k+1})\|
+
\|\hat n(\tilde \xi^{k+1};\xi^k)-n(\tilde \xi^{k+1})\|
\\
&\le
r_0+\gamma,
\end{align*}
as long as $\tilde \xi^{k+1}\in U_{r_0}$. Thus
\[
\|\xi_\tau(t)-\xi^k\|
\le
\tau M+\tau\lambda(r_0+\gamma)
\le
C\tau
\]
for $\gamma$ bounded. Combining this with \eqref{eq:grid-error-bound} gives
\[
\sup_{t\in[0,T]}\|\xi_\tau(t)-\xi(t)\|
\le
C_T(\tau+\beta+\gamma).
\]

Finally, since the exact solution lies in $U_{r_1}$ with $r_1<r_0$, estimate
\eqref{eq:approx-flow-error-bound} implies that, for $\tau$, $\beta$ and
$\gamma$ sufficiently small,
\[
d(\xi_\tau(t),\cM)\le r_0
\qquad
\text{for all }t\in[0,T].
\]
Thus the above estimates are self-consistent, and the learned flow remains in
the tubular neighbourhood on $[0,T]$. This proves the theorem.
\end{proof}

The second convergence theorem establishes that the discrete scheme will find the local minimiser.

\begin{theorem}
\label{thm:GD:Unknown:FullyDiscreteConvergence}
Let $\xi^*\in\cM$ be a local minimiser of $\cE$ on $\cM$ and assume that
$\xi^*$ is a locally exponentially stable equilibrium of the exact projected
flow
\begin{equation}
\label{eq:exact-projected-flow}
\dot \xi
=
F(\xi)
:=
-\Pi_{\cT_\xi\cM}\nabla \cE(\xi)
+
\lambda n(\xi),
\qquad
n(\xi):=P_\cM(\xi)-\xi .
\end{equation}
More precisely, assume that there exists $R>0$, $\mu>0$ and $L_F>0$ such that
$B(\xi^*,R)$ is contained in a tubular neighbourhood of $\cM$ on which
$P_\cM$ is well-defined, $F$ is Lipschitz continuous with Lipschitz constant
$L_F$, $F(\xi^*)=0$, and
\begin{equation}
\label{eq:local-dissipativity}
\langle \xi-\xi^*,F(\xi)\rangle
\le
-\mu \|\xi-\xi^*\|^2
\qquad
\text{for all } \xi\in B(\xi^*,R).
\end{equation}
Assume also that $n$ is Lipschitz continuous on $B(\xi^*,R)$ with Lipschitz
constant $L_n$, and set
\[
M:=\sup_{\xi\in B(\xi^*,R)}\|\nabla \cE(\xi)\|<+\infty .
\]

Let $\{\xi^k\}_{k\ge0}$ be generated by the fully discrete learned scheme defined by~\eqref{eq:GD:Manifold:GradDes} ~\eqref{eq:GD:Manifold:ManProj}
where $\tau_t=\tau_n$, $w_i^k\ge0$ and $\sum_{i\in\cI^k}w_i^k=1$. Define
\[
\widehat \Pi^k
:=
\sum_{i\in\cI^k} w_i^k\Pi_{x_i,n,\eps}.
\]
Assume that 
the following consistency estimates
hold whenever the relevant points lie in $B(\xi^*,R)$:
\begin{align}
\|\widehat \Pi^k-\Pi_{\cT_{\xi^k}\cM}\|_{\op}
&\le \beta ,
\label{eq:fully-discrete-tangent-error}
\\
\|\hat n(\tilde \xi^{k+1};\xi^k)-n(\tilde \xi^{k+1})\|
&\le \gamma .
\label{eq:fully-discrete-normal-error}
\end{align}
Suppose that the step size satisfies
\begin{equation}
\label{eq:fully-discrete-stepsize}
0<\tau\le \min\left\{\frac{1}{\mu},\frac{\mu}{L_F^2}\right\},
\end{equation}
and define
\begin{equation}
\label{eq:fully-discrete-delta}
\delta_\tau
:=
M\beta+\lambda\gamma+\lambda L_nM\tau .
\end{equation}
If
\begin{equation}
\label{eq:fully-discrete-initial-condition}
\|\xi^0-\xi^*\|
+
\frac{2\delta_\tau}{\mu}
\le R,
\end{equation}
then the iterates remain in $B(\xi^*,R)$ and satisfy
\begin{equation}
\label{eq:fully-discrete-convergence-rate}
\|\xi^k-\xi^*\|
\le
\left(1-\frac{\mu\tau}{2}\right)^k
\|\xi^0-\xi^*\|
+
\frac{2\delta_\tau}{\mu}
\qquad
\text{for all } k\ge0.
\end{equation}
Consequently,
\begin{equation}
\label{eq:fully-discrete-limsup}
\limsup_{k\to\infty}\|\xi^k-\xi^*\|
\le
\frac{2}{\mu}
\left(
M\beta+\lambda\gamma+\lambda L_nM\tau
\right).
\end{equation}
In particular, since $\xi^*\in\cM$, one also has
\begin{equation}
\label{eq:fully-discrete-distance-to-manifold}
\limsup_{k\to\infty}d(\xi^k,\cM)
\le
\frac{2}{\mu}
\left(
M\beta+\lambda\gamma+\lambda L_nM\tau
\right).
\end{equation}
Thus, if 
\[
\tau\to0,
\qquad
\beta\to0,
\qquad
\gamma\to0,
\]
then the fully discrete learned scheme converges to $\xi^*$ in the sense that,
whenever $k\tau\to+\infty$,
\[
\xi^k\to \xi^* .
\]
\end{theorem}

\begin{proof}
From the proof of Theorem~\ref{thm:GD:Unknown:ApproxFlowConv} we have
\begin{equation} \label{eq:GD:Unknown:ekRec}
e^{k+1} = e^k + \tau F(\xi^k) + \tau R^k.
\end{equation}
where $\|R^k\|\leq \delta_\tau$ and
\[ e^k:=\xi^k-\xi^*. \]

Using
\eqref{eq:local-dissipativity}, the Lipschitz continuity of $F$ and $F(\xi^*)=0$, we obtain
\begin{align*}
\|e^k+\tau F(\xi^k)\|^2
&=
\|e^k\|^2
+
2\tau\langle e^k,F(\xi^k)\rangle
+
\tau^2\|F(\xi^k)\|^2
\\
&\le
\|e^k\|^2
-
2\mu\tau\|e^k\|^2
+
\tau^2 L_F^2\|e^k\|^2
\\
&=
\bigl(1-2\mu\tau+L_F^2\tau^2\bigr)\|e^k\|^2 .
\end{align*}
By the step-size condition~\eqref{eq:fully-discrete-stepsize},
\[
1-2\mu\tau+L_F^2\tau^2
\le
1-\mu\tau .
\]
Therefore,
\[
\|e^k+\tau F(\xi^k)\|
\le
\sqrt{1-\mu\tau}\,\|e^k\|
\le
\left(1-\frac{\mu\tau}{2}\right)\|e^k\|.
\]
Combining this with~\eqref{eq:GD:Unknown:ekRec}, we get
\begin{equation}
\label{eq:error-recursion-fully-discrete}
\|e^{k+1}\|
\le
\left(1-\frac{\mu\tau}{2}\right)\|e^k\|
+
\tau\delta_\tau .
\end{equation}

Iterating the recursion yields
\begin{align*}
\|e^k\|
&\le
\left(1-\frac{\mu\tau}{2}\right)^k\|e^0\|
+
\tau\delta_\tau
\sum_{\ell=0}^{k-1}
\left(1-\frac{\mu\tau}{2}\right)^\ell
\\
&\le
\left(1-\frac{\mu\tau}{2}\right)^k\|e^0\|
+
\tau\delta_\tau
\frac{1}{\mu\tau/2}
\\
&=
\left(1-\frac{\mu\tau}{2}\right)^k\|\xi^0-\xi^*\|
+
\frac{2\delta_\tau}{\mu}.
\end{align*}
This proves~\eqref{eq:fully-discrete-convergence-rate}.

The initial condition~\eqref{eq:fully-discrete-initial-condition} implies
\[
\|e^k\|\le R
\qquad
\text{for all }k\ge0,
\]
so the iterates remain in $B(\xi^*,R)$. This justifies the use of the
local assumptions throughout the argument. Taking the upper limit as
$k\to\infty$ in~\eqref{eq:fully-discrete-convergence-rate} gives
\eqref{eq:fully-discrete-limsup}. Finally, since $\xi^*\in\cM$,
\[
d(\xi^k,\cM)
\le
\|\xi^k-\xi^*\|,
\]
and hence~\eqref{eq:fully-discrete-distance-to-manifold} follows.

If $\tau\to0$, $\beta\to0$ and $\gamma\to0$, then
$\delta_\tau\to0$. Therefore, whenever $k\tau\to+\infty$,
\[
\left(1-\frac{\mu\tau}{2}\right)^k\|\xi^0-\xi^*\|
\to0
\]
and the residual term $2\delta_\tau/\mu$ also tends to zero. Hence
$\xi^k\to\xi^*$, completing the proof.
\end{proof}

The convergence of the discrete scheme with learned manifold is now an application of either Theorem~\ref{thm:GD:Unknown:ApproxFlowConv} or Theorem~\ref{thm:GD:Unknown:FullyDiscreteConvergence} with Theorem~\ref{thm:GD:Unknown:TngConv} and Lemma~\ref{lem:GD:Unknown:NormalConv}.
We summarise in the final remark of the section.

\begin{remark}
The constants $\beta$ and $\gamma$ from Theorem~\ref{thm:GD:Unknown:ApproxFlowConv} quantify, respectively, the tangent space and the normal vector estimation errors. 
From Lemma~\ref{lem:GD:Unknown:AveragedTangentConsistency}, we obtain, for iterates
satisfying $d(\xi^k,\cM)\le r$,
\[
\beta
\lesssim
\alpha+\eps+r.
\]
Similarly, Lemma~\ref{lem:GD:Unknown:NormalConv} yields
\[
\gamma
\lesssim
(r+\tau M+\eps)(r+\tau M+\eps+\alpha).
\]
Therefore the asymptotic error in
\eqref{eq:fully-discrete-limsup} vanishes provided
\[
\tau\to0,
\qquad
\eps\to0,
\qquad
\alpha\to0,
\qquad
r\to0.
\]
For instance, if $\eps=\eps_n$ and $\alpha=\alpha_n$ are chosen
so that
\[
\sum_{n=1}^\infty n\exp(-c n\alpha_n\eps_n^d)<\infty,
\]
then 
\[ \sup_{t\in [0,T]} \|\xi_\tau(t) - \xi(t)\| \leq C(\alpha+\eps+\tau+r) \]
holds eventually almost
surely by the Borel--Cantelli lemma.
Theorem~\ref{thm:GD:Unknown:FullyDiscreteConvergence} can be similarly combined with Lemma~\ref{lem:GD:Unknown:AveragedTangentConsistency} and Lemma~\ref{lem:GD:Unknown:NormalConv}.
\end{remark}

\section{Synthetic Two Dimensional Manifold in \texorpdfstring{$\bbR^3$}{R3}} \label{sec:SynEx}

To illustrate the performance of Algorithm~\ref{alg:GD:Manifold}, we consider a simple example of a two-dimensional manifold in 3D. As our manifold $\cM$ we take a rotationally symmetric surface given by the following relation in spherical coordinates
\begin{align*}
    r(\theta,\phi) = 1 + 0.1\sin(4\theta), \quad \theta \in [0,\pi], \; \phi \in [0,2\pi].
\end{align*}

We sample it using 4000 random values of $(\theta,\phi)$ to obtain a point cloud, see Figure~\ref{fig:3D}. As the objective we take the squared Euclidean distance to the magenta point in the bottom right corner of the figure corresponding to $\theta=3\pi/4$, $\phi=11\pi/6$, and $r=1$ (the point lies on the manifold)
\begin{equation*}
    \cE(x) \defeq \frac12 \norm{x-x_*}^2, \quad x_* = \left[\frac{\sqrt{6}}{4}, \, -\frac{\sqrt{2}}{4}, \, -\frac{\sqrt{2}}{2} \right].
\end{equation*}

\begin{figure}[t!]
    \captionsetup[subfigure]{justification=centering}	
    \centering
    \begin{subfigure}[b]{0.48\textwidth}
        \centering
        \includegraphics[width=\textwidth]{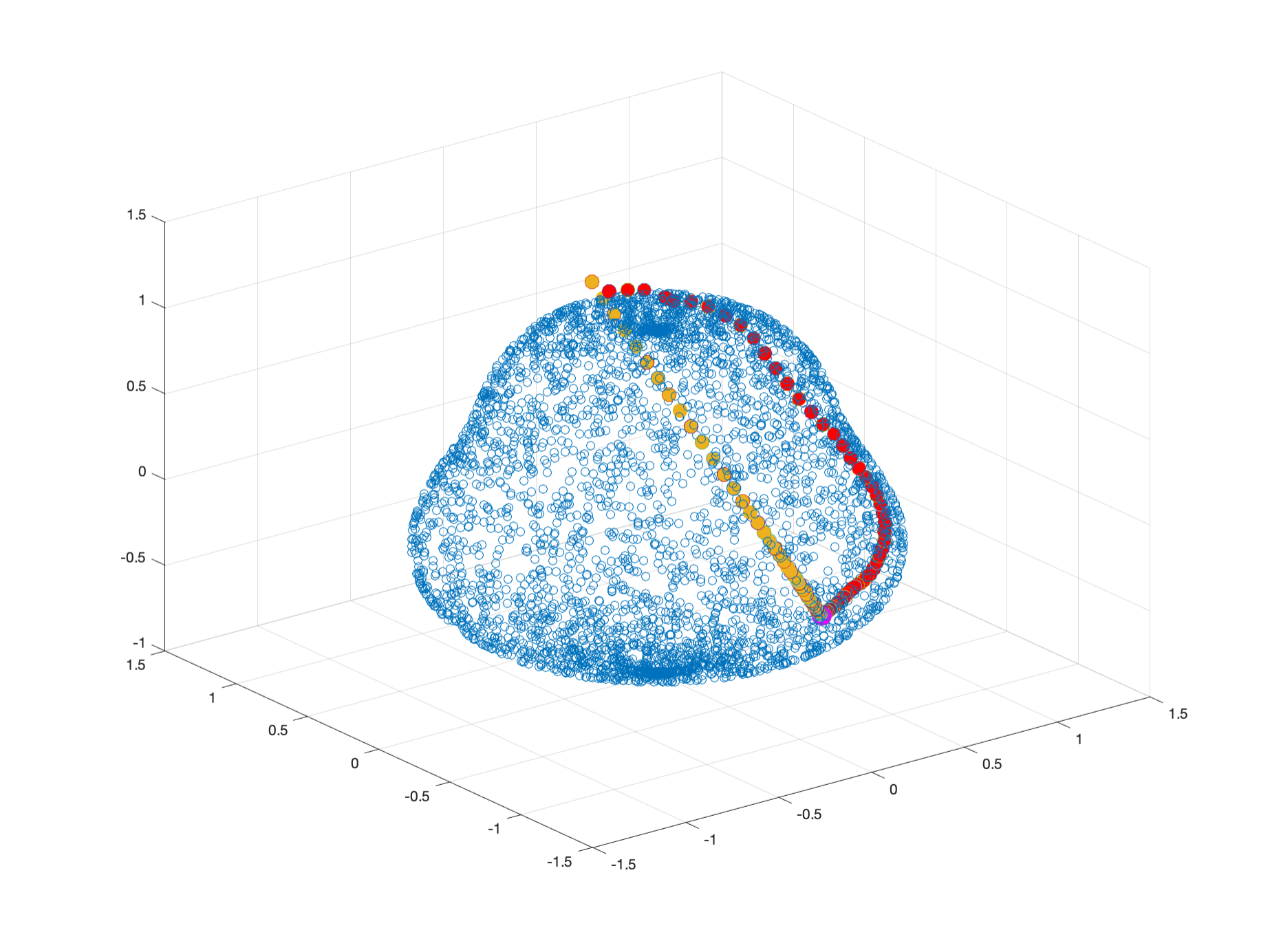}
        \caption{Correct gradient (4000 points) \\ Euclidean GD (yellow): 56 iterations \\ Point-cloud GD (red): 69 iterations}
        \label{fig:3D-clean-4000}
    \end{subfigure}
    \begin{subfigure}[b]{0.48\textwidth}
        \centering
        \includegraphics[width=\textwidth]{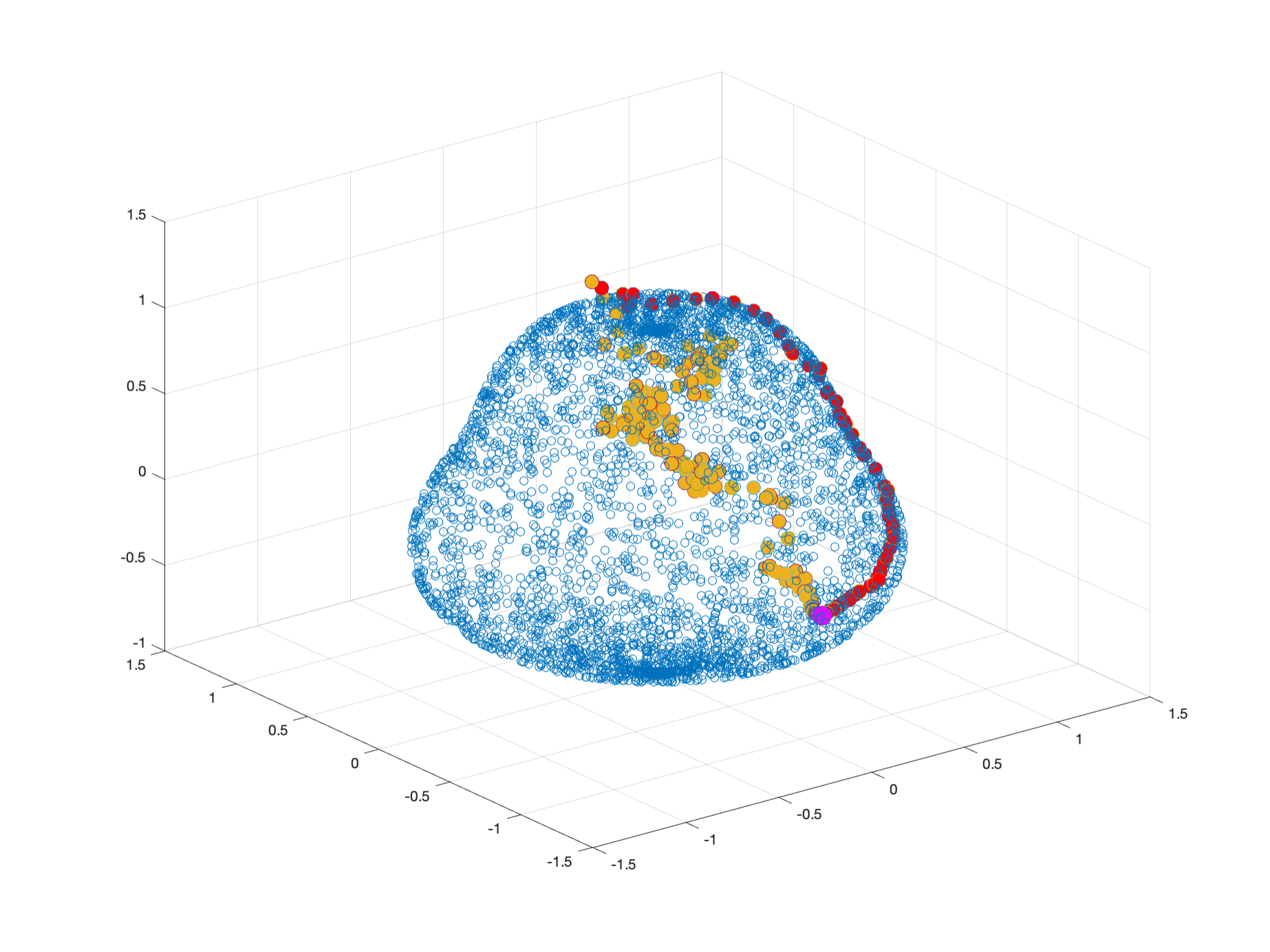}
        \caption{Noisy gradient (4000 points) \\ Euclidean GD (yellow): 151 iterations \\ Point-cloud GD (red): 74 iterations}
        \label{fig:3D-noise-4000}
    \end{subfigure}
    \caption{Optimisation on a two-dimensional manifold in 3D: Euclidean gradient descent vs point-cloud gradient descent with exact~(\subref{fig:3D-clean-4000}) and noisy~(\subref{fig:3D-noise-4000}) values of the gradient. While with exact gradients Euclidean gradient descent converges faster, its performance deteriorates significantly if noise is added to the gradients, while the point-cloud gradient descent is affected much less.}
    \label{fig:3D}
\end{figure}

The starting point $x^0$ for the iterations (in the top left corner in the figure) is taken slightly off the manifold and corresponds to the values
\begin{equation*}
    \theta_0=\pi/8, \, \phi_0=\pi/2, \, r_0 = 1 + 0.2\sin(4\theta_0) = 1.2.
\end{equation*}

Our motivation for using the point-cloud gradient descent, Algorithm~\ref{alg:GD:Manifold}, comes from problems where the (Euclidean) gradient of the objective function is inaccurate away from the manifold. To compare Algorithm~\ref{alg:GD:Manifold} with standard Euclidean gradient descent, we perform two experiments, one with the exact gradient $\grad \cE = x - x_*$ and one with a noisy gradient, where the amount of noise is proportional to the distance from the manifold $d(x) \defeq \dist(x,\cM) \approx \min_{i=1,...,n} \norm{x-x_i}$
\begin{equation}\label{eq:noisy-grad-3D}
    \widetilde{\grad\cE}(x) \defeq \grad \cE(x) + 10 d(x) \norm{\grad \cE(x)} \xi, \quad \xi \sim \cN(0,Id),
\end{equation}
where $\cN(0,Id)$ is the standard normal distribution. We emphasise that this setting is very simple compared to the one we have in mind in the context of operator correction in inverse problems. While using the noisy gradient~\ref{eq:noisy-grad-3D} corresponds to a variant of  stochastic gradient descent, for which convergence is known, in the context of inverse problems there is a systematic error between the exact gradient and the approximate one, and they can even have different null-spaces. What we aim to demonstrate here is that the point-cloud gradient descent from Algorithm~\ref{alg:GD:Manifold} converges faster (in the number of iterations) than Euclidean gradient descent when both are given noisy values of the gradient~\eqref{eq:noisy-grad-3D}.

The results of our experiments are shown in Figure~\ref{fig:3D}. First, in Figure~\ref{fig:3D-clean-4000}, we compare the two algorithms if both are provided with exact values of the gradient. Unsurprisingly, Euclidean gradient descent converges faster because it can cut through the interior of the surface while the point-cloud gradient descent has to take a long route along the surface. But if we give the algorithms the noisy gradient~\eqref{eq:noisy-grad-3D}, Euclidean gradient descent has to spend a lot of time in the interior of the surface, where no reliable gradients are available, while the point-cloud gradient descent is affected much less, see Figure~\ref{fig:3D-noise-4000}.

\begin{figure}[t!]
    \captionsetup[subfigure]{justification=centering}	
    \centering
    \begin{subfigure}[b]{0.48\textwidth}
        \centering
        \includegraphics[width=\textwidth]{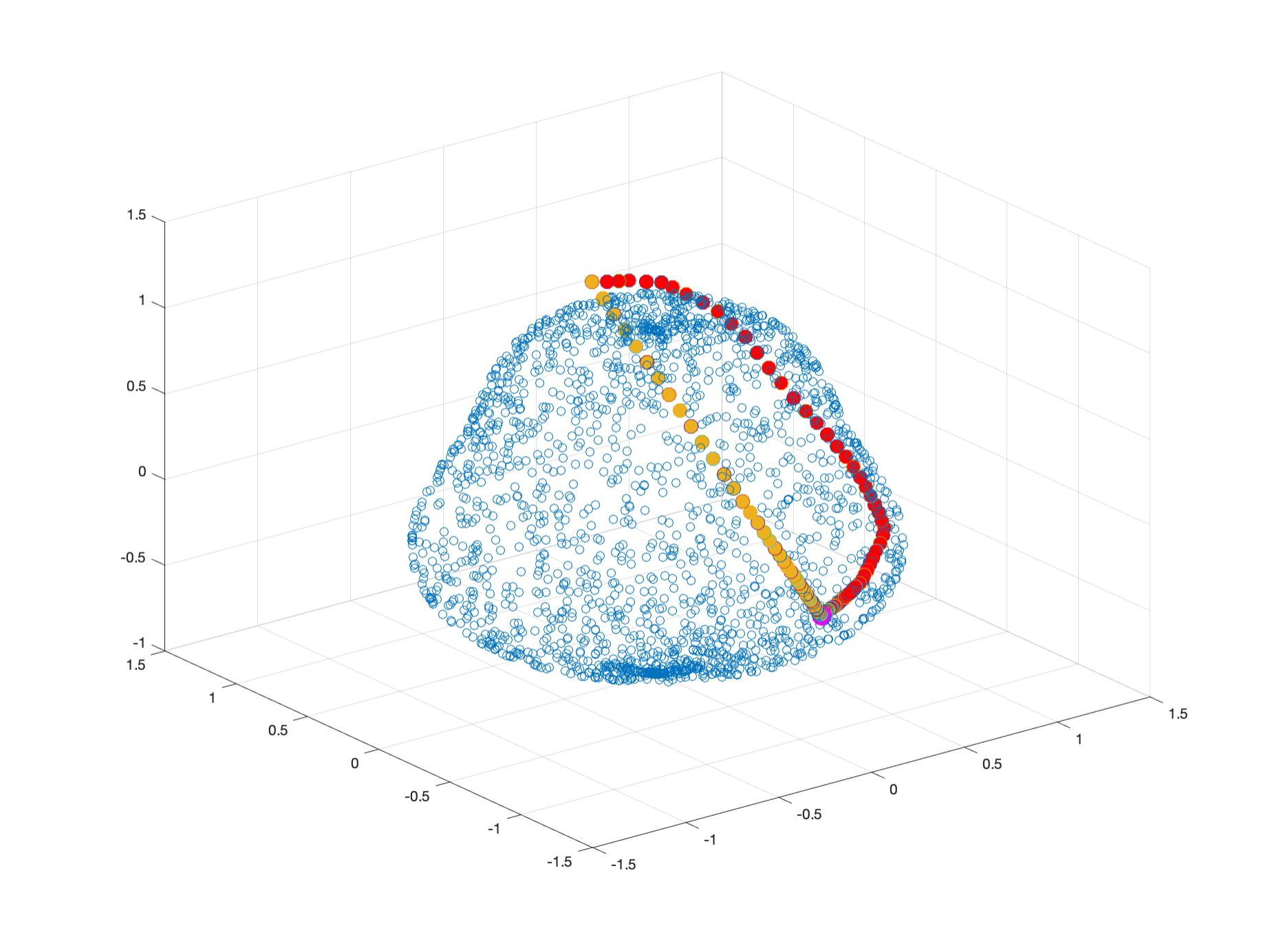}
        \caption{Correct gradient (2000 points) \\ Euclidean GD (yellow): 56 iterations \\ Point-cloud GD (red): 76 iterations}
        \label{fig:3D-clean-2000}
    \end{subfigure}
    \begin{subfigure}[b]{0.48\textwidth}
        \centering
        \includegraphics[width=\textwidth]{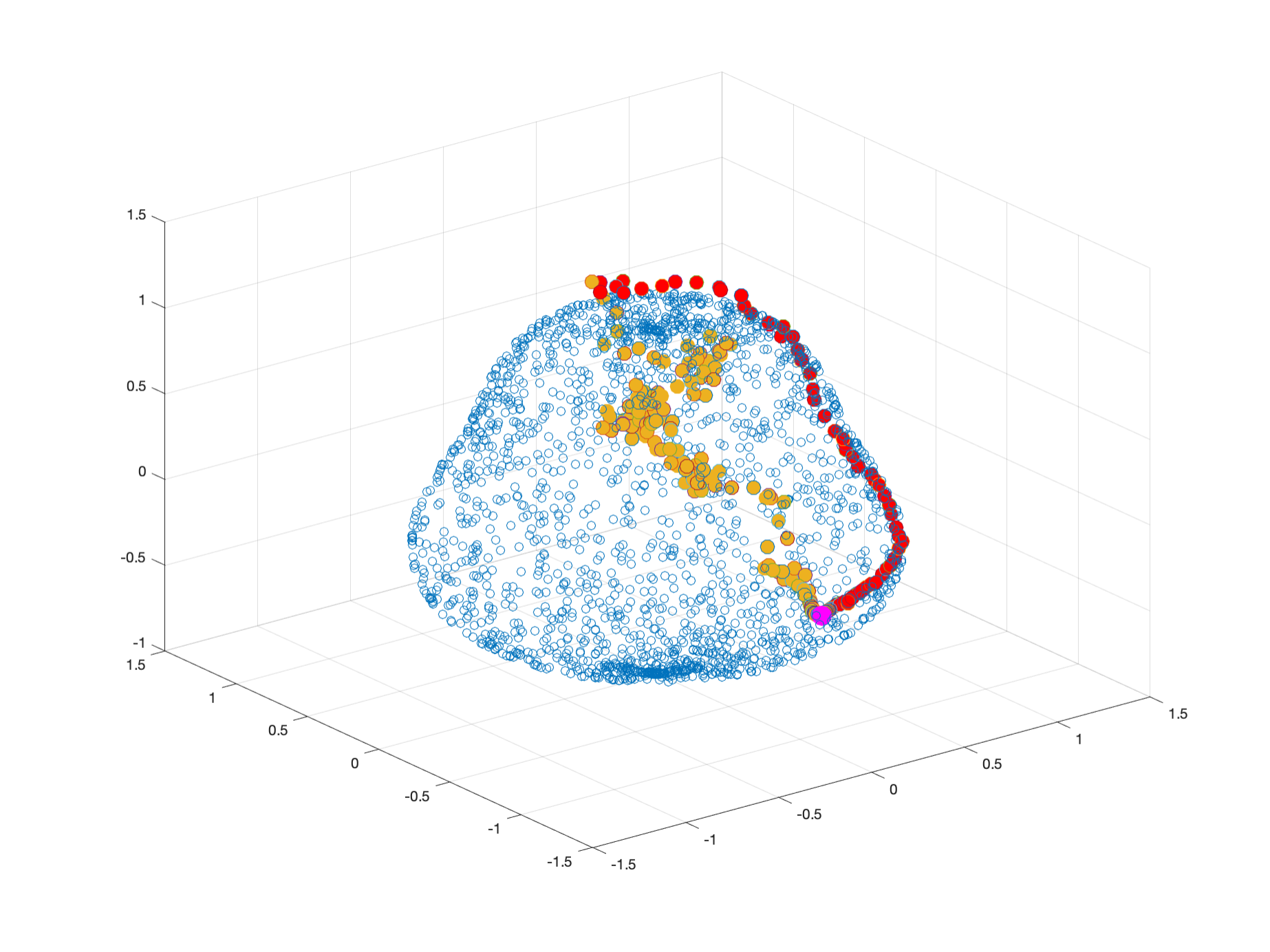}
        \caption{Noisy gradient (2000 points) \\ Euclidean GD (yellow): 151 iterations \\ Point-cloud GD (red): 90 iterations}
        \label{fig:3D-noise-2000}
    \end{subfigure}
    \\
    \begin{subfigure}[b]{0.48\textwidth}
        \centering
        \includegraphics[width=\textwidth]{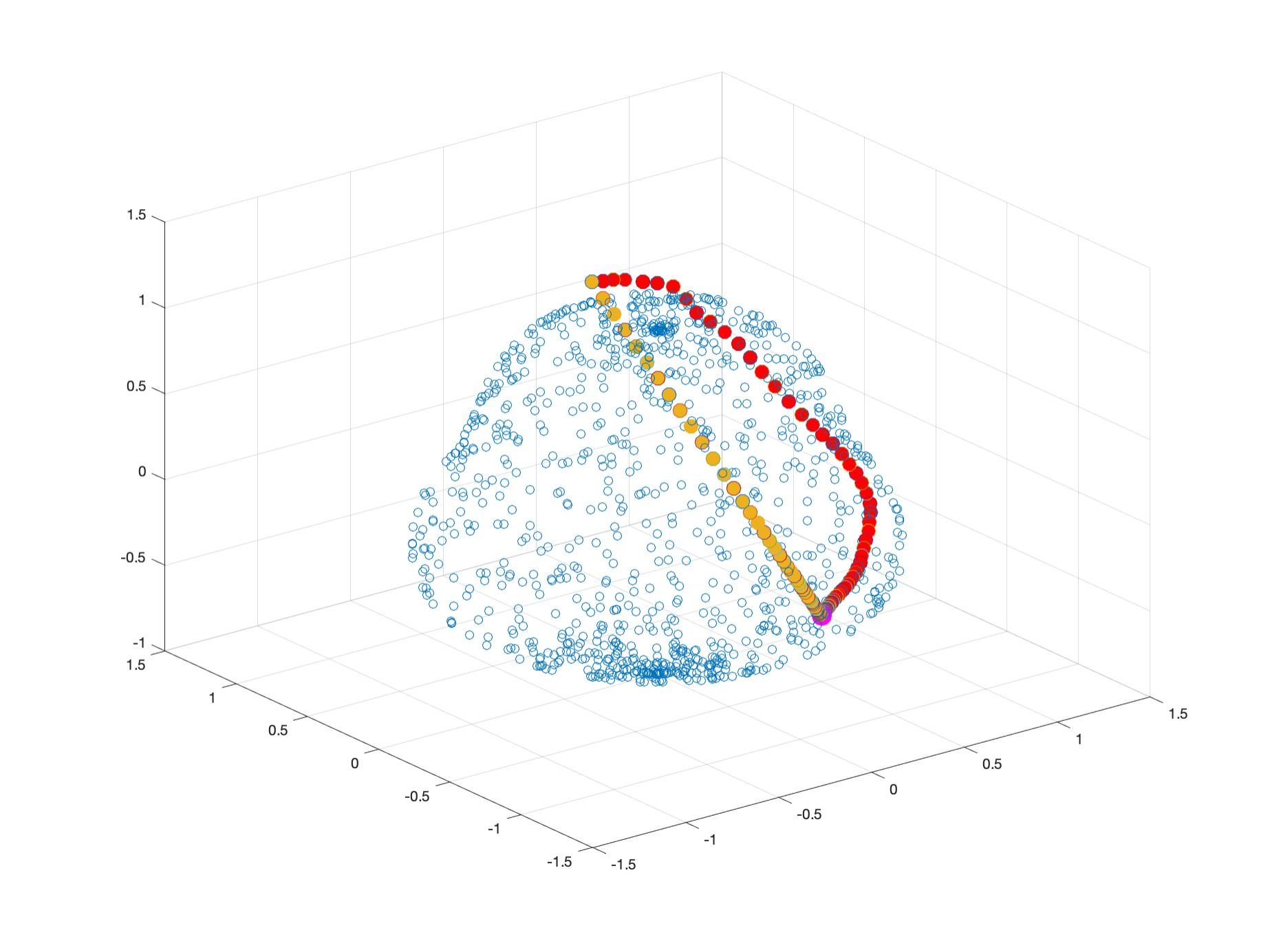}
        \caption{Correct gradient (1000 points) \\ Euclidean GD (yellow): 56 iterations \\ Point-cloud GD (red): 68 iterations}
        \label{fig:3D-clean-1000}
    \end{subfigure}
    \begin{subfigure}[b]{0.48\textwidth}
        \centering
        \includegraphics[width=\textwidth]{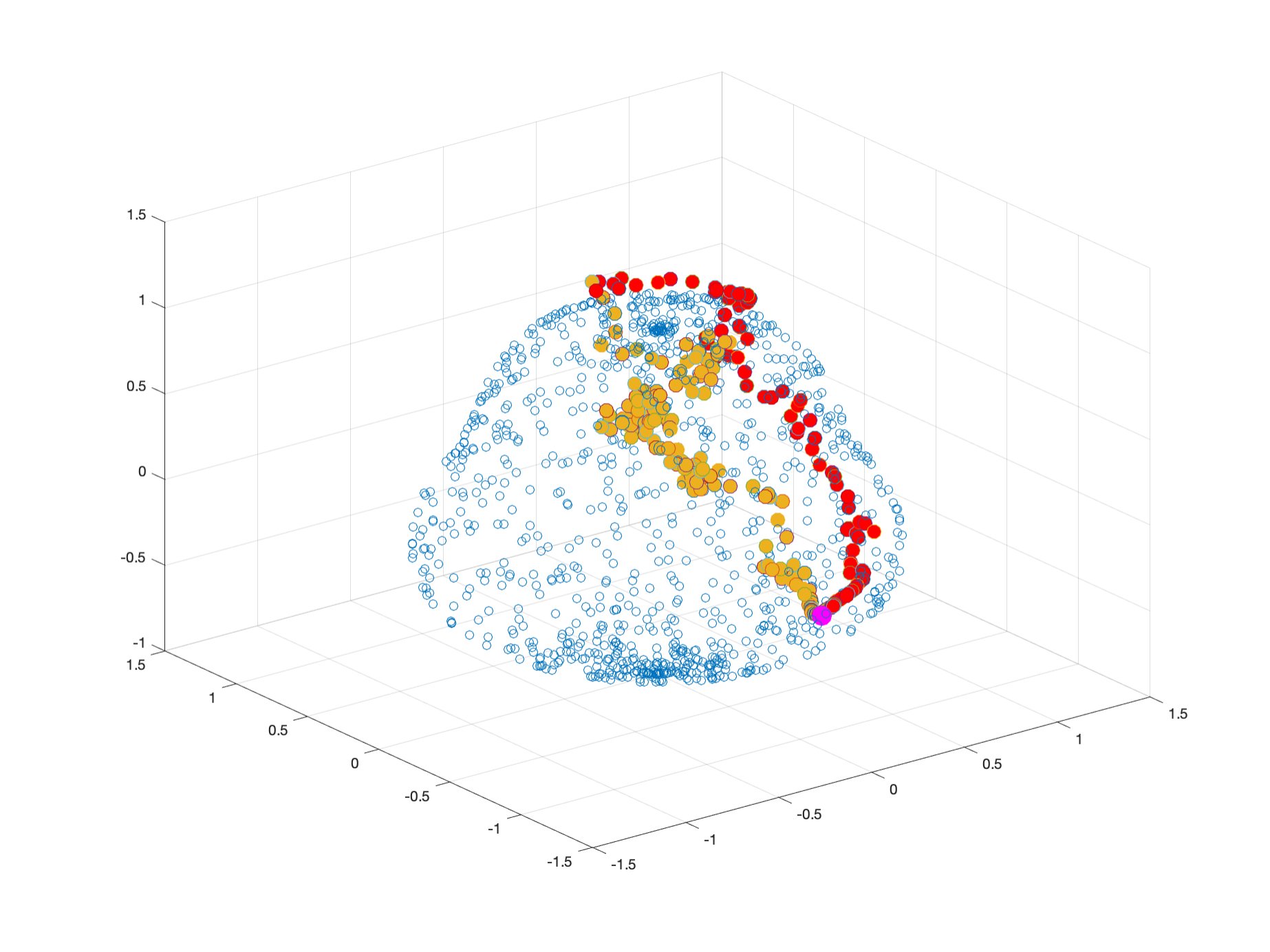}
        \caption{Noisy gradient (1000 points) \\ Euclidean GD (yellow): 151 iterations \\ Point-cloud GD (red): 101 iterations}
        \label{fig:3D-noise-1000}
    \end{subfigure}
    \caption{The effect of the number of samples. The point-cloud gradient descent with exact gradients seems robust with respect to the number of samples. In the case of noisy gradients, the point-cloud gradient descent needs more iterations as the number of samples decreases. This is because the noise in the gradient~\eqref{eq:noisy-grad-3D} is proportional to the distance to the closest point in the point cloud, which increases as the density of samples decreases. The Euclidean gradient descent is, of course, not affected. }
    \label{fig:3D-sample-size}
\end{figure}

Finally, we explore the effect of the number of samples we use to approximate the manifold. The results are shown in Figure~\ref{fig:3D-sample-size}. With exact gradients, the point-cloud gradient descent seems robust to the number of samples and its performance remains good in our experiments with 2000 and 1000 samples (compared to 4000 in Figure~\ref{fig:3D}). If noise is added to the gradient, the point-cloud gradient descent requires more iterations for smaller numbers of samples. The reason is that the noise we add is proportional to the distance to the closest point in the point cloud, which increases as the sample size decreases.

In Figure~\ref{fig:conv-speed} we show the convergence speed (the distance between the current iterate and the true minimiser, measured in the $2$-norm, as a function of the iteration number). The point-cloud gradient descent is consistently faster than the Euclidean one but it slows down as the number of samples from the manifold decreases.

\begin{figure}[t!]
    \captionsetup[subfigure]{justification=centering}	
    \centering
    \begin{subfigure}[b]{0.32\textwidth}
        \centering
        \includegraphics[width=\textwidth]{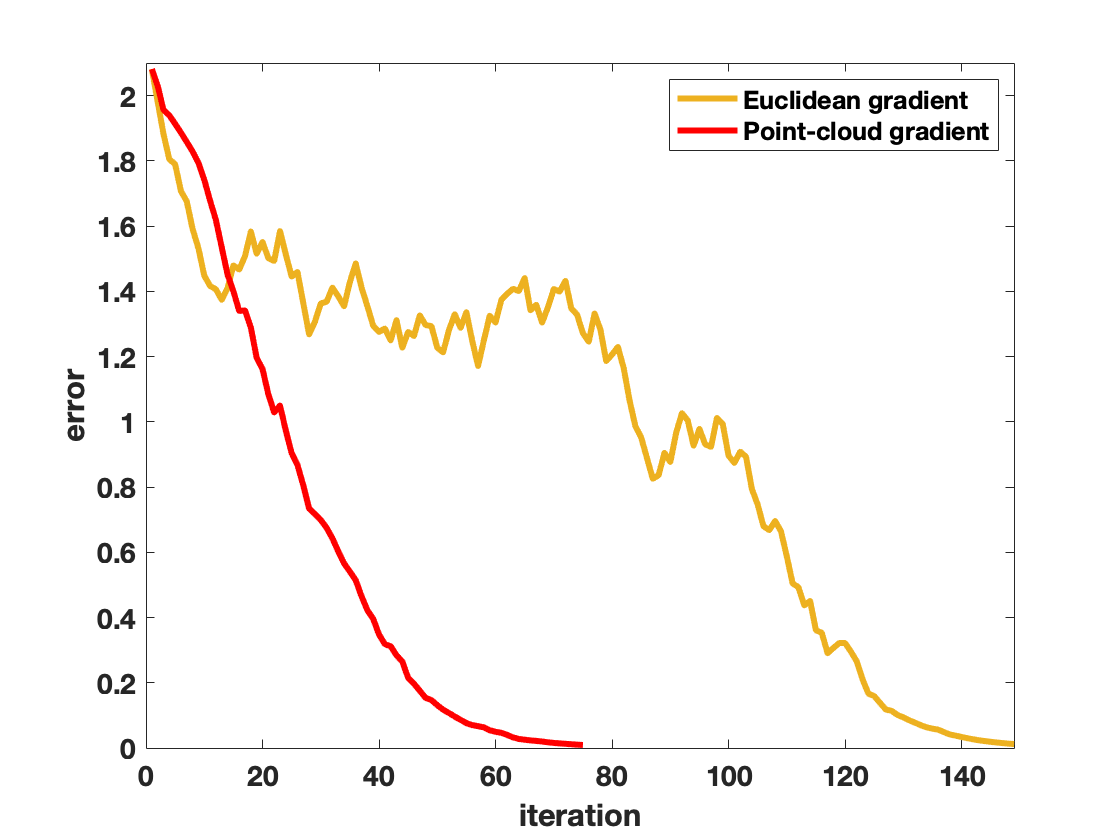}
        \caption{Noisy gradient (4000 points)}
        \label{fig:speed-4000}
    \end{subfigure}
    \begin{subfigure}[b]{0.32\textwidth}
        \centering
        \includegraphics[width=\textwidth]{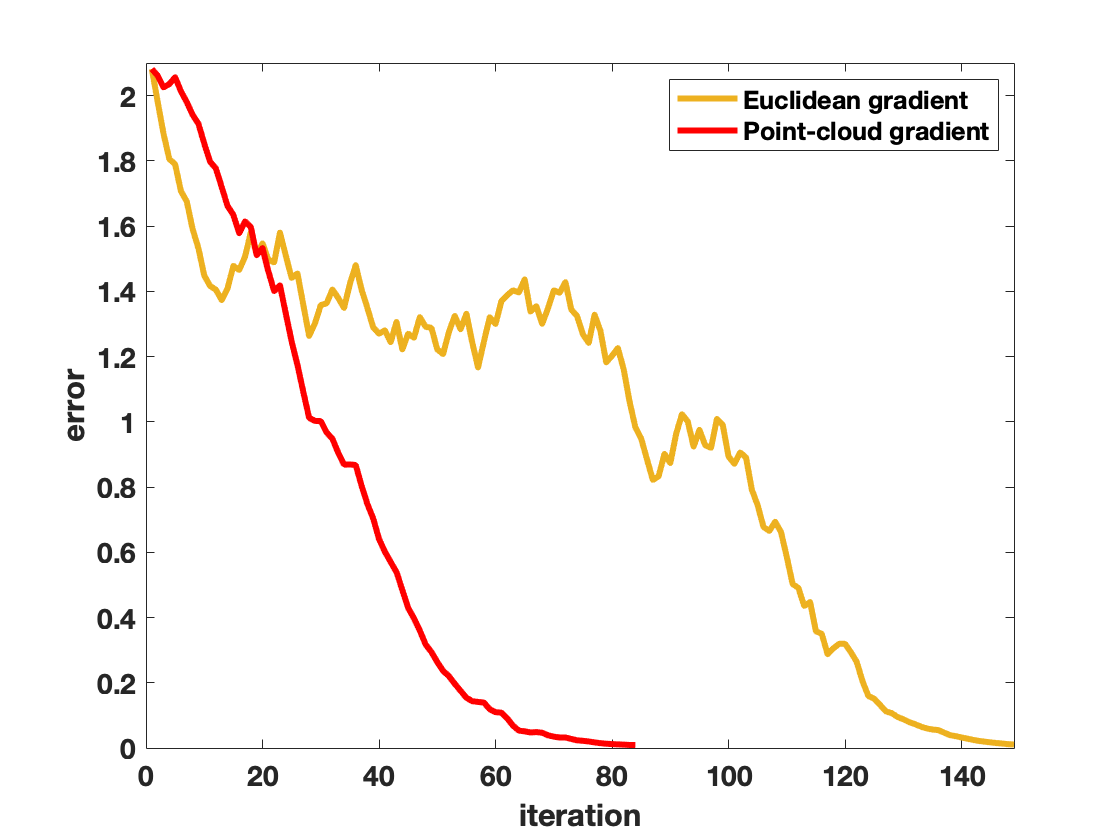}
        \caption{Noisy gradient (2000 points)}
        \label{fig:speed-2000}
    \end{subfigure}
    \begin{subfigure}[b]{0.32\textwidth}
        \centering
        \includegraphics[width=\textwidth]{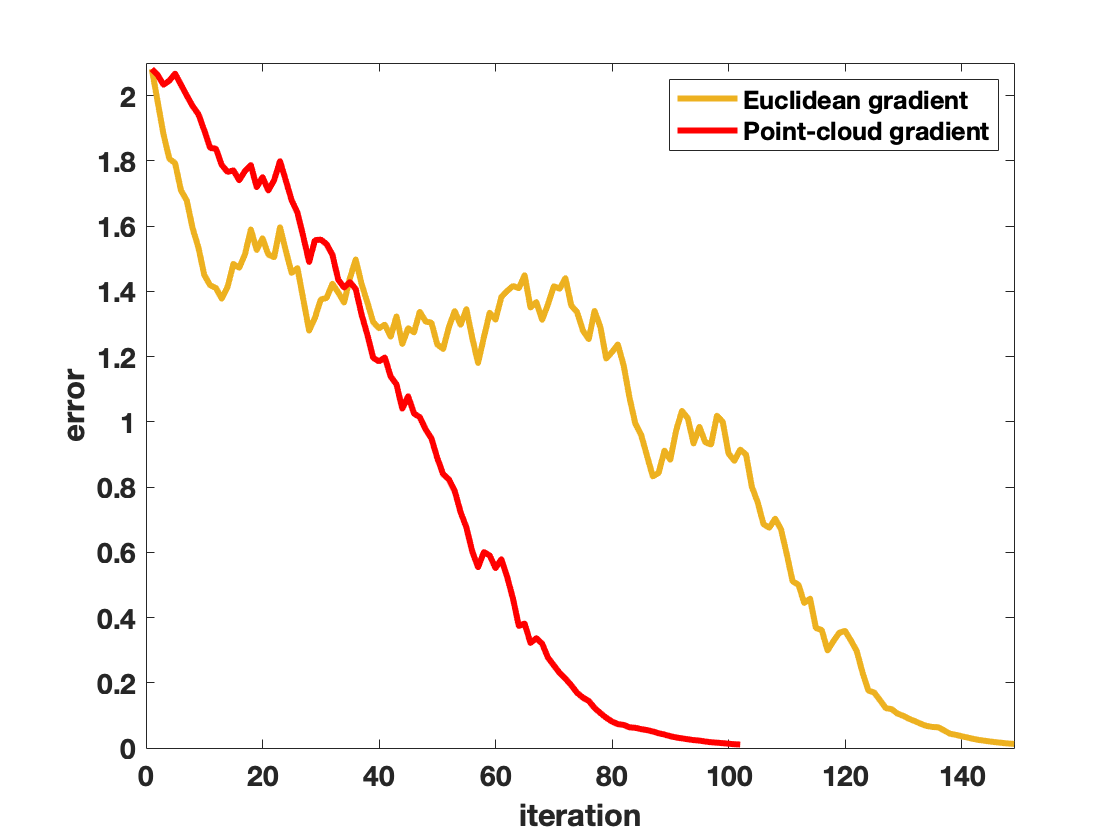}
        \caption{Noisy gradient (1000 points)}
        \label{fig:speed-1000}
    \end{subfigure}
    \caption{Convergence speed of  Euclidean  vs point-cloud gradient descent for different numbers of samples. Point-cloud gradient descent converges consistently faster but its speed decreases for smaller numbers of samples from the manifold.}
    \label{fig:conv-speed}
\end{figure}

\section{Learned Operator Correction}\label{sec:correction}
In this section we apply Algorithm~\ref{alg:GD:Manifold} to solving inverse problems with learned operator corrections. We consider the inverse problem~\eqref{eq:Ax=y} in the setting where both the input space $X$ and output space $Y$ of the forward model $A$ are finite-dimensional.

\subsection{Problem Setup}

Let us return to~\eqref{eq:var-reg-A} and write down a gradient descent iteration for it, assuming that the regulariser $R$ is subdifferentiable
\begin{equation}\label{eq:GD-exact-A}
    x^{k+1} = x^k - \tau_k (A^*Ax^k - A^*y + \alpha \partial R(x^k)),
\end{equation}
where $\partial R(x^k)$ is a subgradient of $R$ at the iterate $x^k$ and $\tau_k > 0$ is the step size. 

As laid out in the introduction, the exact model $A$ may be computationally expensive and its use within an iterative scheme prohibitive in time-sensitive applications. On the other hand, a computationally inexpensive, but less accurate, surrogate $\tilde A$ may be available which renders computations faster but introduces artifacts.  

In~\cite{lunz21}, it has been proposed to use a non-linear learned correction for the forward operator $\tilde A$ of the following form
\begin{equation*}
    A_\Theta(x) \defeq F_\Theta(\tilde A(x)), \quad x \in X,
\end{equation*}
where $F_\Theta$ is a neural network. We recall that there exist other approaches to learn a correction of $\tilde A$ \cite{arridge2023inverse}, including an estimation of the error statistics \cite{arridge:2006}. 
If we now assume differentiability of $F_\Theta$, we can replace the iteration~\eqref{eq:GD-exact-A} with the following one
\begin{subequations}
\label{eq:GD-A-Theta}
\begin{align}
    x^{k+1} &= x^k - \tau_k \left(\left(F'_\Theta(x^k) \tilde A\right)^* \left(F_\Theta(\tilde A x^k)  - y \right) + \alpha \partial R(x^k)\right) \\
    &= x^k - \tau_k \left(\tilde A^* \left[F'_\Theta(\tilde A x^k)\right]^* \left(F_\Theta(\tilde A x^k)  - y \right) + \alpha \partial R(x^k)\right),
\end{align}
\end{subequations}
\noindent \sloppy where $F'_\Theta$ is the Jacobian of $F_\Theta$. A close look at~\eqref{eq:GD-A-Theta} reveals a problem: the range of the update  $\tilde A^* \left[F'_\Theta(\tilde A x^k)\right]^* \left(F_\Theta(\tilde A x^k)  - y \right)$ is not larger than that of $\tilde A^*$. In our application to photoacoustic tomography, the adjoint of the simplified model $\tilde A$ has a strictly smaller range than that of the exact model $A$ (see Section~\ref{sec:PAT}), which results in an unavoidable model error and artifacts in the reconstruction regardless of how well the neural network $F_\Theta$ corrects the modelling error. 
To avoid these problems,~\cite{lunz21} proposed to learn \emph{different} corrections for the forward model and its adjoint, which made training more complicated. We refer to~\cite{lunz21} for details. 

Similar observations about the role of the learning of the adjoint  were made in~\cite{asp-kor-sch-2020} in the context of learned regularisation by projection. Iterative reconstruction methods with mismatched adjoints have also been studied outside the learning context, see for instance \cite{dong2019fixing}.

However, from~\eqref{eq:GD-exact-A} we observe  that these complications can be avoided by learning a correction for the normal operator $\tilde A^* \tilde A$ rather than two separate corrections for $\tilde A$ and its adjoint, assuming that we can compute the accurate adjoint $A^*y$ once.  
In the following, we will use the training images $\{x_i\}_{i=1}^n$ and evaluate the normal operator of the exact model $A$ to obtain $\{z_i \defeq A^*A x_i\}_{i=1}^n$. Using the pairs $\{(x_i,z_i)\}_{i=1}^n$, a neural network $N_\Theta \colon X \to X$ is trained to satisfy $N_\Theta(\tilde A^* \tilde A x_i) \approx A^*Ax_i$. 
The iterations then take the following form
\begin{equation}\label{eq:GD-N-theta}
    x^{k+1} = x^k - \tau_k (N_\Theta(\tilde A^* \tilde A x^k) - A^*y + \alpha \partial R(x^k)).
\end{equation}
In contrast with~\eqref{eq:GD-A-Theta}, the network $N_\Theta$ can correct for the mismatch of the ranges of the adjoints.

A natural question that one may ask is the following one. If we have training pairs  $\{(x_i,z_i)\}_{i=1}^n$ for the exact normal operator, why not train a network $\hat N_\Theta \colon X \to X$ to satisfy $\hat N_\Theta(x_i) \approx A^*Ax_i$ and discard the approximate model $\tilde A$ altogether? The answer is, at least on the conceptual level, that the approximate model $\tilde A$ still carries a great deal of information about the exact model $A$, so that learning the correction can be expected to require much less training data than learning the full operator. Nevertheless, even if the full normal operator is learned, for instance by Fourier Neural Operators \cite{li2021fourier}, our algorithm still applies.

Another problem remains in both cases, however. While one can expect the approximation to be reasonable near the manifold of training images, it may be unreliable outside, which was indeed observed in~\cite{lunz21}. If the iterations~\eqref{eq:GD-N-theta} leave the vicinity of the manifold of training data, the updates will become unreliable and convergence may be compromised. 
This motivates us to constrain gradient descent algorithms to the vicinity of this manifold. In the next section we will modify the iteration~\eqref{eq:GD-N-theta} using the point-cloud gradient descent algorithm from Section~\ref{sec:GD}.

\subsection{Gradient Descent on the Training Data Manifold}

The Euclidean gradient of the objective function after a correction network for the normal operator is trained is now given by
\begin{equation*}
    \grad \cE = N_\Theta(\tilde A^*\tilde A x^k) - b + \alpha \partial R(x^k),
\end{equation*}
where $b \defeq A^*y$ is referred to as the backprojected measurement (a terminology that comes from tomographic applications). Computing it requires one evaluation of the expensive exact model, which can be done before the iterations start. We also use the backprojection $b$ to compute the starting point $x^0$ for our iterations. Alternatively, one could also perform a projection to the data manifold given $b$, which will ensure that the computations start on the manifold and hence the corrected gradients are better in the beginning.

When solving \eqref{eq:var-reg-A} the regulariser $\mathcal{R}$ is needed to define a well-posed solution operator to the inverse problem, i.e. the minimiser of \eqref{eq:var-reg-A}. This regulariser can be understood in two ways, enforcing certain regularity of the solutions as well as introducing prior knowledge encoded in $R:X\to\R$. In fact, if the optimisation problem \eqref{eq:var-reg-A} is solved with proximal-type methods, such as proximal gradient descent where the iterations \eqref{eq:GD-N-theta} turn into
\[
x^{k+1}=\mathrm{prox}_\mathcal{R}\left(x^k - \tau_k\left( N_\Theta(\tilde A^*\tilde A x^k) - b\right)\right).
\]
The proximal operator $\mathrm{prox}_\mathcal{R}:X\to X$ with respect to $R$ projects solutions to the admissible space defined by the regulariser. 
We can see, that this is essentially the same as the projection to the data manifold suggested in Algorithm \ref{alg:GD:Manifold}. Hence, in the following we will omit the regulariser $\mathcal{R}$ when optimising over the data manifold and the projection acts as regulariser.

\subsection{Photoacoustic Tomography}\label{sec:PAT}

We will proceed to test the performance of our method in photoacoustic tomography~\cite{beard:2011} (PAT). This is a suitable practical test problem for our application, as the forward model is expensive to solve and fast approximate models are available, but introduce artefacts if not corrected, as we discuss shortly. 

The corresponding analytic forward model is given as an initial value problem for the wave equation
~\cite{cox2005fast},
\begin{equation}
\label{eqn:PATfwd}
(\partial_{tt} - c^2 \Delta) p(\zeta,t) = 0, \quad p(\zeta,t = 0) = x(\zeta), \quad \partial_t p(\zeta,t = 0) = 0, \text{ with } \zeta \in\R^2,
\end{equation}
modelling the propagation of acoustic waves in biological tissue. 
The pressure field $p(\zeta,t)$ restricted to the boundary $\partial\Omega$ of the computational domain $\Omega\subset\R^2$ is measured by a linear operator $\mathcal{M}$ on a finite time window:
\begin{equation}
y = \mathcal{M} \, p_{|\partial \Omega \times (0,T)}. \label{eqn:Measurement}
\end{equation}
The linear forward problem in the form of \eqref{eq:Ax=y} is then given by equations \eqref{eqn:PATfwd} and \eqref{eqn:Measurement}, mapping initial pressure $x$ to the measured time series $y$.

An accurate solution can be computed, for instance, using a pseudo-spectral time-stepping method~\cite{treeby2012modeling}. While accurate, time-stepping can be computationally expensive, specifically for fine temporal sampling. Consequently, a pseudo-spectral time-stepping method implemented in the toolbox k-wave will be our accurate forward model $A$ in the following \cite{treeby2010kWave}. 
A computationally cheaper model $\widetilde A$ can be obtained by solving the problem under the assumption of constant speed-of-sound and a linear measurement domain (planar in 3D) by a one-step approach in the Fourier domain \cite{cox2005fast,kostli2001}. The drawback of this method, when used as a forward model, is that it results in aliasing artifacts due to the presence of a singularity in the Fourier domain. This can be overcome by thresholding, but will result in aliasing artefacts in the simulated measurement data.
For a more detailed description we refer to~\cite[Appendix A]{lunz21}\cite{hauptmann2018approximate}. Nevertheless, this offers an ideal candidate for this study and will be considered as the approximate model $\widetilde{A}$. 

The difference between the accurate $A$ and approximate model $\widetilde{A}$ is illustrated in Figure \ref{fig:TrainingData}. We also note, that due to the thresholded information in Fourier space, some information is lost, which results in differing ranges for the two models. Specifically, in the computational example below the accurate and approximate forward models, if represented as matrices, are of size $8192\times 4096$. The accurate forward model has a rank of 4076, whereas the approximate model has a rank of 2876.

\subsection{Data Generation}

In the following we aim to generate a point-cloud to test the performance of the proposed gradient descent on manifolds. We will observe that for convergence of the gradient descent scheme, it will be important to have a sufficient number of samples in the neighbourhood of the unknown $x$. To ensure this we will restrict ourselves here to sampling random discs in the computational domain. More precisely, the computational domain is given by $\Omega=[0,1]\times [0,1]$ in a rectangular discretisation of $64\times 64$ pixel, the measurement domain (sensor) is located on the left edge at $\zeta_1=0$. The background value is set to zero and we sample random discs with parameters uniformly distributed as follows: center $(\zeta_1,\zeta_2)\in [0.25,0.75]\times[0.25,0.75]$, radius $r\in [0.1,0.2]$, and absorption value $p\in[0.5,1]$. 

For the point-cloud we sampled a total of $K_{\text{net}}=2^{14}=16384$ such discs as described above. For testing the reconstruction performance, we sampled an additional dataset of 64 discs, where we also simulated the measurement data $y_i=Ax_i+n_\delta$ with additive random Gaussian noise of 1\%. 
An illustration of the forward operator $A$ the corresponding adjoint $A^*$ as well as the approximate counter part is shown in Fig. \ref{fig:TrainingData} for one sample from the training manifold. The aliasing artefacts in the forward projection are clearly visible in the approximate model resulting in different characteristics of the normal operators, and hence wrong gradient information when minimising the data fidelity.

For training of the correction network $N_\theta$ we have used the $K_{\text{net}}=2^{14}$ sampled discs from the data manifold, followed by the application of the normal operators $A^*Ax_i$ and $\widetilde A^*\widetilde A x_i$ to create the corresponding training pairs, see also Fig. \ref{fig:TrainingData}. The network architecture is chosen as a smaller version of the established U-Net architecture, here with 3 scales (2 maxpool layers) and 64, 128, 256 channels from finest to coarsest scale. The network is trained to minimise the mean squared error
\[
\theta^* = \arg \min_\theta \frac{1}{K_{\text{net}}}\sum_{i=1}^{K_{\text{net}}} \| N_\Theta(\widetilde A^* \widetilde A x_i) - A^*Ax_i \|_2^2.
\]
For the optimisation we used the in-built Adam optimiser in pytorch using a cosine decay for the learning rate with initial value $5\cdot 10^{-4}$ and roughly 3 epochs. 

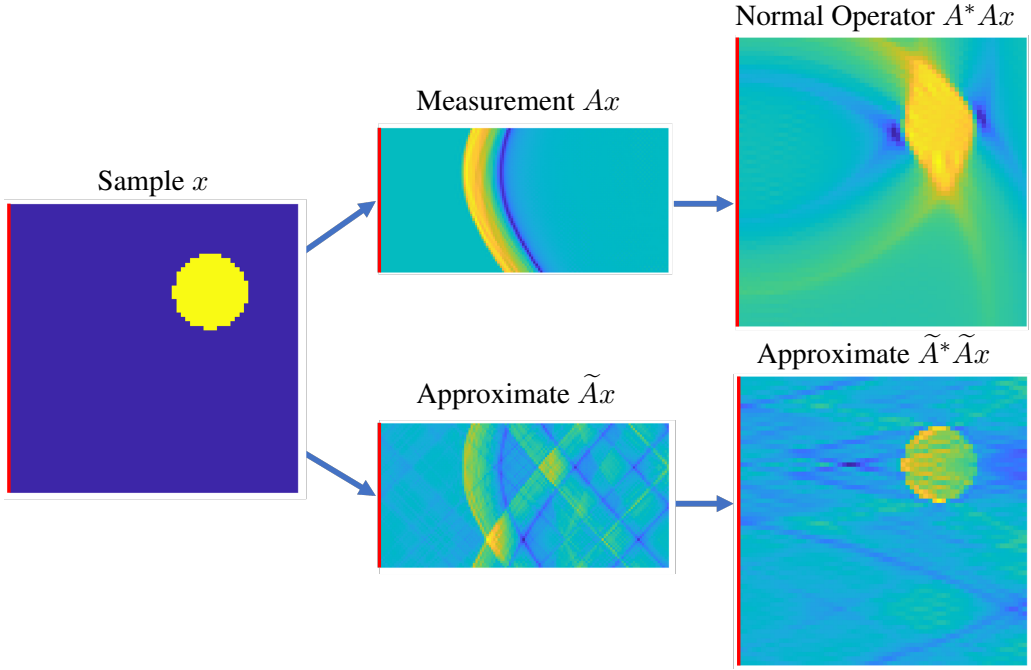
\begin{figure}[t!]
   \begin{picture}(400,250)
    \put(50,0){\includegraphics[width=0.8\textwidth]{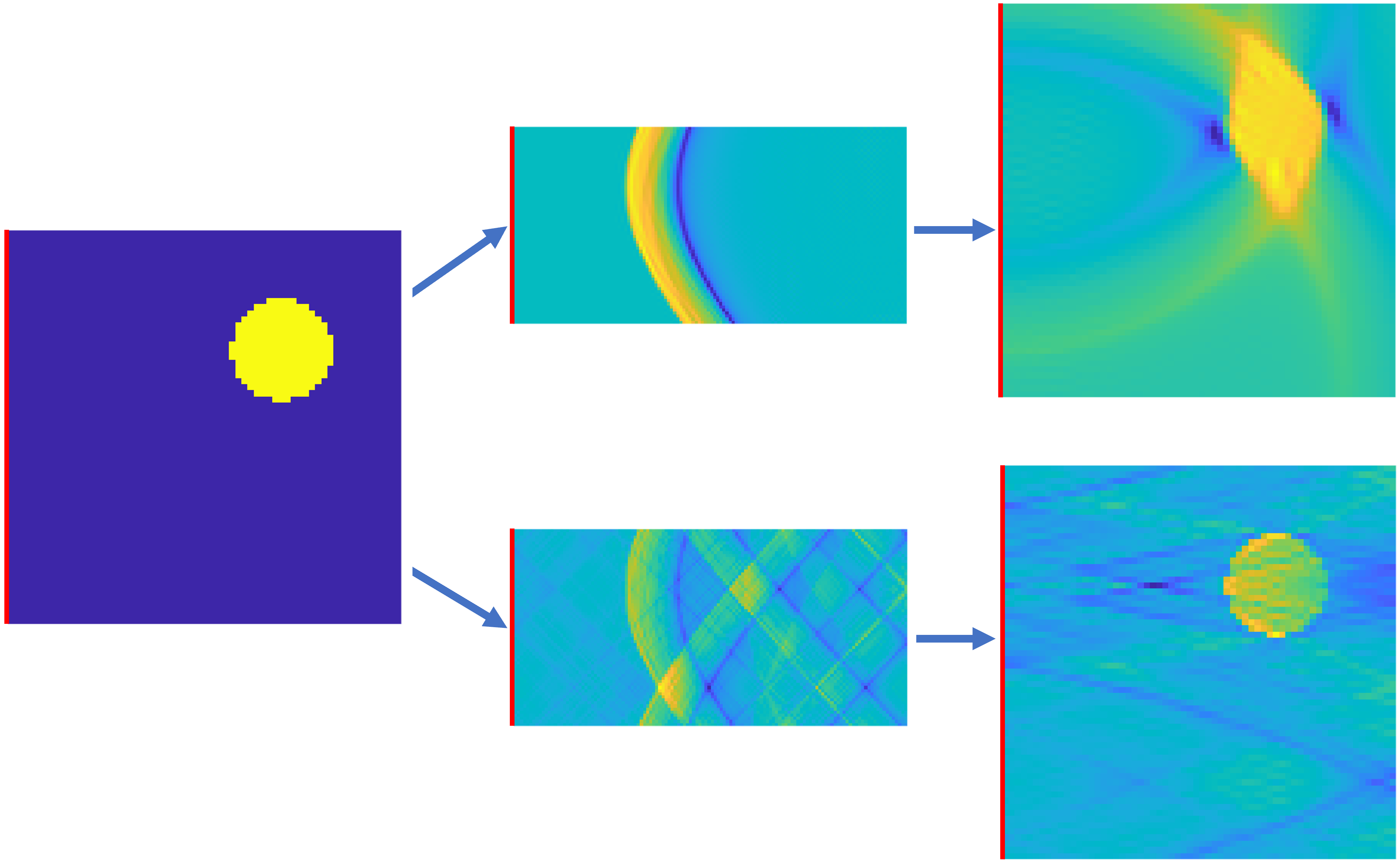}}
    \put(85,180){Sample $x$}
    \put(205,210){Measurement $Ax$}
    \put(205,100){Approximate $\widetilde Ax$}
    \put(325,242){Normal Operator $A^*Ax$}
    \put(333,115){Approximate $\widetilde A^*\widetilde Ax$}
\end{picture}
 \caption{\label{fig:TrainingData} Illustration of forward operator $A$ and corresponding approximate model $\widetilde A$. The sensor is indicated with a red line (left/right image), in the measurements (middle) the red line corresponds to initial time $t=0$. On the right are the normal operator $A^*A$ and the approximate pair $\widetilde A^*\widetilde A$, these pairs are used to train the correction network $N_\theta$.}
\end{figure}

\subsection{Numerical Results}

We want to confirm now numerically, that if a model correction is trained on samples that coincide with the data manifold, this will be sufficient to ensure convergence when optimising over the manifold itself.  
This approach is in contrast to the study in \cite{lunz21}, where a training along the trajectories of the minimisation was necessary to ensure convergence. 
In order to validate this assumption we will need to confirm the following two essential hypotheses:
\begin{itemize}
    \item[i.)] When optimising without the manifold we can not obtain (good) convergence with the corrected model and when not trained along the trajectory.
    \item[ii.)] When optimising over the manifold we can ensure convergence of the corrected model with comparable performance to the accurate model.
\end{itemize}

\subsubsection{Optimisation without manifold}

We will first examine hypothesis i.) by computing a total variation regularised reconstruction, that is we aim to solve the following problem
\begin{equation}\label{eqn:classicTV}
    x^* = \arg \min_{x>0} \|Ax-y\|_2^2 + \alpha \mathrm{TV}(x).
\end{equation}
Here we will use a smoothed version of $\mathrm{TV}$ to make it differentiable. The gradient descent updates given the corrected normal operator $N_\Theta(\tilde A^* \tilde A x^k)$ are then given as in eq. \eqref{eq:GD-N-theta}. We expect that the gradient information is wrong far away from the manifold. Thus, convergence is not expected with the initialisation $x_0=A^*y$. Indeed, we can see in Fig. \ref{fig:TV_recon} that convergence fails and we can not obtain a good reconstruction when using a corrected normal operator that is only trained with samples on the manifold. In fact, the reconstruction shows stronger artefacts compared to simply using the uncorrected approximate normal operator $\tilde A^* \tilde A x^k$, see right image in Fig. \ref{fig:TV_recon} for the corresponding reconstruction. We also note, that the non-negativity constraint in eq. \eqref{eqn:classicTV} is essential to stabilise the reconstructions, without the constraint the corrected and approximate model diverge.

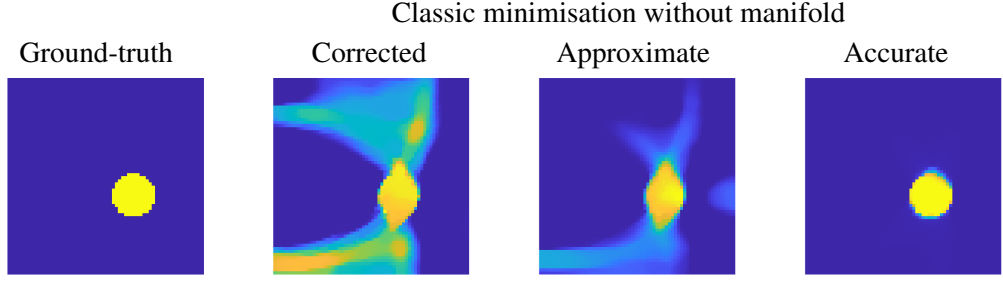
\begin{figure}[t!]
    \centering
    \begin{picture}(400,120)
    
        \put(-10,0){\includegraphics[width=0.25\textwidth]{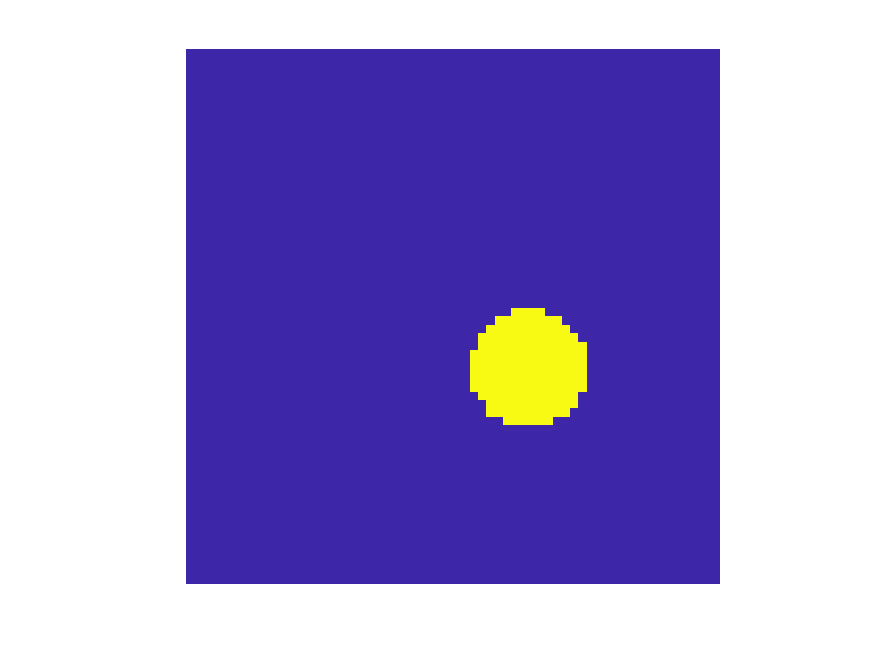}}
        
        \put(90,0){\includegraphics[width=0.25\textwidth]{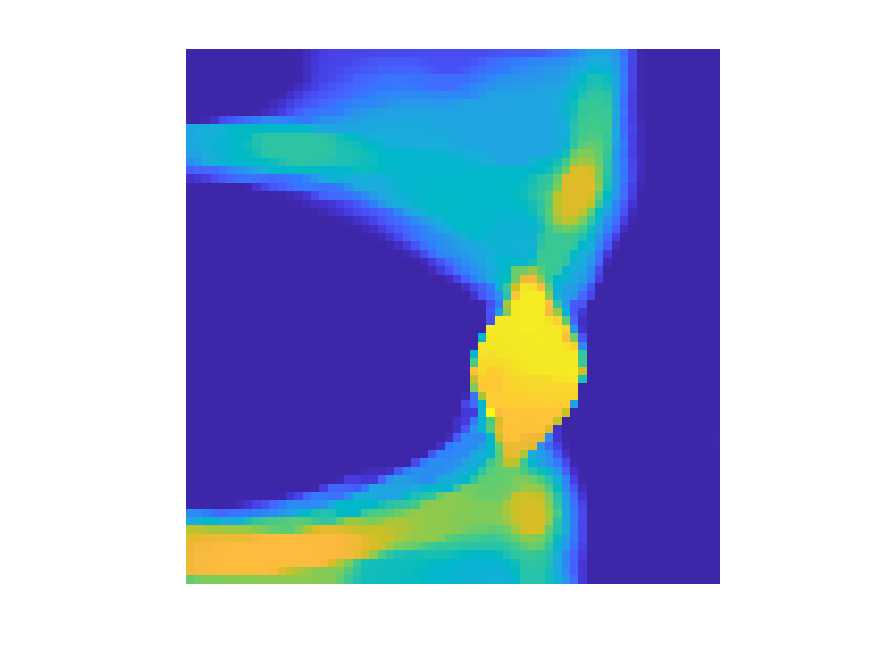}}
        
        \put(190,0){\includegraphics[width=0.25\textwidth]{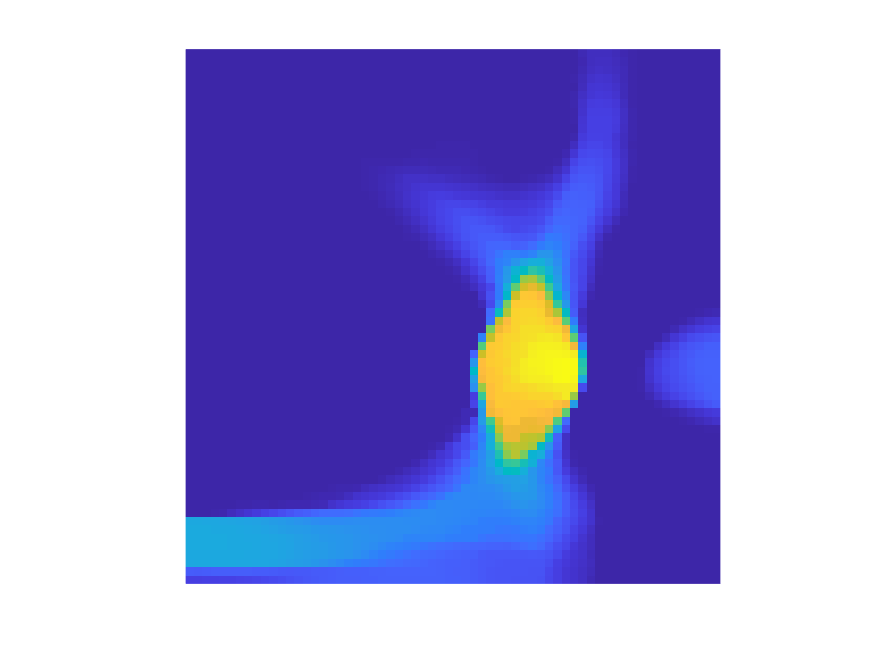}}
        
        \put(290,0){\includegraphics[width=0.25\textwidth]{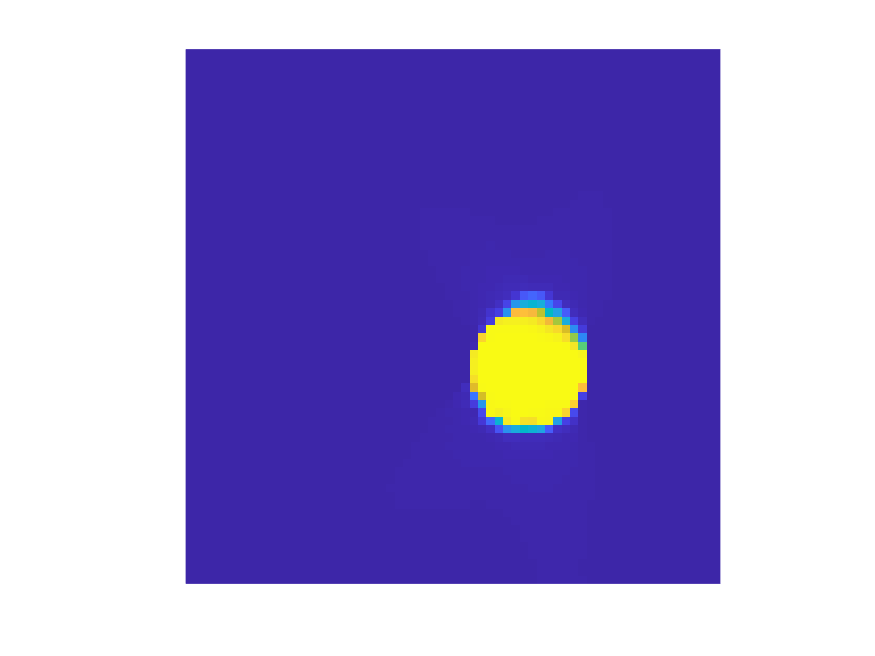}}
        \put(20,90){Ground-truth}
        \put(130,90){Corrected}
        \put(222,90){Approximate}
        \put(330,90){Accurate}

        \put(160,105){Classic minimisation without manifold}
    \end{picture}
        
    \caption{\label{fig:TV_recon} Classic minimisation solving the total variation regularised problem fails when correction is only trained on manifold, but the optimisation is not restricted to the manifold. From left to right: Ground-truth, minimisation result for correction trained only on manifold, minimisation result for approximate model without correction, and result using the accurate model. A non-negativity constraint has been used for all methods to stabilise reconstructions. }
\end{figure}

\begin{figure}[t!]
\centering
   \begin{picture}(400,270)
    \put(-10,160){\includegraphics[width=400 pt]{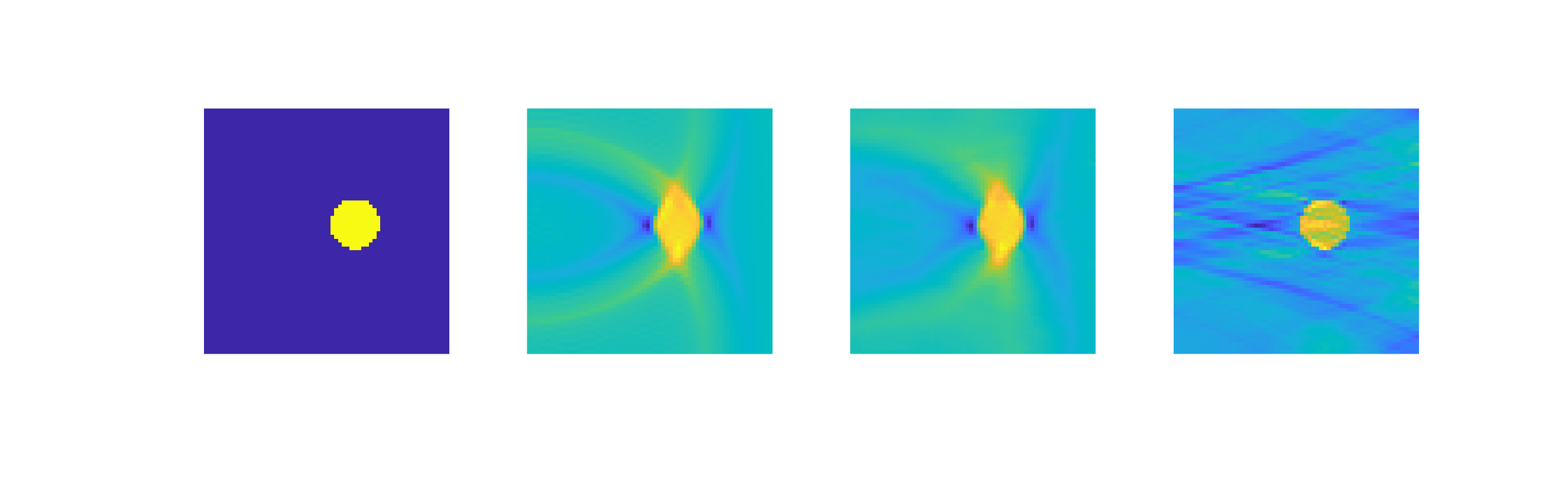}}
    \put(-10,70){\includegraphics[width=400 pt]{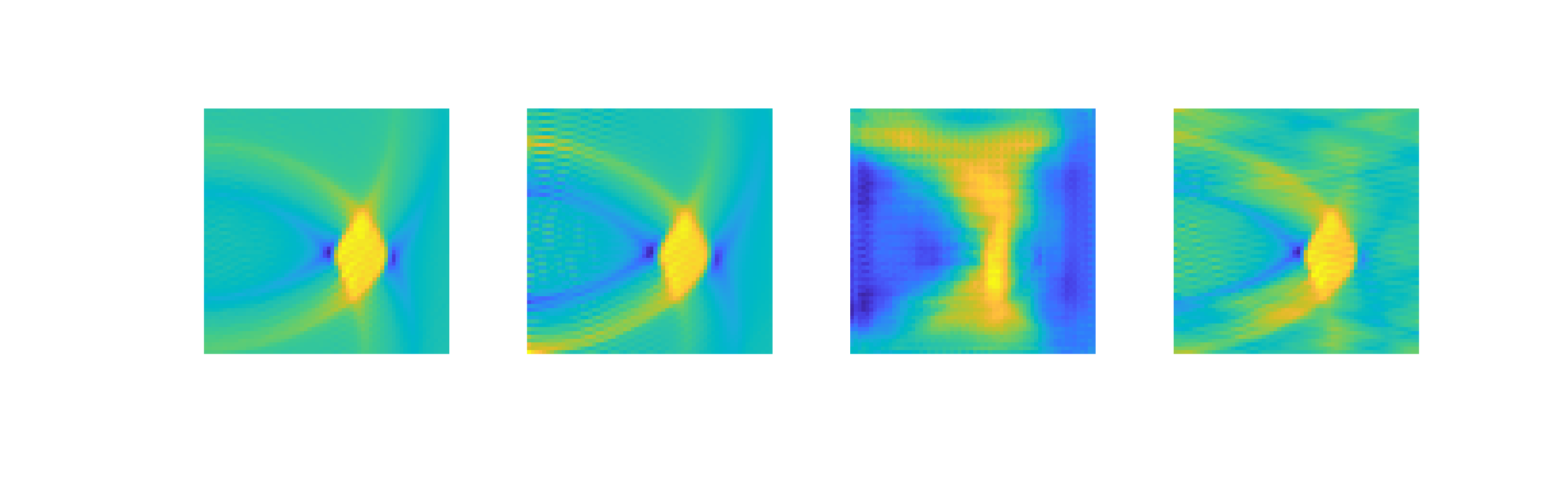}}
    \put(-10,-20){\includegraphics[width=400 pt]{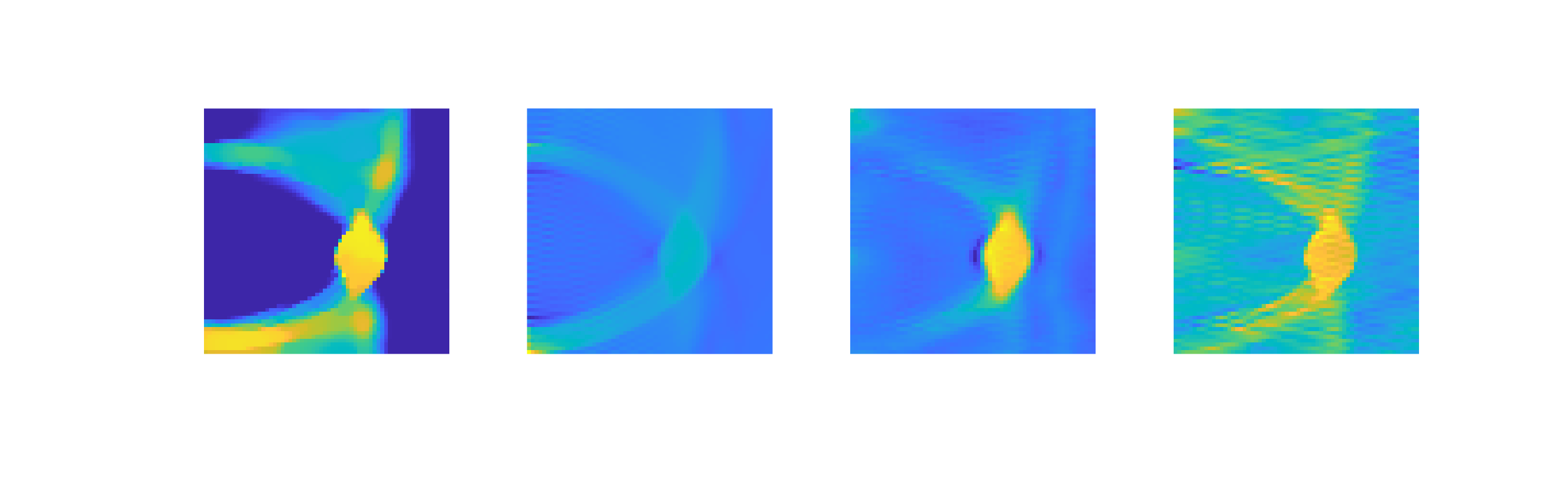}}
    \put(50,260){Sample $x^k$}
    \put(140,260){$A^*A x^k$}
    \put(210,260){ $N_\theta(\widetilde{A}^*\widetilde{A}x^k)$}
    \put(300,260){ $\widetilde{A}^*\widetilde{A}x^k$}

    \put(10,192){\rotatebox{90}{ On manifold}}
    \put(10,60){\rotatebox{90}{Off manifold}}

    \put(25,120){\rotatebox{90}{ $A^*y$}}
    \put(25,5){\rotatebox{90}{ Along trajectory }}

    \end{picture}
    
    \caption{\label{fig:network_performance} Illustration and comparison of correction data on the data manifold and for samples along optimisation trajectory. Top row shows a sample on the manifold and the corresponding normal operators for: accurate, corrected, approximate. Middle row shows the normal operators for the initialisation of the minimisation using $A^*y$. Bottom row shows performance along the trajectory, if samples are not projected onto the manifold. In both cases, the correction fails when samples are too far from the manifold. }
\end{figure}
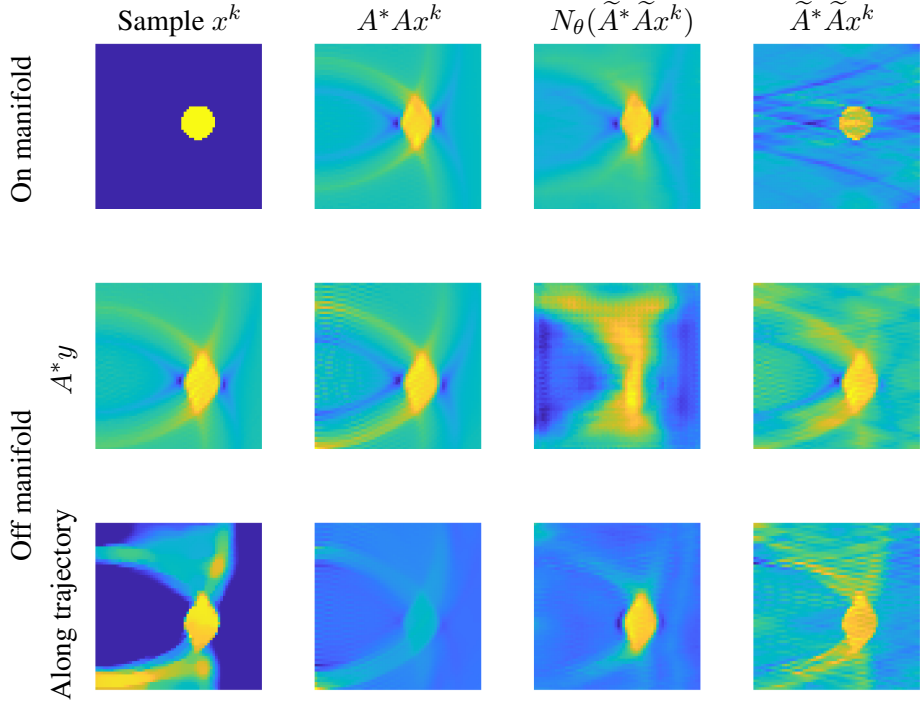

This misfit of the correction is further illustrated in Fig. \ref{fig:network_performance}. One can clearly see in the top row that the correction is doing well for a sample that lies on the data manifold, whereas the correction completely fails for the sample with $x^0=A^*y$ shown in the middle row. The correction is doing better for a sample close to the (wrong) minimiser, but has already diverged too far from the data manifold at this point. This clearly shows that a convergence to the (correct) minimiser can not be achieved.

\subsubsection{Results when optimising over data manifold}

Let us now continue to examine hypothesis ii.) and the behaviour of the gradient descent on the manifold. Here, we will compare 3 scenarios to test Algorithm \ref{alg:GD:Manifold}: Using the correct normal operator $A^* A$ to compute gradients, using only the approximate model $\widetilde A^* \widetilde A$, as well as the proposed learned correction for the normal operator $N_\theta(\widetilde A^* \widetilde A)$. The projection to the manifold should ensure that all three variants produce reasonable results for samples that lie on the data manifold. 
We test performance with the same sample as shown in Figure \ref{fig:TV_recon}, together with the corresponding measurement $y=Ax+n_\delta$ and adjoint $A^*y$. We note, that even though the targets are rather simple the severe limited-view setting (with only a sensor on one side) makes the reconstruction task challenging.

\begin{figure}[t!]
    \captionsetup[subfigure]{justification=centering}	
    \centering
    \begin{subfigure}[b]{0.48\textwidth}
        \centering
        \includegraphics[width=\textwidth]{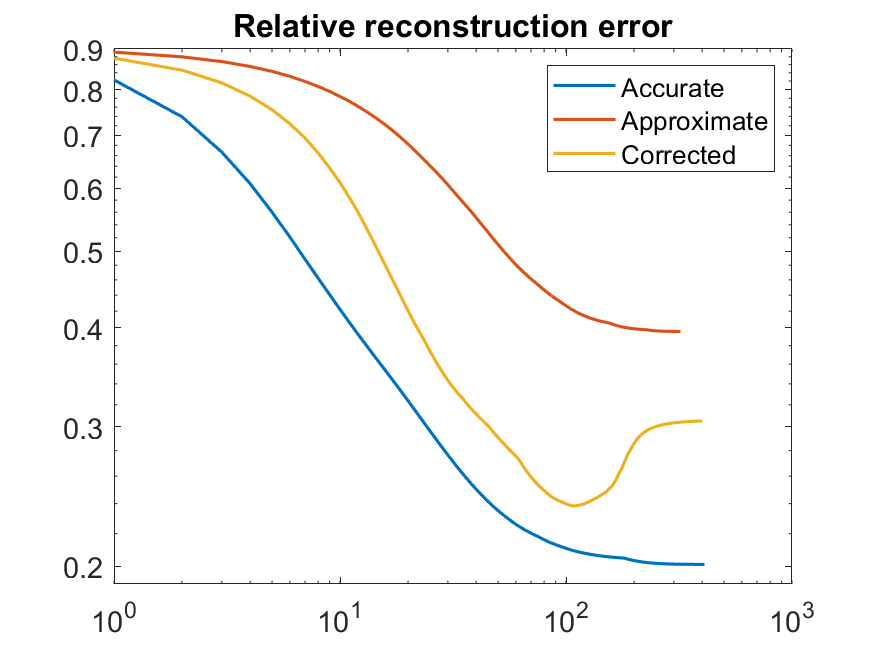}
    \end{subfigure}
    \begin{subfigure}[b]{0.48\textwidth}
        \centering
        \includegraphics[width=\textwidth]{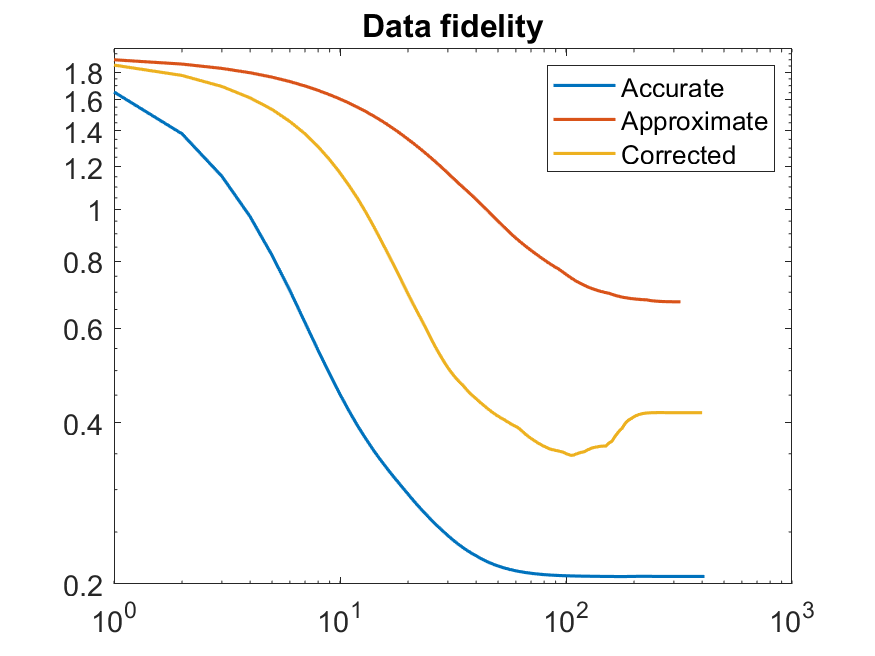}
    \end{subfigure}
    
    \caption{Convergence plots under optimal parameter choice for each method. (Left) Relative $L^2$ reconstruction error, (Right) data fidelity value (with correct model) along the trajectory. 
    }\label{fig:ConvergencePlots}
\end{figure}

\begin{figure}[t!]
\centering
   \begin{picture}(400,120)
    \put(0,0){\includegraphics[width=150pt]{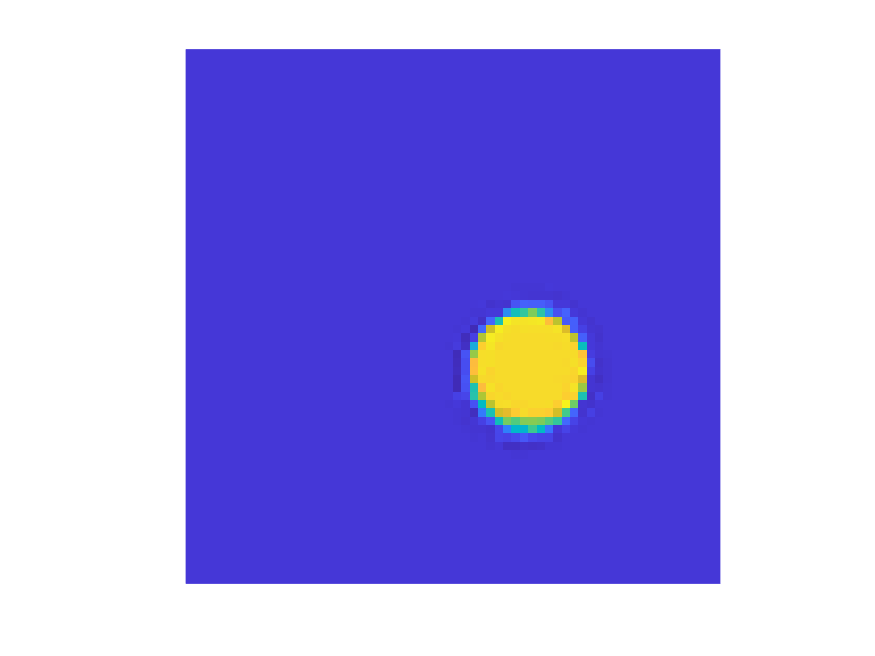}}
    \put(125,0){\includegraphics[width=150pt]{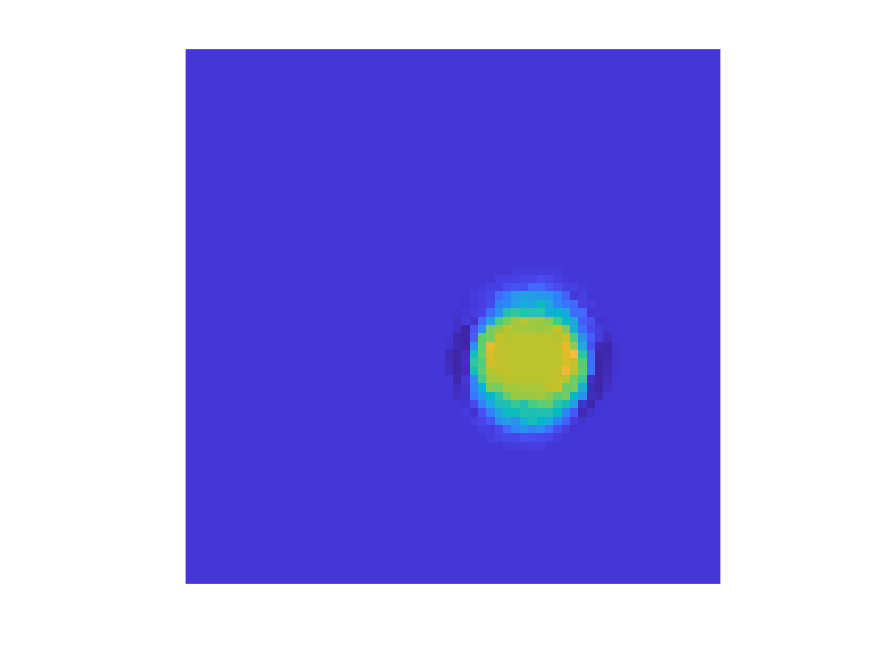}}
    \put(250,0){\includegraphics[width=150pt]{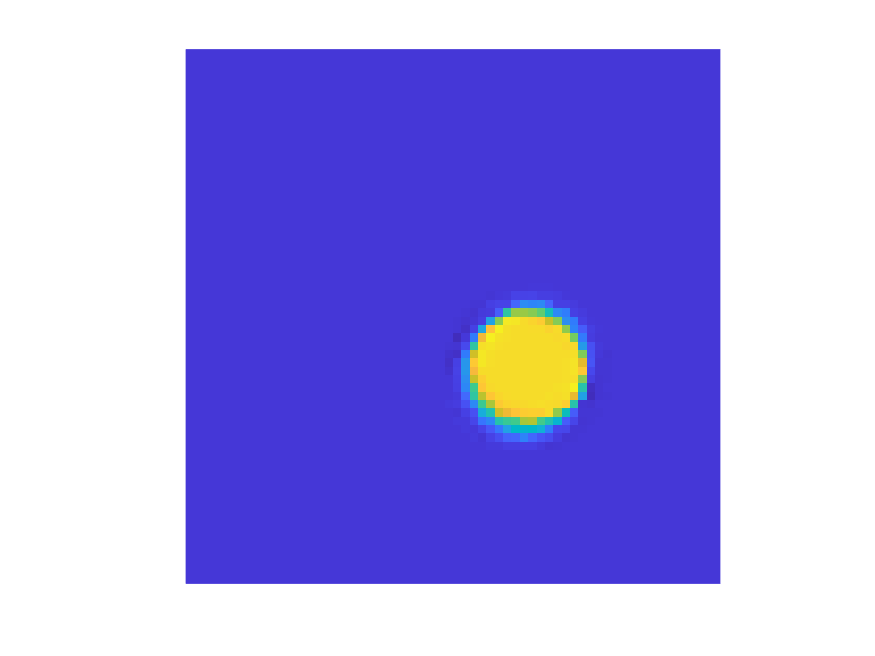}}
    \put(375,8){\includegraphics[height=100pt]{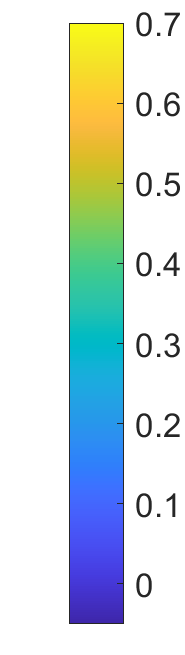}}
    \put(57,110){Accurate}
    \put(175,110){Approximate}
    \put(305,110){Corrected}
    \end{picture}
    \caption{Reconstructions obtained when performing minimisation over the data manifold with the accurate, approximate, and corrected model with early stopping.  Visually, the result obtained with the corrected model is close to the accurate with slightly stronger blurring on the edges. }\label{fig:ReconsComparison}
\end{figure}
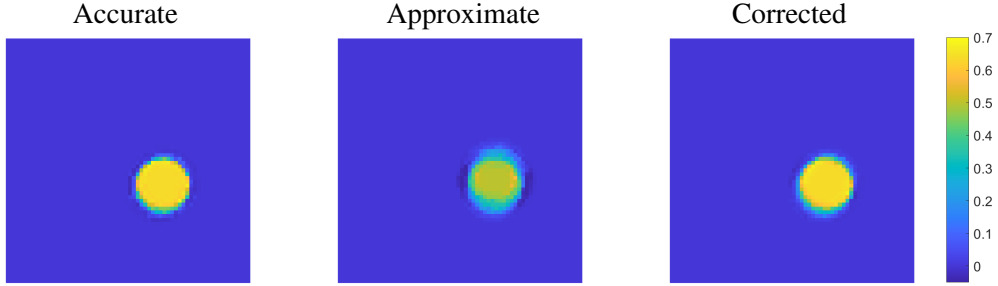

\begin{table}[ht!]
\centering
\caption{Performance comparison and parameter choices for all methods}\label{tab:comparison}
\begin{tabular}{lcccc}
\toprule
\textbf{Method} & {Rel. Error} &   ${\tau_t}$ & ${\tau_n}$ & Iterations \\
\midrule
Accurate &          0.2011         &        1       &  0.02         &  410 \\
Corrected (early stopping) &      0.2384         &         0.2        &    0.05 & 106\\
Corrected (converged) &      0.3052        &        0.2       &    0.05        & 400 \\
Approximate  &   0.3959        &         0.1      &      0.02       & 321 \\
\midrule
TV: Accurate &         0.1713      &       -        &       -  & 2136 \\
TV: Approximate &        0.6874        &       -       &     -    &  2862 \\
TV: Corrected  &         1.6712      &         -      &       -    & 200 \\
\bottomrule
\end{tabular}
\end{table}

We have performed a parameter search for all 3 scenarios to find optimal parameters $\tau_t$ and $\tau_n$, the step sizes in tangential and normal direction, respectively. These parameters can be also understood as projection strength, the larger $\tau_n$ the stronger we enforce projection to the manifold. We ran the optimisation for each scenario until the difference between iterates is small enough, i.e. $\|x^{k+1}-x^k\|<10^{-4}$, or a maximum of 500 iterates is reached.
The resulting convergence plots can be seen in Fig. \ref{fig:ConvergencePlots} and the corresponding reconstructions in Fig. \ref{fig:ReconsComparison}, we also report the quantitative values and chosen parameters in Tab. \ref{tab:comparison}.

The first thing to note is that the optimisation curve for the corrected gradient, while initially nicely following the accurate model, does diverge closely to the minimiser. We attribute this to the fact that the corrected gradient is not perfect, even for samples on the data manifold, and hence leads to a slightly different dynamics around the minimiser. This necessitates a strategy for early stopping or adjustment of the parameters. Here we simply track the norm of the gradient $\nabla \mathcal{E}$ which follows a similar behaviour as seen in Fig. \ref{fig:ConvergencePlots} and allows to stop the optimisation when the norm grows again.
Even with the approximate gradient, the optimisation performs reasonably well, but clearly worse than with the corrected or accurate gradient.

The reconstructed images showcase some slight differences, the reconstructions for the accurate and corrected gradient (with early stopping) are quite similar, where the corrected result does show more smoothing around the edges. This smoothing is even stronger in the approximate result and is the main reasons for the worse quantitative performance. Additionally, the reconstructions show some negative values around the corners, due to the missing non-negativity constraint.
Compared to the TV reconstruction, the results on the manifold for accurate and corrected appear more circular, but are missing the sharpness provided by TV. Finally, TV provides still the best quantitative result for the accurate model.

Examining Tab. \ref{tab:comparison} further, we see that the corrected gradient requires a stronger projection strength onto the manifold $\tau_n$ compared to the accurate model. Surprisingly, the approximate model uses a smaller projection strength, but smaller step size along tangential direction. It is worth pointing out, that reaching a sufficient minimiser with TV takes more iterations than with the gradient descent on the manifold, if we would limit the TV reconstructions to fewer iterations the results would be worse than those obtained with the accurate and approximate model on the manifold. Whereas the iterations for TV with the corrected model have been terminated at 200, as no improvement was achieved any more.
Finally, it should be noted that one iteration on the manifold is computationally more expensive than computing the descent direction for TV, this is primarily caused by the need to find the $k$ nearest neighbours.

\subsubsection{Dependence on sampling density of data manifold}

We examine performance of the optimisation depending on sampling density of the manifold, which will help to provide insights into the suitability of the manifold optimisation as regularisation method. We have subsampled the data manifold with a total of $2^{14}=16384$ to 6 coarser samplings starting at the coarsest set of only $2^{10}=1024$ samples. The optimisation is performed with the same parameter set as determined for the full sampling case as reported in Tab. \ref{tab:comparison}. We have recorded the smallest relative $L^2$ error along the trajectory. The results of this are shown in the plot in Fig. \ref{fig:Density}.

\begin{figure}[t!]
\centering
\includegraphics[width=350pt]{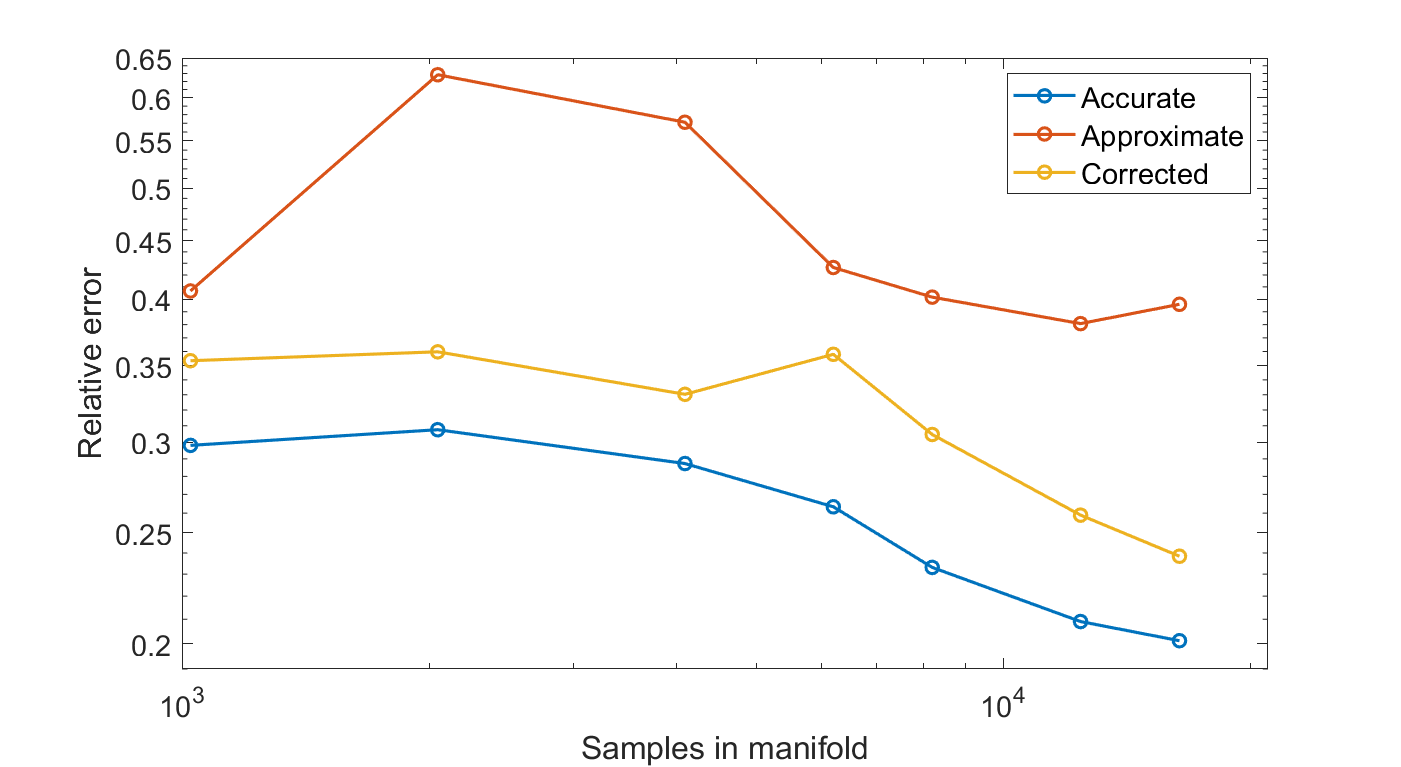}
\caption{Relative errors obtained for varying sampling density of the data manifold for the three cases of: Accurate, Approximate, and corrected.}\label{fig:Density}
\end{figure}

We can see, that the performance hierarchy is maintained as with the finest sampling: The accurate model performs best, the corrected slightly worse, and the approximate the worst, but still better than optimisation without the manifold. In particular, it is clear that the results are getting better with finer sampling density. Interestingly, the results for the approximate operator do not improve much with finer sampling and seem to be fundamentally limited by the incorrect gradients. This suggests, that the manifold approach is particularly suitable if only wrong, or approximate, gradients are available.

For the accurate as well as the corrected model, a fine sampling of the data manifold is necessary, which can be a limiting factor for applications. This suggests that an adaptive sampling procedure may be necessary close to the minimiser. That means, the manifold could be locally refined instead of the global sampling employed here.

\section{Summary and Conclusions}

We proposed a gradient descent framework for minimising energies over manifolds that are not explicitly known but only accessible through point cloud samples. The method combines a tangential descent step, based on locally estimated tangent spaces, with a normal correction that keeps iterates close to the manifold. We established convergence guarantees in the idealised setting of known manifolds and extended the approach to practical scenarios where tangent planes are learned using local PCA.

Numerical experiments, particularly in learned operator correction for inverse problems, demonstrate that constraining optimisation near the data manifold improves stability and convergence when gradients are inaccurate away from the training distribution. Overall, the proposed framework offers a simple and effective bridge between manifold-based optimisation and data-driven modelling, enabling gradient-based methods in settings where the underlying geometry is not explicitly known.

\section*{Acknowledgments}

AH has been supported in part by the Research Council of Finland
(Project Nos. 353093, 353086, Finnish Centre of Excellence in Inverse Modelling and Imaging;
and Project Nos. 359186, 358944, Flagship of Advanced Mathematics for Sensing Imaging and
Modelling; Project No. 338408, Academy Research Fellow (AI-SOL)), the European Research
Council (ERC) under the European Union’s Horizon 2020 Research and Innovation Programme
(Grant Agreement No. 101001417—QUANTOM), and the Finnish Ministry of Education and
Culture’s Pilot for Doctoral Programmes (Pilot project Mathematics of Sensing, Imaging and
Modelling). 

YK is grateful to the EPSRC (fellowship EP/V003615/2 and and Programme Grant
EP/V026259/1) which supported part of this research. 

MT is grateful for the support of the EPSRC Mathematical and Foundations of Artificial Intelligence Probabilistic AI Hub (grant agreement EP/Y007174/1), the NHSBT award 177PATH25 ``Harnessing Computational Genomics to Optimise Blood Transfusion Safety and Efficacy'', the Leverhulme Trust Project Award ``Robust Learning: Uncertainty Quantification, Sensitivity and Stability'' (grant agreement RPG-2024-051) and the EPSRC-JST award “Statistical Safeguarding: A Japan-UK Collaboration Towards the Responsible Data-driven Learning Paradigm”.

The authors would also like to thank the Isaac Newton Institute for Mathematical Sciences, Cambridge, for support and hospitality during the program Mathematics of Deep Learning (supported by EPSRC grant EP/R014604/1) where part of this work was discussed.

\paragraph{Tool and computational resource disclosure.} In support of the Leiden Declaration on Artificial Intelligence and Mathematics\footnote{\url{https://leidendeclaration.ai/}}, we declare our use of generative AI in the preparation of this manuscript. Generative AI was used to assist with proofreading the article and for sketching the proofs of \cref{prop:GD:Cont:ConvGFMin,lem:GD:Unknown:AveragedTangentConsistency,lem:GD:Unknown:NormalConv,lem:GD:Unknown:DiscreteDistanceControl,thm:GD:Unknown:ApproxFlowConv,thm:GD:Unknown:FullyDiscreteConvergence}, which were independently verified and completed by the authors. We accept full responsibility for the correctness of the results.

\printbibliography

\end{document}

%% file: packages.tex
\usepackage[T1]{fontenc}
\usepackage[utf8]{inputenc}
\usepackage{mattsstyle}
\usepackage{amsbsy}

\usepackage{graphicx}

\usepackage{amssymb,amsthm}
\numberwithin{equation}{section}

\usepackage[backend=bibtex,maxnames=5,url=false,style=numeric-comp,isbn=false, giveninits=true,eprint=false,sorting=nyt,doi=false,date=year]{biblatex}

\usepackage{hyperref}

\usepackage{todonotes}
\usepackage{url}

\usepackage{subcaption}
\usepackage{xfrac}
\usepackage{nicefrac}
\usepackage{tikz}
\usepackage{pgfplots}
\usetikzlibrary{3d,calc}


%% file: commands.tex
\newcommand{\R}{\mathbb{R}}

\newcommand{\norm}[1]{\left\Vert #1 \right\Vert}

\newcommand{\RI}{{\R \cup \{+\infty\}}}

\newcommand{\grad}{\nabla}
\newcommand{\defeq}{:=}

\makeatletter
\newsavebox{\@brx}
\newcommand{\llangle}[1][]{\savebox{\@brx}{\(\m@th{#1\langle}\)}%
  \mathopen{\copy\@brx\kern-0.5\wd\@brx\usebox{\@brx}}}
\newcommand{\rrangle}[1][]{\savebox{\@brx}{\(\m@th{#1\rangle}\)}%
  \mathclose{\copy\@brx\kern-0.5\wd\@brx\usebox{\@brx}}}
\makeatother

\renewcommand{\phi}{\varphi}

\newcommand{\ssubset}{\subset\joinrel\subset}